\documentclass[11pt,a4paper]{amsart}

\usepackage{a4wide}
\usepackage{amsfonts,amsthm,amsmath,amssymb,amscd}
\usepackage{dsfont}
\usepackage{graphicx,subfigure,color}
\usepackage{enumerate}
\theoremstyle{plain}
\newtheorem{lem}{Lemma}[section]
\newtheorem{prop}[lem]{Proposition}
\newtheorem{thm}[lem]{Theorem}
\newtheorem{cor}[lem]{Corollary}

\theoremstyle{definition}
\newtheorem{defn}[lem]{Definition}
\newtheorem*{defn*}{Definition}

\newtheorem*{ex*}{Example}

\newtheorem{rem}[lem]{Remark}
\newtheorem*{rem*}{Remark}
\newtheorem*{ack}{Acknowledgement}

\theoremstyle{remark}
\newtheorem*{notation}{Notation}

\DeclareMathOperator{\id}{id}

\DeclareMathOperator{\Leb}{Leb}
\DeclareMathOperator{\Ber}{Ber}

\DeclareMathOperator{\supp}{supp}
\DeclareMathOperator{\sgn}{sgn}
\DeclareMathOperator{\conv}{conv}
\DeclareMathOperator{\PSL}{PSL}

\newcommand{\C}{\mathbb C}

\newcommand{\R}{\mathbb R}
\newcommand{\Z}{\mathbb Z}
\newcommand{\N}{\mathbb N}
\newcommand{\Q}{\mathbb Q}
\newcommand{\PP}{\mathbb P}

\newcommand{\DD}{\mathcal D}
\newcommand{\MM}{\mathcal M}
\newcommand{\TT}{\mathcal T}
\newcommand{\FF}{\mathcal F}
\newcommand{\II}{\mathcal I}
\newcommand{\JJ}{\mathcal J}
\newcommand{\RR}{\mathcal R}

\newcommand{\ii}{\underline{i}}
\newcommand{\jj}{{\bf j}}
\newcommand{\rr}{{\bf r}}
\newcommand{\FFF}{{\bf F}}
\newcommand{\onto}{\xrightarrow[\text{onto}]{}}

\begin{document}

\title[Singular stationary measures for random interval homeomorphisms]{Singular stationary measures for random piecewise affine interval homeomorphisms}

\author{Krzysztof Bara\'nski}
\address{Institute of Mathematics, University of Warsaw,
ul.~Banacha~2, 02-097 Warszawa, Poland}
\email{baranski@mimuw.edu.pl}

\author{Adam \'Spiewak}
\address{Institute of Mathematics, University of Warsaw,
ul.~Banacha~2, 02-097 Warszawa, Poland}
\email{A.Spiewak@mimuw.edu.pl}

\thanks{Supported by the National Science
Centre, Poland, grant no 2018/31/B/ST1/02495.}


\dedicatory{Dedicated to Llu\'{\i}s Alsed\`a on his 65th birthday and Micha\l{} Misiurewicz on his 70th birthday}

\bibliographystyle{amsalpha}

\begin{abstract} We show that the stationary measure for some random systems of two piecewise affine homeomorphisms of the interval is singular, verifying partially a conjecture by Alsed\`a and Misiurewicz and contributing to a question of Navas on the absolute continuity of stationary measures, considered in the setup of semigroups of piecewise affine circle homeomorphisms. We focus on the case of resonant boundary derivatives.
\end{abstract}

\maketitle

\section{Introduction}\label{sec:intro}

For the last forty years there has been an intensive interest in the study of non-autonomous real one-dimensional dynamical systems, especially in the context of the theory of groups of smooth diffeomorphisms acting on the unit circle (see e.g.~\cite{ghys-survey,navas-book} and the references therein). In a probabilistic approach, such a system equipped with an appropriate probability distribution generates in a natural way a Markov process on the circle (see e.g.~\cite{arnold-book,kifer-book1} as general references on random dynamical systems).

Recently, a continuously growing interest in random dynamics has led to an intensive study of random systems given by groups or semigroups of one-dimensional non-smooth maps, for instance interval or circle homeomorphisms (see e.g.~\cite{alseda-misiurewicz,SZ1,homburg16,homburg,malicet,gelfert}). In this paper we consider the properties of stationary measures for a certain class of such systems. 

Let $f_1, \ldots, f_m$, $m \geq 2$, be homeomorphisms of a $1$-dimensional compact  manifold $X$ (the closed interval or the unit circle). Such a system of maps generates a semigroup consisting of iterates $f_{i_n} \circ \cdots \circ f_{i_1}$ for $i_1, \ldots, i_n \in \{1, \dots, m\}$, $n \in \{0, 1, 2,\ldots\}$. Let $(p_1, \ldots, p_m)$ be a probability vector. 
A Borel probability measure on $X$ is called \emph{stationary}, if 
\[
\mu(A) = \sum_{i = 1}^m p_i \mu(f_i^{-1}(A))
\]
for every Borel set $A \subset X$. The Krylov--Bogolyubov Theorem shows that such a measure always exists (but is non-necessarily unique). However, in most cases little is known about its properties. Assuming some regularity of the system (e.g.~forward and backward non-singularity of the transformations) and the uniqueness of the stationary measure, which occur for a wide class of systems (see e.g.~\cite{DKN1}), we know that the stationary measure is either absolutely continuous or singular with respect to the Lebesgue measure. Determining which of the two possibilities occur is a well-known problem, especially in the context of groups of smooth diffeomorphisms acting on the circle (see e.g.~\cite[Question~18]{navas}). Up to now, an answer has been given only in some particular cases. For instance, a conjecture by Y.~Guivarc'h, V.~Kaimanovich and F.~Ledrappier (see \cite[Conjecture 1.21]{deroin_kleptsyn_navas09} states that for a finitely generated subgroup of $\PSL(2, \R)$ acting smoothly on the circle, the stationary measure is singular. The conjecture was proved by Y.~Guivarc'h and Y.~Le Jan in \cite{LeJan} for non-cocompact subgroups and by B.~Deroin, V.~Kleptsyn and A.~Navas in \cite{deroin_kleptsyn_navas09} for some minimal actions of the Thompson group and subgroups of $\PSL(2, \R)$ by $C^2$-diffeomorphisms. On the other hand, the absolute continuity of the stationary measure was proved to hold for a number of random systems of non-homeomorphic maps of the interval (usually expanding at least at average), see e.g.~\cite{pelikan,Buzzi,Araujo-Solano}.

Let us note that the question of determining singularity or absolute continuity of the stationary measure is non-trivial even in the apparently simple case of two contracting similarities $f_1, f_2$ of the unit interval $[0,1]$, given by $f_1(x) = \lambda x$, $f_2(x) = \lambda x + 1 - \lambda$ for $\lambda \in (0, 1)$. Then the unique stationary measure $\nu_{\lambda}$ for the probability vector $(1/2, 1/2)$ is called the \emph{symmetric Bernoulli convolution} and is always either singular or absolutely continuous. It is known (see \cite{shmerkin}) that the set of parameters $\lambda > 1/2$ for which $\nu_{\lambda}$ is singular has Hausdorff dimension zero, and the only known values of ``singular'' parameters are the reciprocals of the Pisot numbers, as proved in \cite{Erdos}. It is a long-standing open question whether these are the only examples of singular Bernoulli convolutions. Despite many results in this direction, a complete answer is still unknown and stimulates an active research. See e.g.~\cite{SixtyBern,Varju} for comprehensive surveys on the subject.

In this paper we consider a random system of two piecewise affine orientation-preserving homeomorphisms of the circle with a unique common fixed point. We look at it as a system of two piecewise affine increasing homeomorphisms $f_-$, $f_+$ of the interval $[0, 1]$. We assume that $f_i(0) = 0$, $f_i(1) = 1$ for $i = -, +$, each $f_i$ has one point of non-differentiability $x_i \in (0,1)$ and $f_-(x) < x < f_+(x)$ for $x \in (0,1)$. See Definition~\ref{defn:AM} and Figure~\ref{fig:graph} for a precise description. Since systems of this type were introduced in \cite{alseda-misiurewicz} by Alsed\`a and Misiurewicz, we call them \emph{Alsed\`a--Misiurewicz systems}, or \emph{AM-systems}. 

We consider $\{f_-, f_+\}$ as a random system with given probabilities $p_-, p_+$, where $p_\pm > 0$, $p_- + p_+ = 1$. Formally, it means that $\{f_-, f_+\}$ defines a \emph{step skew product}
\[
\FF^+ \colon \Sigma_2^+ \times [0,1] \to \Sigma_2^+ \times [0,1], \qquad \FF^+(\ii, x) = (\sigma(\ii), f_{i_0}(x)),
\]
where $\ii = (i_n)_{n \in \N}$ and $\sigma$ is shift on the space $\Sigma_2^+$ of infinite one-sided sequences of two symbols $\{-, +\}$, with the Bernoulli probability distribution given by $(p_-, p_+)$ (see Section~\ref{sec:prelim}). However, in this paper we are mainly interested in the behaviour of the system 
in the \emph{phase space} $[0,1]$, studying trajectories of points $x \in [0,1]$ under $\{f_-, f_+\}$, i.e.~$\{f_{i_n} \circ \cdots \circ f_{i_1}(x)\}_{n=0}^\infty$ for $i_1, i_2, \ldots \in \{-,+\}$.

Note that on the intervals $(0,\min(x_-, x_+))$ and $(\max(x_-, x_+),1)$ the system $\{f_-, f_+\}$ is equivalent, respectively, to two (typically different and non-symmetric) one-dimensional random walks, which are glued in a continuous way. This makes such systems interesting from a probabilistic point of view and we believe they can serve as models for many stochastic phenomena which appear in random one-dimensional dynamics. 

The behaviour of an AM-system depends on the values of the \emph{Lyapunov exponents} at the interval endpoints, i.e.
\[
\Lambda(0) = p_- \ln f_-'(0) + p_+ \ln f_+'(0) +,\qquad \Lambda(1) =
p_- \ln f_-'(1) + p_+ \ln f_+'(1).
\]
For instance, if $\Lambda(0)$, $\Lambda(1)$ are negative, then the endpoints of the interval are attracting in average, so a typical trajectory converges to one of them, which can give rise to two \emph{intermingled basins} for the step skew product $\FF^+$ (see e.g.~\cite{kan,bonifant_milnor,homburg}). In this paper we assume that the Lyapunov exponents $\Lambda(0), \Lambda(1)$ are positive. Then for almost all paths $\underline{i} = (i_n)_n \in \Sigma^+_2$, any two trajectories defined by $\underline{i}$ converge to each other, i.e.~$|f_{i_n} \circ \cdots \circ f_{i_1}(x) - f_{i_n} \circ \cdots \circ f_{i_1}(y)| \to \infty$ as $n \to \infty$ for $x,y \in [0,1]$. This phenomenon is called \emph{synchronization} (see e.g.~\cite{antonov,baxendale,synchronization,kleptsyn_volk14}). 
Moreover, apart from purely atomic stationary measures supported at the common fixed points $0$ and $1$, there exists a (unique) stationary measure $\mu$ on $[0,1]$ such that $\mu(\{0,1\}) = 0$.
In this paper we study the properties of the measure $\mu$, which we call the \emph{stationary measure for the AM-system}.

In \cite{alseda-misiurewicz} L.~Alsed\`a and M.~Misiurewicz showed that for some parameters of an AM-system the stationary measure $\mu$ is equal to the Lebesgue measure and conjectured that  $\mu$ should be singular for typical parameters. In this paper we provide a precise condition under which the stationary measure is equal to the Lebesgue measure (Theorem~\ref{thm:Leb}) and verify the conjecture on singularity for some set of parameters, showing that for a class of AM-systems with \emph{resonant} boundary derivatives (i.e.~ with $\ln f'_+(0)/\ln f'_-(0) = \ln f'_+(1)/\ln f'_-(1) = - k / l\in \Q$) the measure $\mu$ is indeed singular and supported on an \emph{exceptional minimal set}, which is a Cantor set of dimension smaller than $1$. See Theorem~\ref{thm:(k:l)} for details. We also determine the value of the Hausdorff dimension of $\mu$ in the case $l=1$ (see Theorem~\ref{thm:(k:1)}). Furthermore, we present an interesting example of an AM-system with a singular stationary measure of full support $[0,1]$ (Theorem~\ref{thm:res-full}). 
Finally, we show that the considered systems with the same resonance are topologically conjugate (Theorem~\ref{thm:conjugacy}). 

To our knowledge, these are the first examples of singular stationary measures for non-expanding random systems generated by semigroups of homeomorphisms of the circle of that type. The fact that the maps are piecewise affine is especially interesting, since such systems are studied intensively and often serve as models for smooth systems (see e.g.~\cite[Questions~12 and~16]{navas}).
In a forthcoming paper \cite{kac_lemma} we prove that the stationary $\mu$ for an AM-system is singular and has Hausdorff dimension smaller than $1$ for an open set of parameters, including also non-resonant cases.

Notice that in the resonant case mentioned above, the stationary measure is supported on an exceptional minimal set (i.e.~invariant Cantor sets where the systems is minimal), while in the non-resonant one, its support is equal to the entire interval $[0,1]$ (see Proposition~\ref{prop:min,supp}). It should be noted that the properties of exceptional minimal sets are a well-known subject of interest, especially in the context of the groups of diffeomorphisms. For instance, a conjecture of Ghys and Sullivan says that exceptional minimal sets for groups of $C^2$-diffeomorphisms have Lebesgue measure zero. The hypothesis has been recently verified by B.~Deroin, V.~Kleptsyn and A.~Navas \cite{DKN3} for real-analytic diffeomorphisms, while the question remains open in the smooth case. Our paper contributes to the study of such sets for piecewise-linear systems.

The plan of the paper is as follows. In Section~\ref{sec:main_results} we describe the AM-systems  and state the results of the paper in a precise way. 
Section~\ref{sec:prelim} contains preliminaries, while Section~\ref{sec:other_results} is devoted to the proofs of the minor results and Theorem~\ref{thm:Leb}. The proofs of the main results (Theorems~\ref{thm:(k:l)} and~\ref{thm:(k:1)}) are split into Section~\ref{sec:l=1} (case $l=1$) and Section~\ref{sec:l>1} (case $l>1$). Sections~\ref{sec:conj} and~\ref{sec:res-full} contain, respectively, the proofs of Theorems~\ref{thm:conjugacy} and~\ref{thm:res-full}.

\begin{ack} We thank Bal\'azs B\'ar\'any, Micha\l{} Misiurewicz, K\'aroly Simon and Anna Zdunik for useful discussions.
\end{ack}

\section{Main results}\label{sec:main_results}

We begin with a precise description of an Alsed\`a--Misiurewicz system.

\begin{defn} \label{defn:AM}
An AM-system is the system $\{f_-, f_+\}$ of increasing homeomorphisms of the interval $[0,1]$ of the form
\[
f_-(x)=
\begin{cases}
a_- x &\text{for }x\in[0,x_-]\\
1 - b_-(1 - x) &\text{for }x\in (x_-, 1]
\end{cases}, \qquad
f_+(x)=
\begin{cases}
b_+ x &\text{for }x\in[0,x_+]\\
1 - a_+(1 - x) &\text{for }x\in (x_+, 1]
\end{cases},
\]
where $0 < a_- < 1 < b_-$, $0 < a_+ < 1 < b_+$ and
\[
x_- = \frac{b_- - 1}{b_- - a_-}, \qquad x_+ = \frac{1 - a_+}{b_+ - a_+}.
\]
See Figure~\ref{fig:graph}. 
\end{defn}

\begin{figure}[ht!]
\begin{center}
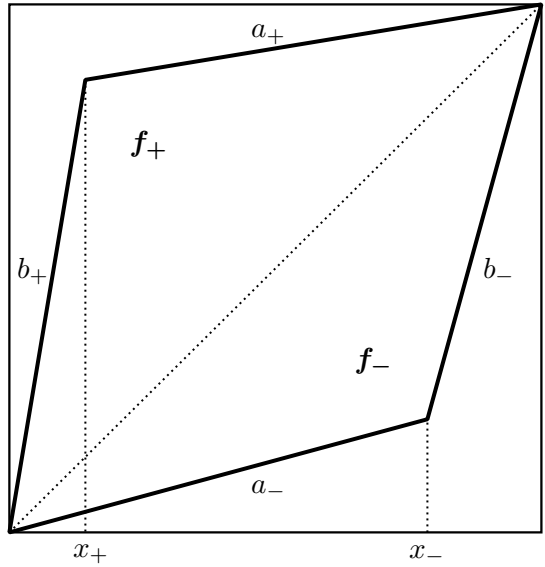
\end{center}
\caption{An example of an AM-system.}\label{fig:graph}
\end{figure}

We consider an AM-system as a random system with probabilities $p_-, p_+$, where $p_-, p_+ > 0$, $p_- + p_+ = 1$.

\begin{defn}\label{defn:Lyap}
The \emph{Lyapunov exponents} of an AM-system $\{f_-, f_+\}$ with probabilities $p_-, p_+$ are defined as
\[
\Lambda(0) = p_- \ln f_-'(0) + p_+ \ln f_+'(0),\qquad
\Lambda(1) = p_- \ln f_-'(1) + p_+ \ln f_+'(1).
\]
\end{defn}

It is known (see \cite{alseda-misiurewicz,homburg16,homburg}) that if the Lyapunov exponents are positive, then there exists a unique \emph{stationary measure} without atoms at the endpoints of $[0,1]$, i.e.~a Borel probability measure $\mu$ on $[0,1]$, such that
\[
\mu = p_- \, (f_-)_* \mu + p_+ \, (f_+)_* \mu,
\]
with $\mu(\{0,1\}) = 0$. For details, see Theorem~\ref{thm:stationary}. Throughout the paper, by a stationary measure for an AM-system with positive Lyapunov exponents we will mean the measure $\mu$.
It is known that if the Lyapunov exponents are positive, then the measure $\mu$ is non-atomic and is either absolutely continuous or singular with respect to the Lebesgue measure (see Propositions~\ref{prop:either} and~\ref{prop:nonatom}).

\begin{defn} We say that an AM-system $\{f_-, f_+\}$ is of:
\begin{itemize}
\item[--] \emph{disjoint type}, if the intervals $[0, f_-(x_-)]$, $[f_+(x_+),1]$ are disjoint, i.e.~$f_-(x_-) < f_+(x_+)$,
\item[--] \emph{border type}, if the intervals $[0, f_-(x_-)]$, $[f_+(x_+),1]$ touch each other, i.e.~ $f_-(x_-) = f_+(x_+)$,
\item[--] \emph{overlapping type}, if the intervals $[0, f_-(x_-)]$, $[f_+(x_+),1]$ overlap, i.e.~$f_-(x_-) > f_+(x_+)$.
\end{itemize}
See Figure \ref{fig:three_types}.
\end{defn}

\begin{figure}[ht!]
\begin{center}
%% Accompanies image file 'three_types.pdf' (pdf, eps, ps)
\begingroup%
  \makeatletter%
  \providecommand\color[2][]{%
    \errmessage{(Inkscape) Color is used for the text in Inkscape, but the package 'color.sty' is not loaded}%
    \renewcommand\color[2][]{}%
  }%
  \providecommand\transparent[1]{%
    \errmessage{(Inkscape) Transparency is used (non-zero) for the text in Inkscape, but the package 'transparent.sty' is not loaded}%
    \renewcommand\transparent[1]{}%
  }%
  \providecommand\rotatebox[2]{#2}%
  \newcommand*\fsize{\dimexpr\f@size pt\relax}%
  \newcommand*\lineheight[1]{\fontsize{\fsize}{#1\fsize}\selectfont}%
  \ifx\svgwidth\undefined%
    \setlength{\unitlength}{421.53538524bp}%
    \ifx\svgscale\undefined%
      \relax%
    \else%
      \setlength{\unitlength}{\unitlength * \real{\svgscale}}%
    \fi%
  \else%
    \setlength{\unitlength}{\svgwidth}%
  \fi%
  \global\let\svgwidth\undefined%
  \global\let\svgscale\undefined%
  \makeatother%
  \begin{picture}(1,0.19114315)%
    \lineheight{1}%
    \setlength\tabcolsep{0pt}%
    \put(0,0){\includegraphics[width=\unitlength,page=1]{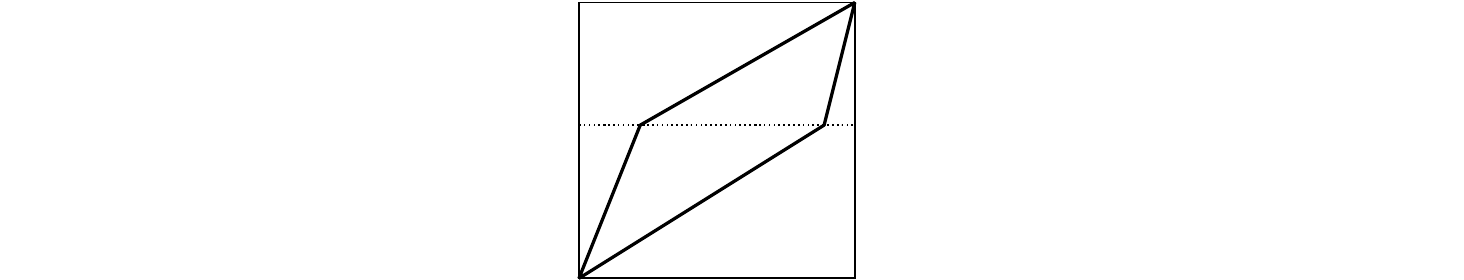}}%
    \put(0.59020678,0.09758452){\color[rgb]{0,0,0}\makebox(0,0)[lt]{\lineheight{1.25}\smash{\begin{tabular}[t]{l}$f_-(x_-)$\end{tabular}}}}%
    \put(0.30943295,0.0996086){\color[rgb]{0,0,0}\makebox(0,0)[lt]{\lineheight{1.25}\smash{\begin{tabular}[t]{l}$f_+(x_+)$\end{tabular}}}}%
    \put(0,0){\includegraphics[width=\unitlength,page=2]{three_types.pdf}}%
    \put(-0.00147187,0.12181706){\color[rgb]{0,0,0}\makebox(0,0)[lt]{\lineheight{1.25}\smash{\begin{tabular}[t]{l}$f_+(x_+)$\end{tabular}}}}%
    \put(-0.00079941,0.05732175){\color[rgb]{0,0,0}\makebox(0,0)[lt]{\lineheight{1.25}\smash{\begin{tabular}[t]{l}$f_-(x_-)$\end{tabular}}}}%
    \put(0,0){\includegraphics[width=\unitlength,page=3]{three_types.pdf}}%
    \put(0.90087388,0.07861455){\color[rgb]{0,0,0}\makebox(0,0)[lt]{\lineheight{1.25}\smash{\begin{tabular}[t]{l}$f_+(x_+)$\end{tabular}}}}%
    \put(0.90027475,0.12035493){\color[rgb]{0,0,0}\makebox(0,0)[lt]{\lineheight{1.25}\smash{\begin{tabular}[t]{l}$f_-(x_-)$\end{tabular}}}}%
  \end{picture}%
\endgroup%

\end{center}
\caption{Three types of AM-systems: disjoint, border and overlapping.}\label{fig:three_types}
\end{figure}

Note that in the case $x_+ < x_-$ (which will be assumed in most of the paper, see Lemma~\ref{lem:x_1}), the system is of 
\begin{itemize}
\item[--] disjoint type, if  $f_-([x_+, x_-])$, $f_+([x_+, x_-])$ are disjoint,
\item[--] border type, if $f_-([x_+, x_-]) \cap f_+([x_+, x_-]) = \{f_-(x_-)\} = \{f_+(x_+)\}$,
\item[--] overlapping type, if $f_-([x_+, x_-])$, $f_+([x_+, x_-])$ overlap.
\end{itemize}

In \cite[Theorem~6.1]{alseda-misiurewicz} Alsed\`a and Misiurewicz showed that if $a_- = a_+ = a$, $b_- = b_+ = b$, $1/a + 1/b = 2$, $p_- = p_- = 1/2$, then the measure $\mu$ is the Lebesgue measure on $[0,1]$. The first result of our paper, presented below, gives an exact condition for an AM-system to have a stationary Lebesgue measure. 

\begin{thm}\label{thm:Leb} Let $\{f_-, f_+\}$ be an AM-system  with probabilities $p_-, p_+$, such that the Lyapunov exponents $\Lambda(0), \Lambda(1)$ are positive. Then the stationary measure $\mu$ is the Lebesgue measure on $[0,1]$ if and only if the system is of border type and 
\[
\frac{p_-}{a_-} + \frac{p_+}{b_+} = 1.
\]
In this case we also have $\frac{p_-}{b_-} + \frac{p_+}{a_+} = 1$.
\end{thm}

In \cite{alseda-misiurewicz} the authors conjectured that the stationary measure $\mu$ for an AM-system with positive Lyapunov exponents is typically singular. The main result of this paper verifies this conjecture for some set of the system parameters. First, we split the AM-systems into two kinds: resonant and non-resonant, which have different kinds of behaviour.

\begin{defn}
We say that that an AM-system $\{f_-, f_+\}$ with probabilities $p_-, p_+$ exhibits a \emph{resonance} at the point $0$, if
\[
\frac{\ln f_+'(0)}{\ln f_-'(0)} \in \Q.
\]
More precisely, a \emph{$(k:l)$-resonance} at $0$ occurs for $k,l \in \N$ if
\[
(f_-'(0))^k(f_+'(0))^l =  a_-^k b_+^l = 1,
\]
which is equivalent to $a_- = f_-'(0) = \rho^l,\ b_+ = f_+'(0) = \rho^{-k}$ for some $\rho \in (0,1)$ and also to $\frac{\ln f_+'(0)}{\ln f_-'(0)}=-\frac{k}{l}$.

Analogously, a $(k:l)$-resonance at $1$ occurs if
\[
(f_-'(1))^l(f_+'(1))^k = a_+^k b_-^l = 1.
\]

Without loss of generality, we always assume that $k, l$ are relatively prime.
\end{defn}

We will show that in the resonant case the (topological) support of the stationary measure $\mu$ for some parameters is a Cantor set in $[0,1]$ of Hausdorff dimension smaller than $1$ (see Theorems~\ref{thm:(k:l)} and~\ref{thm:(k:1)}). A different situation occurs in the non-resonant case, as shown in the following proposition (for the proof see Proposition~\ref{prop:minimal} and Corollary~\ref{cor:min-supp}).

\begin{prop}\label{prop:min,supp}
If an AM-system with positive Lyapunov exponents has no resonance at one of the endpoints $0, 1$, then it is minimal in $(0,1)$ and the support of $\mu$ is equal to $[0,1]$.
\end{prop}

Before stating the main results of this paper, we need to present some definitions. Let
\[
\II\colon [0, 1] \to [0, 1], \qquad \II(x) = \II^{-1} (x) = 1 - x
\]
be the symmetry of $[0,1]$ with respect to its center.

\begin{defn} An AM-system $\{f_-, f_+\}$ is called \emph{symmetric}, if $\II \circ f_- = f_+\circ \II$. 
\end{defn}

Obviously, a system $\{f_-, f_+\}$ is symmetric if and only if $a_- = a_+$ and $b_- = b_+$. It is straightforward that for symmetric systems we have $x_+ = \II(x_-)$ and $f_+(x_+) = \II(f_-(x_-))$. Moreover, for symmetric systems the existence of $(k:l)$-resonance at $0$ is equivalent to the existence of $(k:l)$-resonance at $1$. Note also that if a symmetric systems exhibits $(k:l)$-resonance, then 
the condition $k > l$ is equivalent to the positivity of the exponents $\Lambda(0), \Lambda(1)$ for $p_- = p_+  =1/2$ (see the proof of Lemma~\ref{lem:x_1}).

\begin{defn}
For an AM-system of disjoint type, we call the interval $(f_-(x_-), f_+(x_+))$ \emph{the central interval} of the system $\{f_-, f_+\}$.
\end{defn}

\begin{defn}\label{defn:jumps}
Let $x \in (0,1)$ and $i_1, i_2, \ldots \in \{-,+\}$. We say that a trajectory $\{f_{i_n}\circ \cdots \circ f_{i_1}(x)\}_{n = 0}^\infty$ \emph{jumps over the central interval} at the time $s$, for $s \ge 0$, if $f_{i_s}\circ \cdots \circ f_{i_1}(x)$ and $f_{i_{s+1}}\circ \cdots \circ f_{i_1}(x)$ are in different components of the complement of the central interval in $[0,1]$.
\end{defn}

The main results of this paper shows the singularity of the stationary measure $\mu$ for some symmetric AM-systems of disjoint type, which exhibit a resonance. 

\begin{thm}\label{thm:(k:l)} Let $\{f_-, f_+\}$ be a symmetric AM-system of disjoint type with positive Lyapunov exponents. If the system exhibits $(k:l)$-resonance for some relatively prime $k, l \in \N$, $k > l$, and satisfies $\rho < \eta$, where
\[
\rho = (f_-'(0))^{1/l} = (f_+'(0))^{-1/k} = (f_+'(1))^{1/l} = (f_-'(1))^{-1/k}
\]
and $\eta \in (1/2,1)$ is the unique solution of the equation $\eta^{k+l} -2 \eta^{k+1} + 2\eta - 1 = 0$, then the stationary measure $\mu$ is singular with
\[
\dim_H (\supp \mu) = \frac{\log \eta}{\log \rho} < 1,
\]
where $\supp \mu$ denotes the topological support of $\mu$. Moreover, $\supp \mu$ is a nowhere dense perfect set consisting of all limit points of trajectories of any point $x \in (0,1)$ under $\{f_-, f_+\}$, which jump over the central interval  infinitely many times.
\end{thm}
\begin{rem}\label{rem:l=1}
The condition $\rho < \eta$ is equivalent to $\rho x_- < \frac 1 2$ and implies that the system is of disjoint type. In the case $l=1$ it holds for all  systems of disjoint type. 
\end{rem}

In the case $l = 1$ we give a more precise description of the measure $\mu$.

\begin{thm}\label{thm:(k:1)} Let $\{f_-, f_+\}$ be a symmetric AM-system of disjoint type with probabilities $p_-, p_+$, such that the Lyapunov exponents are positive. If the system exhibits $(k:1)$-resonance for some $k \in \{2, 3, \ldots\}$, then 
\[
\dim_H\mu =  \frac{\sum \limits_{r=1}^k r\Big(\frac{p_+}{p_-} \eta_-^r\log \eta_-  + \frac{p_-}{p_+} \eta_+^r\log \eta_+\Big)}{\sum \limits_{r=1}^k r\Big( \frac{p_+}{p_-}  \eta_-^r + \frac{p_-}{p_+}\eta_+^r \Big)\log \rho},
\]
where $\rho$ is defined as above and $\eta_-, \eta_+ \in (0,1)$ are, respectively, the unique solutions of the equations 
\[
p_+\eta_-^{k+1} - \eta_-  + p_- = 0, \qquad p_-\eta_+^{k+1} - \eta_+  + p_+ = 0.
\]
In particular, if $p_- = p_+ = 1/2$, then
\[
\dim_H\mu = \dim_H (\supp \mu) = \frac{\log \eta}{\log \rho} < 1.
\]
\end{thm}

\begin{rem}\label{rem:struct}
Under the assumptions of Theorem~\ref{thm:(k:l)}, if $l=1$ or $l > 1$, $p_-=p_+ = 1/2$, then the stationary measure $\mu$ is a countable sum of (geometrically) similar copies, with disjoint supports, of a self-similar measure of an iterated function system with the Strong Separation Condition. In the case $l=1$ this iterated function system consists of $k$ maps, while in the case $l>1$, $p_-=p_+ = 1/2$ it is infinite. See Propositions~\ref{prop:symbolic_measure} and~\ref{prop:struct}.
\end{rem}

\begin{rem}\label{rem:(k:1)}
For every $k \in \{2, 3, \ldots\}$ and $\rho \in (0, \eta)$ and probability vector $(p_-, p_+)$ with $p_-, p_+ \in (1/(k+1), k/(k+1))$, the assumptions of Theorem~\ref{thm:(k:1)} are fulfilled for some  AM-system with $\rho = f_-'(0) = f_+'(1)$ and probabilities $p_-, p_+$. In particular, the theorem gives examples of AM-systems with $\dim_H \mu = d$ for arbitrary $d \in (0, 1)$.    
\end{rem}

The next result shows that the considered resonant systems are uniquely determined (up to topological conjugacy) by their resonance data.

\begin{thm}\label{thm:conjugacy}
Let $\{f_-, f_+\}$, $\{g_-, g_+ \}$ be symmetric AM-systems of disjoint type. If both system exhibit $(k:l)$-resonance for some relatively prime $k, l \in \N$, $k > l$, and satisfy $\rho < \eta$, then they are topologically conjugated, i.e.~there exists an increasing homeomorphism $h\colon [0,1] \to [0,1]$ such that
\[
g_- \circ h = h \circ f_-, \qquad g_+ \circ h = h \circ f_+.
\]
\end{thm}

The last result shows that there exist symmetric resonant AM-systems with singular stationary measure of full support.

\begin{thm}\label{thm:res-full}
If a symmetric AM-system with probabilities $p_- = p_+ = 1/2$ and positive Lyapunov exponents exhibits $(5:2)$-resonance and satisfies $\rho  = \eta$, with $\rho,\eta$ defined as above, then $\mu$ is singular with 
\[
\dim_H \mu < 1, \qquad \supp \mu = [0,1].
\]
\end{thm}

Note that in this case the condition $\rho = \eta$ is equivalent to
\[
\rho^7 - 2\rho^6 +2\rho - 1 = 0,
\]
which gives $\rho \approx 0.513649$.

\begin{rem} The resonance $(5:2)$ was chosen because the proof is relatively short in this case. Similar arguments work also for some other values of the resonance $(k:l)$ with $l > 1$. 
 
\end{rem}

\section{Preliminaries}\label{sec:prelim}

\begin{notation} We write $\Z^* = \Z \setminus \{ 0 \}$. 
For $j \in \Z^*$ we set
\[
\sgn(j) = 
\begin{cases}
- & \text{for } j < 0\\
+ & \text{for } j > 0
\end{cases}.
\]
For $x \in \R$, $A \subset \R$ we use the notation
\[
xA = \{xy: y \in A\}.
\]
The convex hull of a set $A$ is denoted by $\conv A$. We write $|I|$ for the length of an interval $I$. 
The symbol $\Leb$ denotes the Lebesgue measure.
By $\dim_H$ (resp.~$\dim_B$) we denote the Hausdorff (resp.~box) dimension. 
The Hausdorff dimension of a Borel measure $\nu$ in $\R^n$ is defined as
\[
\dim_H \nu = \inf\{\dim_H A: A \text{ is a Borel set and } \nu(\R^n \setminus A) = 0\}. 
\]
\end{notation}

Throughout this section we assume that $f_1, \ldots, f_m$, $m \geq 2$, are piecewise $C^1$ increasing homeomorphisms of the interval $[0, 1]$, such that $f_i(0) = 0$, $f_i(1) = 1$ and $f_i(x) \neq x$ for $x \in (0,1)$, $i = 1, \ldots, m$.

For a set $A \subset [0,1]$ we define
\[
f(A) = f_1(A) \cup \ldots \cup f_m(A), \qquad  f^{-1}(A) = f_1^{-1}(A) \cup \ldots \cup f_m^{-1}(A)
\]
and, inductively,
\[
f^0(A) = A, \qquad f^n(A) = f(f^{n-1}(A)), \qquad f^{-n}(A) = f^{-1}(f^{-(n-1)}(A))
\]
for $n \in \N$.

\begin{defn} Suppose $f(X) \subset X$ for some $X \subset [0,1]$. We say that the system $\{f_1, \ldots, f_m\}$ is (forward) \emph{minimal} in $X$, if the union of forward trajectories under $\{f_1, \ldots, f_m\}$ of every point in $X$ is dense in $X$, i.e.~for every $x \in X$ and every non-empty open subset $U$ of $X$ there exist $i_1, \ldots, i_n \in \{1, \ldots, m\}$, $n \ge 0$, such that $f_{i_n} \circ \cdots \circ f_{i_1}(x) \in U$. 
\end{defn}

Let $(p_1, \ldots, p_m)$ be a probability vector, i.e.~$p_1,\ldots, p_m \in (0, 1)$ and $p_1 + \cdots + p_m = 1$.  
We consider the symbolic space
\[
\Sigma_m^+ = \{1,\ldots, m\}^\N
\]
equipped with the Bernoulli measure 
\[
\Ber_{p_1, \ldots, p_m}^+ = \bigotimes_\N \PP_{p_1, \ldots, p_m},
\]
where $\PP_{p_1, \ldots, p_m}$ is the probability distribution on $\{1,\ldots, m\}$ given by $\PP_{p_1, \ldots, p_m}(\{i\}) = p_i$, $i = 1, \ldots, m$.

We study $\{f_1, \ldots, f_m\}$ as the random systems of maps, given by the \emph{step skew product}
\[
\FF^+\colon \Sigma_m^+ \times [0,1] \to \Sigma_m^+ \times [0,1], \qquad \FF^+(\ii, x) = (\sigma(\ii), f_{i_0}(x)),
\]
where $\ii = (i_n)_{n \in \N}$ and $\sigma\colon \Sigma_m^+ \to \Sigma_m^+$ is the left-side shift, i.e.~$\sigma((i_n)_{n \in \N}) = (i_{n+1})_{n \in \N}$.

Let $\MM$ be the space of all Borel probability measures on $[0,1]$.
For $\nu \in \MM$ we denote by $\supp \nu$ the \emph{topological support} of $\nu$, i.e.~the intersection of all closed sets in $[0,1]$ of full measure $\nu$.

\begin{defn} A measure $\mu \in \MM$ is called a \emph{stationary measure} of the system $\{f_1, \ldots, f_m\}$ with probabilities $p_1, \ldots, p_m$, if 
\[
\mu = p_1 \, (f_1)_* \mu + \cdots + p_m \, (f_m)_* \mu.
\]
\end{defn}

\begin{defn} 
The \emph{Markov} (or \emph{transfer}) \emph{operator} $\TT\colon \MM \to \MM$ is defined as
\[
\TT \nu = p_1 \, (f_1)_* \nu + \cdots + p_m \, (f_m)_* \nu.
\]
Analogously, the transfer operator $T\colon L^1([0, 1], \Leb) \to L^1([0, 1], \Leb)$ is defined as
\[
T g  = p_1 \, (f_1^{-1})' \, g \circ f_1^{-1}  + \cdots + p_m \, (f_m^{-1})' \, g \circ f_m^{-1}.
\]
\end{defn}

It is clear that the stationary measures of the system $\{f_1, \ldots, f_m\}$ with probabilities $p_1, \ldots, p_m$ coincide with the fixed points of the transfer operator $\TT$, while the stationary densities (densities of stationary measures with respect to the Lebesgue measure) are the fixed points of the transfer operator $T$. 

\begin{prop}\label{prop:min->supp} Suppose that $f(X) \subset X$ for some $X \subset [0,1]$ and the system $\{f_1, \ldots, f_m\}$ is minimal in $X$. If $\mu$ is a stationary measure for the system and $\supp \mu \subset X$, then $\supp \mu = X$. 
\end{prop}
The proof of this proposition is standard and can found e.g.~in 
\cite[Lemme~5.1]{DKN1} or \cite[Lemma 2]{gelfert}.

Note that since the maps $f_i$ fix the endpoints of the interval, the Dirac measures at $0$ and $1$ are stationary for any probabilities $p_i$. If we assume that the endpoints are repelling in average, then there exists a stationary measure with no atoms at $0, 1$. More precisely, we have the following.

\begin{defn} Assuming $f'_i(0), f'_i(1) > 0$, $i = 1, \ldots, m$, the \emph{Lyapunov exponents} of the system $\{f_1, \ldots, f_m\}$ with probabilities $p_1, \ldots, p_m$ are defined as
\[
\Lambda(0) = p_1 \ln f_1'(0) + \cdots + p_m\ln f_m'(0),\qquad
\Lambda(1) = p_1 \ln f_1'(1) + \cdots + p_m\ln f_m'(1).
\]
\end{defn}

\begin{thm}[{\cite[Proposition 4.1]{homburg16}, \cite[Lemmas 3.2--3.4]{homburg}}]\label{thm:stationary} If $\Lambda(0), \Lambda(1) > 0$, then there exists a unique probability stationary measure $\mu$ for the system $\{f_1, \ldots, f_m\}$ with probabilities $p_1, \ldots, p_m$, such that $\mu(\{0, 1\}) = 0$. 
Moreover, there exist positive constants $c, \alpha_0, \delta_0$ such that for every $\alpha \in (0, \alpha_0)$, $\delta \in (0, \delta_0)$ and for
\[
\DD_{c, \alpha, \delta} = \{\nu \in \MM: \nu([0, x]), \nu([1-x,1]) < c x^\alpha \text{ for every } x \in (0,\delta)\},
\]
we have $\TT (\DD_{c, \alpha, \delta}) \subset \DD_{c, \alpha, \delta}$ and $\mu \in \DD_{c, \alpha, \delta}$.
\end{thm}

\begin{rem} Actually, in \cite{homburg16,homburg} the theorem was proved for 
systems of $C^1$-diffeo\-morphisms, but the proof goes through if we only assume that the map are smooth in some neighbourhoods of $0, 1$. 
\end{rem}

\begin{rem}\label{rem:stability}
The uniqueness of the stationary measure $\mu \in  \DD_{c, \alpha, \delta}$ implies
\[
\frac{1}{N} \sum \limits_{n=0}^{N-1} \TT^n \nu \to \mu \quad \text{as } N \to \infty \quad \text{in weak-* topology} \quad \text{for every }  \nu \in \DD_{c, \alpha, \delta}.
\]
\end{rem}

\begin{rem}
The measure $\mathbb \Ber_{p_1, \ldots, p_m}^+ \times \mu$ is an $\FF^+$-invariant measure on $\Sigma_m^+ \times [0,1]$. Moreover, there is a Borel probability measure on $\Sigma_m \times [0,1]$, where $\Sigma_m = \{1, \ldots, m\}^\Z$, invariant with respect to the (extended) step skew product, which is associated to $\mu$ in a unique way (see \cite{arnold-book,homburg}). 
\end{rem}

It is well-known (see e.g.~\cite[Theorem 2.5]{dubins_freedman}) that whenever the operator $\TT$ preserves absolute continuity and singularity of measures (with respect to the Lebesgue measure) and the stationary measure is unique, then it is of pure type (i.e.~is either absolutely continuous or singular with respect to the Lebesgue measure). It is easy to see that the same holds for the measure $\mu$ from Theorem~\ref{thm:stationary}. Hence, we have the following.

\begin{prop}\label{prop:either}
The stationary measure $\mu$ is either absolutely continuous or singular with respect to the Lebesgue measure.
\end{prop}

Another standard fact is that $\mu$ cannot have atoms.

\begin{prop}\label{prop:nonatom} The stationary measure $\mu$ is non-atomic.
\end{prop}
\begin{proof} The proof follows \cite[proof of Lemma 2]{gelfert} (see also \cite[Lemme~5.1]{DKN1}). By Theorem~\ref{thm:stationary}, $\mu$ has no atoms at $0, 1$. Suppose there exists an atom in $(0,1)$ and take $x \in (0,1)$ such that $\mu(\{x\}) = \max\{\mu(\{y\}): y \in (0,1)\}$. Then, by the definition of stationary measure, $\mu(\{f_i^{-1}(x)\}) = \mu(\{x\})$ for every $i = 1, \ldots, m$ and, consequently, $\mu(\{f_i^{-n}(x)\}) = \mu(\{x\}) > 0$ for every $n > 0$. Since $f_i$ has no fixed points in $(0,1)$, the trajectory $\{f_i^{-n}(x)\}_{n=0}^\infty$ is strictly monotonic and thus infinite, which contradicts the finiteness of $\mu$.
\end{proof}

The following lemma is useful in determining singularity of the measure.

\begin{lem}\label{lem:supp} If $X \subset (0,1)$ is non-empty, closed as a subset of $(0,1)$, and $f(X) \subset X$, then $\supp \mu \subset X \cup \{0, 1\}$ and $\mu(X) = 1$. Consequently, if there exists such a set $X$ of Lebesgue measure $0$, then $\mu$ is singular.
\end{lem}
\begin{proof}
Take $x \in X$. Since $x \in (0,1)$, the Dirac measure $\delta_x$ at $x$ is in $\DD_{c, \alpha, \delta}$ for sufficiently small $\delta > 0$, so by Remark \ref{rem:stability} we have
\[
\frac{1}{N} \sum \limits_{n=0}^{N-1} \TT^n \delta_x \to \mu \quad \text{as } N \to \infty \quad \text{in weak-* topology}.
\]
Since
\[
\TT^n \delta_x = \sum_{i_1, \ldots, i_n \in \{1,\ldots, m\}} p_{i_1} \cdots p_{i_n} \delta_{f_{i_n} \circ \cdots \circ f_{i_1}(x)},
\]
the measures $\frac{1}{N} \sum \limits_{n=0}^{N-1} \TT^n \delta_x$ have topological support in $\overline{\bigcup \limits_{n=0}^{N-1}f^{n}(\{x\})}$, which is contained in $\overline X$, as $f(X) \subset X$. Since $\overline{X} = X \cup \{0, 1\}$ and $\mu(\{0,1\}) = 0$, we have $\supp \mu \subset X \cup \{0, 1\}$ and $\mu(X) = 1$. 
\end{proof}

\section{Preliminary results and proof of Theorem~\ref{thm:Leb}}\label{sec:other_results}

In this section we prove Theorem~\ref{thm:Leb} together with other preliminary results on the AM-systems. We begin with the following observation.

\begin{lem}\label{lem:x_1} Let $\{f_-, f_+\}$ be an AM-system. If the Lyapunov exponents $\Lambda(0), \Lambda(1)$ are positive for the probabilities $p_- = p_+ = 1/2$, then $x_+ < x_-$. In particular, $x_+ < x_-$ holds if the system is symmetric and exhibits a $(k:l)$-resonance for $k,l \in \N$, $k > l$.
\end{lem}
\begin{proof}
The inequality $x_+ < x_-$ can be written as
\[
\frac{1 - a_+}{b_+ - a_+} < \frac{b_- - 1}{b_- - a_-},
\]
which is equivalent to
\begin{equation}\label{eq:x<x}
(1-a_-)(1-a_+) < (b_--1)(b_+-1).
\end{equation}
By the positivity of the Lyapunov exponents for $p_- = p_+ = 1/2$, 
\[
b_- > \frac{1}{a_+}, \qquad b_+ > \frac{1}{a_-},
\]
so
\[
(b_--1)(b_+-1) > \left(\frac{1}{a_+} - 1\right) \left(\frac{1}{a_-} - 1\right) = \frac{(1-a_-)(1-a_+)}{a_-a_+} > (1-a_-)(1-a_+),
\]
which gives \eqref{eq:x<x}. As already noted, if the system is symmetric and exhibits a $(k:l)$-resonance for $k > l$, then the assumption on the positivity of the Lyapunov exponents for $p_- = p_+ = 1/2$ is satisfied. Indeed, in this case we have $a_- = a_+ = a \in (0,1)$ and $b_- = b_+  = a^{-k/l}$, so 
\[
\frac{1}{2}\ln f'_-(x) + \frac{1}{2}\ln f'_+(x)) =  \frac{1-k/l}{2}\ln a > 0.
\]
for $x = 0, 1$.
\end{proof}

The following lemma is used in the proof of Theorem~\ref{thm:Leb}.

\begin{lem}\label{lem:hor} If an AM-system $\{f_-, f_+\}$ with probabilities $p_-, p_+$ is of border type and $\frac{p_-}{a_-} + \frac{p_+}{b_+} = 1$, then $\frac{p_-}{b_-} + \frac{p_+}{a_+} = 1$. Conversely, if 
\[
\frac{p_-}{a_-} + \frac{p_+}{b_+} = \frac{p_-}{b_-} + \frac{p_+}{a_+} = 1, 
\]
then the AM-system $\{f_-, f_+\}$ with probabilities $p_-, p_+$ is of border type.
\end{lem}
\begin{proof}
An elementary calculation shows that the system is of border type if and only if
\[
\frac{a_-+a_+-1}{a_- a_+} = \frac{b_-+b_+-1}{b_- b_+},
\]
which is equivalent to
\begin{equation}\label{eq:hor}
\frac{1 - 1/b_+}{1/a_- - 1/b_+} = \frac{1 - 1/a_+}{1/b_- - 1/a_+}. 
\end{equation}

Suppose that the system is of border type and 
\[
\frac{p_-}{a_-} + \frac{p_+}{b_+} = 1.
\]
Then 
\[
p_- = \frac{1 - 1/b_+}{1/a_- - 1/b_+},
\]
so by \eqref{eq:hor},
\[
p_- = \frac{1 - 1/a_+}{1/b_- - 1/a_+},
\]
which gives
\[
\frac{p_-}{b_-} + \frac{p_+}{a_+} = 1.
\]

Conversely, suppose 
\[
\frac{p_-}{a_-} + \frac{p_+}{b_+} =  \frac{p_-}{b_-} + \frac{p_+}{a_+} = 1. 
\]
Then
\[
p_- = \frac{1 - 1/b_+}{1/a_- - 1/b_+} = \frac{1 - 1/a_+}{1/b_- - 1/a_+},
\]
which gives \eqref{eq:hor}.
\end{proof}

The following proposition, which gives the first part of Proposition~\ref{prop:min,supp}, is essentially proved in \cite[Lemma 3]{ilyashenko10} and \cite[Proposition~2.1]{homburg} (formally, in the case of diffeomorphisms). For completeness, we present the proof suited to our setup.

\begin{prop}\label{prop:minimal} If an AM-system $\{f_-, f_+\}$ has no resonance at one of the endpoints $0, 1$, then it is minimal in $(0,1)$.
\end{prop}
\begin{proof} To fix notation, assume that the system has no resonance at $0$ (in the other case the proof is analogous). Choose $x_0 \in (0,1)$. Since both families of intervals $[f^{n+1}_-(x_0), f^n_-(x_0))$, $n \in \Z$, and $[f^n_+(x_0), f^{n+1}_+(x_0))$, $n \in \Z$, cover $(0,1)$, it is sufficient to prove that for every $x,y \in K$, where 
\[
K = [f^{n_0}_+(x_0), f^{n_0+1}_+(x_0))
\]
with some chosen $n_0 \in \Z$ and every $\varepsilon > 0$ there exist $n \in \N$ and $i_1, \ldots, i_n \in \{-,+\}$ such that 
\begin{equation}\label{eq:<eps}
f_{i_n} \circ \cdots \circ f_{i_1}(x) \in K \qquad \text{and} \qquad |f_{i_n} \circ \cdots \circ f_{i_1}(x) - y| < \varepsilon. 
\end{equation}
To show \eqref{eq:<eps}, we choose $n_0$ so that $K \subset (0, x_+)$ and let
\[
\alpha = -\frac{\ln a_-}{\ln b_+},
\]
Since we assume that $\{f_-, f_+\}$ has no resonance at $0$, we have $\alpha \in \R^+ \setminus \Q$. Hence, for any $y \in K$ and $\delta > 0$ we can find $k, l \in \N$ such that 
\begin{equation}\label{eq:<delta}
0 < k - \alpha l  - \frac{\ln(y/x)}{\ln b_+} < \delta.
\end{equation}
As
\[
b_+^k a_-^l x = e^{(k - \alpha l) \ln b_+ + \ln x} = y b_+^{k - \alpha l -\ln(y/ x)/\ln b_+},
\]
\eqref{eq:<delta} implies
\[
y < b_+^k a_-^lx < y b_+^\delta = y + y(b_+^\delta - 1) < y + b_+^\delta - 1 < y + \min(\varepsilon, \sup K - y),
\]
if $\delta$ is chosen sufficiently small. In particular,
\begin{equation}\label{eq:ab}
b_+^k a_-^l x \in K \qquad \text{and} \qquad |b_+^k a_-^l x  - y| < \varepsilon.
\end{equation}
Since $x \in K \subset (0, x_+)$, we have $f_-^l(x) = a_-^l x$. Moreover, \eqref{eq:ab} implies $b_+^j a_-^l x \in (0, x_+)$ for $ j = 0, \ldots, k$, which gives $f_+^k(f_-^l(x)) = b_+^k a_-^l x$. This together with \eqref{eq:ab} shows \eqref{eq:<eps} and ends the proof.
\end{proof}

Assume now that an AM-system $\{f_-, f_+\}$ with probabilities $p_-, p_+$ has positive Lyapunov exponents, which is equivalent to 
\[
a_-^{p_-} b_+^{p_+} > 1, \qquad b_-^{p_-} a_+^{p_+} > 1.
\]
Then, by Theorem~\ref{thm:stationary}, there exists a unique probability stationary measure $\mu$ for the system, such that $\mu(\{0, 1\}) = 0$. By Propositions~\ref{prop:either} and~\ref{prop:nonatom}, we have the following.

\begin{prop}\label{prop:AM-either}
The stationary measure $\mu$ is non-atomic. Moreover, it is either absolutely continuous or singular with respect to the Lebesgue measure.
\end{prop}

Propositions~\ref{prop:min->supp} and~\ref{prop:minimal} imply the following corollary, which completes the proof of Proposition~\ref{prop:min,supp}.

\begin{cor}\label{cor:min-supp} If the system has no resonance at one of the endpoints $0, 1$, then $\supp \mu = [0, 1]$.
\end{cor}

We end the section by proving Theorem~\ref{thm:Leb}.

\begin{proof}[Proof of Theorem~\rm \ref{thm:Leb}] The transfer operator $T$ on $L^1([0,1], \Leb)$ has the form
\[
Tg = p_- \, (f_-^{-1})' \, g \circ f_-^{-1}  + p_+ \, (f_+^{-1})' \, g \circ f_+^{-1},
\]
The measure $\mu$ is the Lebesgue measure if and only if
\begin{equation}\label{eq:T(1)}
T{\mathds 1} = {\mathds 1}
\end{equation}
for the constant unity function ${\mathds 1}$. 
If the system is of border type, then
\[
T{\mathds 1}(x) =
\begin{cases}
\frac{p_-}{a_-} + \frac{p_+}{b_+}& \text{for } x \le f_-(x_-)\\
\frac{p_-}{b_-} + \frac{p_+}{a_+}& \text{for } x >f_-(x_-)
\end{cases},
\]
so \eqref{eq:T(1)} is equivalent to 
\begin{equation}\label{eq:frac}
\frac{p_-}{a_-} + \frac{p_+}{b_+} = \frac{p_-}{b_-} + \frac{p_+}{a_+} = 1. 
\end{equation}
Conversely, if \eqref{eq:T(1)} holds, then applying it to points $x \in [0,1]$ close to the endpoints of $[0, 1]$ we get \eqref{eq:frac}. To end the proof, it is enough to use Lemma~\ref{lem:hor}.
\end{proof}

\begin{rem} As noted in the introduction, for the case $a_- = a_+ = a$, $b_- = b_+ = b$, $1/a + 1/b = 2$, $p_- = p_- = 1/2$, Theorem~\ref{thm:Leb} was proved in \cite[Theorem~6.1]{alseda-misiurewicz}.
\end{rem}

\section{Proofs of Theorems~\ref{thm:(k:l)} (case $l = 1$) and~\ref{thm:(k:1)}.}
\label{sec:l=1}

In Theorems~\ref{thm:(k:l)} and~\ref{thm:(k:1)} we consider a symmetric AM-system of disjoint type $\{f_-, f_+\}$ with probabilities $p_-, p_+$, positive Lyapunov exponents and a $(k:l)$-resonance for some relatively prime $k,l\in\N$, $k > l$. 
In this section we prove the results in the case $l=1$. The proof is divided into several parts concerning consecutive assertions of the theorems. 

\subsection*{Preliminaries}
By assumption, $a_-=a_+=\rho$, $b_-=b_+=\rho^{-k}$, so the maps have the form
\[
f_-(x)=
\begin{cases}
\rho x &\text{for }x\in[0,x_-]\\
\II(\rho^{-k}\II(x)) &\text{for }x\in (x_-, 1]
\end{cases}, \qquad
f_+(x)=
\begin{cases}
\rho^{-k} x &\text{for }x\in[0,x_+]\\
\II(\rho\II(x)) &\text{for }x\in (x_+, 1]
\end{cases},
\]
where $\rho \in (0,1)$ and
\begin{alignat*}{2}
x_- &= \frac{1 - \rho^k}{1 - \rho^{k+1}}, &\qquad x_+ &= \II(x_-) = \frac{\rho^k - \rho^{k+1}}{1 - \rho^{k+1}},\\
f_-(x_-) &= \frac{\rho - \rho^{k+1}}{1 - \rho^{k+1}}, &\qquad f_+(x_+) &= \II(f_-(x_-)) = \frac{1 - \rho}{1 - \rho^{k+1}}.
\end{alignat*}
Note that $x_+ < x_-$ (see Lemma~\ref{lem:x_1}) and $x_+ < f_-(x_-)$.
The assumption that the system is of disjoint type, i.e.~the condition $f_-(x_-) < f_+(x_+)$, is equivalent to 
\begin{equation}\label{eq:desc}
\rho^{k+1} - 2\rho +1 > 0
\end{equation}
and also to $\rho x_-<\frac 1 2$. 
For the function $h(\rho) = \rho^{k+1} - 2\rho +1$, $\rho \ge 0$ we have  $h(1/2) > 0$, $h(1) = 0$, $h'(\rho) < 0$ for $\rho < \rho_0$ and $h'(\rho) > 0$ for $\rho > \rho_0$, where $\rho_0 = (2/(k+1))^{1/k}\in (1/2, 1)$. This implies that $h$ on $(0, 1)$ has a unique zero $\eta \in (1/2, 1)$, i.e. 
\[
\eta^{k+1} - 2 \eta + 1 = 0
\]
and the condition $\rho < \eta$ is equivalent to \eqref{eq:desc} (this shows Remark~\ref{rem:l=1} in the case $l=1$).

Since the system is symmetric, in fact we have
\begin{equation}\label{eq:1/2}
x_+ < f_-(x_-) < \frac{1}{2} < f_+(x_+) < x_+.
\end{equation}

A simple computation shows that the condition of the positivity of the Lyapunov exponents is equivalent to
\begin{equation}\label{eq:p12}
p_-, p_+ \in \Big( \frac{1}{k+1}, \frac{k}{k+1}\Big).
\end{equation}
Note that the above considerations prove Remark~\ref{rem:(k:1)}.

\subsection*{Construction of the set \boldmath $\Lambda$}
Now we construct a set $\Lambda \in (0,1)$ which will be shown later to be the support of the measure $\mu$ in $(0,1)$. Our strategy is the following. First, we construct a family of disjoint closed intervals $I_j$, $j\in\Z^*$, with the union $I= \bigcup_{j \in \Z^*} I_j$ being forward-invariant under $\{f_-, f_+\}$. The disjointness of $I_j$ follows from the assumption that the system is of disjoint type. We check that the intervals $I_{-k}, \ldots , I_{-1}$ are mapped by $f_+$ into $I_{1}$ with separation gaps, i.e.~$f_+(I_{-k}), \ldots , f_+(I_{-1})$ are disjoint subsets of $I_{1}$ (see Lemma~\ref{lem:intervals} and Figure~\ref{fig:intervals}). Further iterates of these images and their similar copies generate an infinite collection of disjoint Cantor sets, whose union $\Lambda$ is fully invariant and minimal under the action of $\{ f_-, f_+\}$ (see Proposition \ref{prop:X}). As we wish to calculate the dimension of $\Lambda$, it is convenient to describe $\Lambda$ as the union of the attractor $\Lambda_{-1}$ of a self-similar iterated function system $\{ \phi_r\}_{r=1}^k$ on $I_{-1}$ and its similar copies. Moreover, as the successive levels of the Cantor set $\Lambda_{-1}$ are produced during jumps over the central interval, we obtain a characterization of $\Lambda$ in terms of limit points of trajectories jumping over the central interval infinitely many times (see Proposition \ref{prop:omega}).

Let
\[
I_{-1} = [\rho f_+(x_+), \rho x_-] = [\rho f_+(x_+), f_-(x_-)] = [\rho\II(f_-(x_-)), f_-(x_-)] =  
\left[ \frac{\rho-\rho^2}{1 - \rho^{k+1}}, \frac{\rho - \rho^{k+1}}{1 - \rho^{k+1}} \right]
\]
and for $j \in \Z^*$ define
\[
I_j = 
\begin{cases}
\rho^{-j-1} I_{-1} & \text{for } j < 0\\
\II(\rho^{j-1}I_{-1}) &\text{for } j > 0
\end{cases}.
\]
The following lemma is elementary and describes the combinatorics of the intervals $I_j, j \in \Z^*$.

\begin{lem}\label{lem:intervals} The following statements hold.

\begin{enumerate}[\rm (a)]
\item $I_{-j} = \II(I_j)$ for $j \in \Z^*$.
\item The sets $I_j$, $j \in \Z^*$ are pairwise disjoint and situated in $(0,1)$ in the increasing order with respect to $j$.
\item $\inf I_{-k} = x_+$, $\sup I_k = x_-$, $\sup I_{-1} = f_-(x_-)$, $\inf I_1 = f_+(x_+)$. In particular,
\[
f_-(x) = 
\begin{cases}
\rho x & \text{for } x \in \bigcup_{j = -\infty}^k I_j\\
\II(\rho^{-k} \II(x)) & \text{for } x \in \bigcup_{j = k+1}^\infty I_j
\end{cases},
\qquad 
f_+(x) = 
\begin{cases}
\rho^{-k} x & \text{for } x \in \bigcup_{j = -\infty}^{-k-1} I_j\\
\II(\rho \II(x)) & \text{for } x \in \bigcup_{j = -k}^\infty I_j
\end{cases}.
\]
\item $f_-(I_j) = I_{j-1}$ for $j \le -1$, $f_-(\conv(I_1 \cup \cdots \cup I_k)) = I_{-1}$, $f_-(I_j) = I_{j-k}$ for $j \ge k + 1$.
\item $f_+(I_j) = I_{j+k}$ for $j \le -k - 1$,
$f_+(\conv(I_{-k} \cup \cdots \cup I_{-1})) = I_1$,
$f_+(I_j) = I_{j+1}$ for $j \ge 1$.

\end{enumerate}
See Figure~{\rm \ref{fig:intervals}}.

\begin{figure}[ht!]
\begin{center}
\scalebox{0.97}{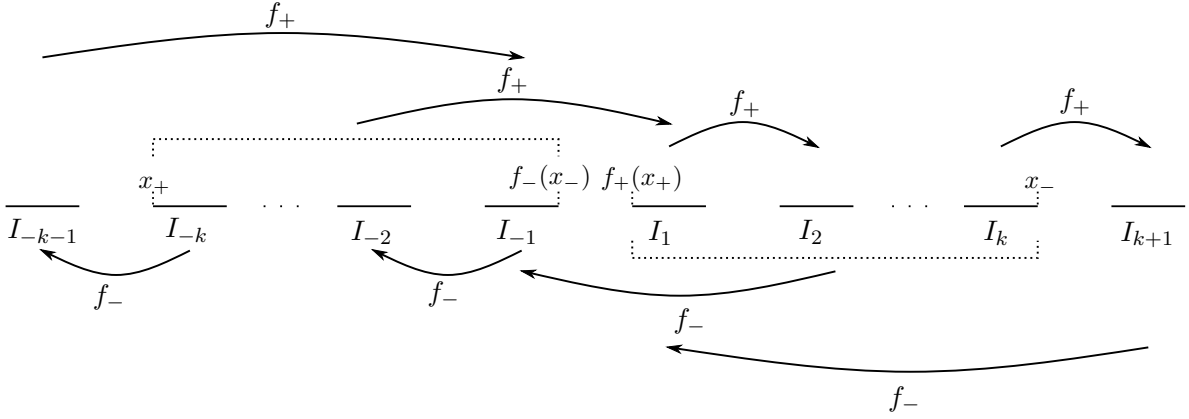}
\end{center}
\caption{A schematic view of the action of $\{f_-, f_+\}$ on the intervals $I_j$.}\label{fig:intervals}
\end{figure}

\end{lem}
\begin{proof} 
The assertion (a) follows directly from the definition of $I_j$. 
To show (b), we first check $\sup I_{-2}  < \inf I_{-1}$. This is equivalent to 
\[
\rho\frac{\rho - \rho^{k+1}}{1 - \rho^{k+1}} < \frac{\rho - \rho^2}{1 - \rho^{k+1}},
\]
which boils down to \eqref{eq:desc}. By \eqref{eq:1/2}, $\sup I_{-1} < \inf I_1$. 
The rest of the assertion (b) follows directly from the above facts and the definition of $I_j$. 

The assertions (c)--(e) are easy consequences of the definition of $I_j$, the symmetry of the system and the fact
\[
f_-^{-1}(x) = \rho^{-1} x, \qquad f_+^{-1}(x) = \rho^k x \qquad\text{for } x \in \bigcup_{j < 0} I_j,
\]
which follows from the definition of $f_\pm$.

\end{proof}

Let 
\[
I = \bigcup_{j\in\Z} I_j, \qquad I^- = \bigcup_{j < 0} I_j, \qquad I^+ = \bigcup_{j>0} I_j.
\]
Note that Lemma~\ref{lem:intervals} implies $f(I) \subset I$. More precisely, for every $i \in \{-,+\}$ and $j \in \Z^*$ we have 
\[
f_i(I_j) \subset I_{j'} \qquad\text{for some } j' = j'(i,j) \in \Z^*. 
\]

\begin{lem}\label{lem:inI} For every $x \in (0,1)$ there exists $i_1, \ldots, i_n \in \{-,+\}$, $n \ge 0$, such that $f_{i_n} \circ \cdots \circ f_{i_1}(x) \in I$. 
\end{lem}
\begin{proof} Enumerate the components of $(0,1) \setminus I$ by $U_j$, $j \in \Z$, such that $U_j$ is the gap between $I_{j-1}$ and $I_j$ for $j < 0$, $U_0$ is the gap between $I_{-1}$ and $I_1$, and $U_j$ is the gap between $I_j$ and $I_{j+1}$ for $j > 0$. Take $x \in (0,1) \setminus I$. Since the system is symmetric, we can assume $x \in U_j$, $j \le 0$. Then to prove the lemma it is enough to notice that by Lemma~\ref{lem:intervals}, we have $f_-\circ f_+^{\lfloor \frac{-j}{k} \rfloor +1}(x) \in I_{-1}$.
\end{proof}

Consider the maps
\[
\phi_r\colon I_{-1} \to I_{-1}, \qquad \phi_r(x) = \rho - \rho^r x , \quad r = 1, \ldots, k.
\]
Note that
\begin{equation}\label{eq:phi-f}
\phi_r(x) = \rho\II(\rho^{r-1}x) = f_-(\II(\rho^{r-1}x)) = \II(f_+(\rho^{r-1}x))
\end{equation}
for $x \in I_{-1}$. Obviously, the maps $\phi_r$ are contracting similarities with $\|\phi_r'\| = \rho^r$. 

Let
\[
\Lambda_{-1} = \bigcap_{n = 1}^\infty \; \bigcup_{r_1, \ldots, r_n = 1}^k \phi_{r_1} \circ \cdots \circ \phi_{r_n} (I_{-1})
\]
be the limit set of the iterated function system generated by  $\{\phi_r\}_{r=1}^k$ on $I_{-1}$. Recall that it is the unique non-empty compact set in $I_{-1}$ satisfying
\[
\Lambda_{-1} = \bigcup_{r=1}^k \phi_r(\Lambda_{-1})
\]
(see e.g.~\cite[Theorem 9.1]{falconer}). For $j \in \Z^*$ define
\[
\Lambda_j = 
\begin{cases}
\rho^{-j-1} \Lambda_{-1} &\text{for } j < 0\\ 
\II(\rho^{j-1} \Lambda_{-1} ) &\text{for } j > 0 
\end{cases},
\qquad \Lambda = \bigcup_{j\in \Z^*} \Lambda_j.
\]
Obviously, $\Lambda_j$ are pairwise disjoint compact sets and $\Lambda_j \subset I_j$. Furthermore, for $n \ge 0$, $r_1, \ldots, r_n \in \{1, \ldots, k\}$ let
\[
I_{j; r_1, \ldots, r_n} = 
\begin{cases}
\rho^{-j-1}\phi_{r_1} \circ \cdots \circ \phi_{r_n} (I_{-1}) &\text{for } j < 0\\
\II(\rho^{j-1}\phi_{r_1} \circ \cdots \circ \phi_{r_n} (I_{-1})) &\text{for } j > 0
\end{cases},
\]
where for $n = 0$ we set $I_{j; r_1, \ldots, r_n} = I_j$, $\phi_{r_1} \circ \cdots \circ \phi_{r_n} = \id$.
Since $|\phi_r'| = \rho^r$, for every $j \in \Z^*$ and an infinite sequence $r_1, r_2, \ldots \in \{1, \ldots, k\}$ the segments $I_{j; r_1, \ldots, r_n}$, $n \ge 0$, form a nested sequence of sets, such that
\[
|I_{j; r_1, \ldots, r_n}| = \rho^{|j| - 1 + r_1 + \cdots + r_n} \le \rho^n \to 0 \text{ as } n \to \infty,
\]
so
\[
\bigcap_{n=1}^\infty I_{j; r_1, \ldots, r_n} = \{x_{j; r_1, r_2, \ldots}\}
\]
for a point $x_{j; r_1, r_2, \ldots} \in \Lambda$ and
\[
\Lambda = \bigcup_{j \in \Z^*} \bigcap_{n=1}^\infty \bigcup_{r_1, \ldots, r_n = 1}^k 
I_{j; r_1, \ldots, r_n} =
\{ x_{j; r_1, r_2, \ldots}: \: j \in \Z^*, r_1, r_2, \ldots \in \{1, \ldots, k\}\}.
\]

\subsection*{Description of trajectories}

Lemma~\ref{lem:intervals} and \eqref{eq:phi-f} imply immediately the following.

\begin{lem}\label{lem:f(I)} 
For $j \in \Z^*$, $r_1, r_2, \ldots \in \{1, \ldots, k\}$, $n \ge 0$,
\[
\begin{aligned}
f_-(I_{j; r_1, \ldots, r_n}) &= 
\begin{cases}
I_{j-1; r_1, \ldots, r_n} &\text{for } j < 0\\
I_{-1; j, r_1, \ldots, r_n} &\text{for } 1 \le j \le k\\
I_{j-k; r_1, \ldots, r_n} &\text{for } j > k
\end{cases},\\
f_+(I_{j; r_1, \ldots, r_n}) &= 
\begin{cases}
I_{j+k; r_1, \ldots, r_n} &\text{for } j < -k\\
I_{1; -j, r_1, \ldots, r_n} &\text{for } -k \le j \le -1\\
I_{j+1; r_1, \ldots, r_n} &\text{for } j > 0
\end{cases}
\end{aligned}
\]
and
\[
\begin{aligned}
f_-(x_{j; r_1, r_2, \ldots}) &= 
\begin{cases}
x_{j-1; r_1, r_2, \ldots} &\text{for } j < 0\\
x_{-1; j, r_1, r_2, \ldots} &\text{for } 1 \le j \le k\\
x_{j-k; r_1, r_2, \ldots} &\text{for } j > k
\end{cases},\\
f_+(x_{j; r_1, r_2, \ldots}) &= 
\begin{cases}
x_{j+k; r_1, r_2, \ldots} &\text{for } j < -k\\
x_{1; -j, r_1, r_2, \ldots} &\text{for } -k \le j \le -1\\
x_{j+1; r_1, r_2, \ldots} &\text{for } j > 0
\end{cases}.
\end{aligned}
\]
\end{lem}

The following lemmas characterize trajectories jumping over the central interval. The first one follows directly from Lemma~\ref{lem:intervals}.

\begin{lem}\label{lem:jump} The following statements hold.
\begin{itemize}
\item[(a)] If a trajectory $\{f_{i_n}\circ \cdots \circ f_{i_1}(x)\}_{n = 0}^\infty$, for $x \in (0,1)$, jumps over the central interval at the time $s$, for $s \ge 0$, then $f_{i_{s+1}} \circ \cdots \circ f_{i_1}(x) \in I_{-1} \cup I_1$.
\item[(b)] A trajectory $\{f_{i_n}\circ \cdots \circ f_{i_1}(x)\}_{n = 0}^\infty$, for $x \in I$, jumps over the central interval at the time $s$, for $s \ge 0$, if and only if
\begin{alignat*}{3}
&f_{i_s} \circ \cdots \circ f_{i_1}(x) \in \bigcup_{j=-k}^{-1} I_j, \qquad &i_{s+1} &= +, \qquad
&f_{i_{s+1}} \circ \cdots \circ f_{i_1}(x) &\in I_1\\
&\text{or}\\
&f_{i_s} \circ \cdots \circ f_{i_1}(x) \in \bigcup_{j=1}^{k} I_j, \qquad &i_{s+1} &= -,\qquad 
&f_{i_{s+1}} \circ \cdots \circ f_{i_1}(x) &\in I_{-1}.
\end{alignat*}
\end{itemize}
In particular, for given $j \in \Z^*$ and $i_1, i_2, \ldots \in \{-, +\}$, for all $x \in I_j$ the trajectories $\{f_{i_n}\circ \cdots \circ f_{i_1}(x)\}_{n = 0}^\infty$ jump over the central interval at the same times.
\end{lem}

For $j, j' \in \Z^*$ such that $\sgn(j) = \sgn(j')$, define
\begin{align*}
F_{j,j'}\colon I_j \onto I_{j'}, \qquad F_{j, j'} =
\begin{cases}
f_-^{j-j'}|_{I_j} &\text{for } j < 0, \; j' \leq j\\
f_+^{\lceil (j'-j)/k\rceil} \circ f_-^{j-j'+k\lceil (j'-j)/k\rceil}|_{I_j} &\text{for } j < 0, \; j' > j\\
f_+^{j'-j}|_{I_j} &\text{for } j > 0, \; j' \geq j\\
f_-^{\lceil (j-j')/k\rceil} \circ f_+^{j'-j+k\lceil (j-j')/k\rceil}|_{I_j} &\text{for } j > 0, \; j' < j
\end{cases}.
\end{align*}
Note that $F_{j,j'} = f_{i_n}\circ \cdots \circ f_{i_1}|_{I_j}$ for some $i_1, \ldots, i_n \in \{-,+\}$, $n \ge 0$, and, by Lemma~\ref{lem:intervals},
\begin{equation}\label{eq:F-phi}
F_{j, j'}(x) =
\begin{cases}
\rho^{j-j'}x  &\text{for } j < 0\\
\II(\rho^{-j+j'}\II(x)) &\text{for } j > 0
\end{cases}
\end{equation}
for $x \in I_j$. In particular, this implies
\[
F_{j, j} = \id|_{I_j}, \qquad F_{j', j''}  \circ  F_{j, j'} = F_{j, j''}
\]
for $j, j', j'' \in \Z^*$ such that $\sgn(j) = \sgn(j') =\sgn(j'')$.
By Lemma~\ref{lem:f(I)},
\begin{equation}\label{eq:F}
F_{j, j'}(I_{j; r_1, \ldots, r_n}) = I_{j'; r_1, \ldots, r_n}, \qquad 
F_{j, j'}(x_{j; r_1, r_2, \ldots}) = x_{j'; r_1, r_2, \ldots}
\end{equation}
for $r_1, r_2, \ldots \in \{1, \ldots, k\}$, $n \ge 0$.

\begin{lem}\label{lem:F}
A trajectory $\{f_{i_n} \circ \cdots \circ f_{i_1}(x)\}_{n=0}^\infty$ of a point $x \in I_j$, $j \in \Z^*$, does not jump over the central interval at any time $0 \le s < n$, for some $n \ge 0$, if and only if
\[
f_{i_n} \circ \cdots \circ f_{i_1}|_{I_j} = F_{j, j'}
\]
for $j' \in \Z^*$ such that $f_{i_n} \circ \cdots \circ f_{i_1}(x) \in I_{j'}$ and $\sgn(j) = \sgn(j')$.
\end{lem}
\begin{proof}
If a trajectory $\{f_{i_n} \circ \cdots \circ f_{i_1}(x)\}_{n=0}^\infty$ of $x \in I_j$ does not jump over the central interval at any time $0 \le s < n$, then by Lemmas~\ref{lem:intervals} and~\ref{lem:jump},
\[
f_{i_n} \circ \cdots \circ f_{i_1}(x) =
\begin{cases}
\rho^{j-j'}x  &\text{for } j < 0\\
\II(\rho^{-j+j'}\II(x)) &\text{for } j > 0
\end{cases}
\]
for $j' \in \Z^*$ such that $f_{i_n} \circ \cdots \circ f_{i_1}(x) \in I_{j'}$.
Therefore, $f_{i_n} \circ \cdots \circ f_{i_1}|_{I_j} = F_{j, j'}$ by \eqref{eq:F-phi}. The other implication follows directly from Lemmas~\ref{lem:intervals} and~\ref{lem:jump}. 
\end{proof}

Define
\[
G^-_r\colon I_1 \to I_{-1}, \qquad G^+_r\colon I_{-1} \to I_1, \qquad r \in \{1, \ldots, k\}
\]
setting
\[
G^-_r = f_- \circ F_{1,r}, \qquad G^+_r = f_+ \circ F_{-1,-r}.
\]
We have $G^\pm_r = f_{i_n}\circ \cdots \circ f_{i_1}|_{I_{\mp 1}}$ for some $i_1, \ldots, i_n \in \{-,+\}$, $n \ge 0$. Moreover, by \eqref{eq:phi-f} and \eqref{eq:F-phi}, 
\begin{equation}\label{eq:G-phi}
G^-_r = \phi_r \circ \II|_{I_1}, \qquad G^+_r = \II \circ \phi_r, 
\end{equation}
while Lemma~\ref{lem:f(I)} and \eqref{eq:F} imply
\begin{equation}\label{eq:G}
\begin{aligned}
G^-_r(I_{1; r_1, \ldots, r_n}) &= I_{-1; r, r_1, \ldots, r_n}, &\quad 
G^-_r(x_{1; r_1, r_2, \ldots}) &= x_{-1; r, r_1, r_2, \ldots},\\
G^+_r(I_{-1; r_1, \ldots, r_n}) &= I_{1; r, r_1, \ldots, r_n}, &\quad 
G^+_r(x_{-1; r_1, r_2, \ldots}) &= x_{1; r, r_1, r_2, \ldots}
\end{aligned}
\end{equation}
for $r_1, r_2, \ldots \in \{1, \ldots, k\}$, $n \ge 0$. 

\begin{lem}\label{lem:G}
A trajectory $\{f_{i_n} \circ \cdots \circ f_{i_1}(x)\}_{n=0}^\infty$ of a point $x \in I$ jumps over the central interval at the time $s$, for some $s \ge 0$, if and only if
\begin{alignat*}{3}
&\qquad f_{i_s} \circ \cdots \circ f_{i_1}(x) &&\in I_{-r}, \qquad
&f_{i_{s+1}}|_{I_{-r}} &= G^+_r \circ F_{-r,-1}\\
&\text{or}\\
&\qquad f_{i_s} \circ \cdots \circ f_{i_1}(x) &&\in I_r, \qquad
&f_{i_{s+1}}|_{I_r} &= G^-_r \circ F_{r,1}
\end{alignat*}
for some $r \in \{1, \ldots, k\}$.
\end{lem}
\begin{proof}
Follows directly from Lemmas~\ref{lem:jump}, \ref{lem:F} and the definitions of the maps $F_{j,j'}$, $G^\pm_r$. 
\end{proof}

\begin{lem}\label{lem:traj}
A trajectory $\{f_{i_n} \circ \cdots \circ f_{i_1}(x)\}_{n=0}^\infty$ of a point $x \in I_j$, $j \in \Z^*$, jumps over the central interval $($exactly$)$ at the times 
$s_1, \ldots, s_m$, for some $0 \le s_1 < \cdots < s_m < n$, $0 \le m \le n$, 
if and only if 
\[
f_{i_n} \circ \cdots \circ f_{i_1}|_{I_j}= 
\begin{cases}
F_{-1,j'} \circ G^-_{r_1}\circ G^+_{r_2}\circ \cdots \circ G^-_{r_{m-1}}\circ G^+_{r_m} \circ F_{j,-1}& \text{for } j < 0,\; m \text{ even}\\
F_{1,j'} \circ G^+_{r_1}\circ G^-_{r_2}\circ \cdots \circ G^+_{r'_{m-2}} \circ G^-_{r_{m-1}}\circ G^+_{r_m}\circ F_{j,-1}& \text{for } j < 0,\; m \text{ odd}\\
F_{1,j'} \circ G^+_{r_1}\circ G^-_{r_2}\circ \cdots \circ G^+_{r_{m-1}}\circ G^-_{r_m} \circ F_{j,1}& \text{for } j > 0,\; m \text{ even}\\
F_{-1,j'} \circ G^-_{r_1}\circ G^+_{r_2}\circ \cdots \circ G^-_{r_{m-2}} \circ G^+_{r_{m-1}}\circ G^-_{r_m}\circ F_{j,1}& \text{for } j > 0,\; m \text{ odd}
\end{cases}
\]
for some $j' \in \Z^*$ and $r_1, \ldots, r_m \in \{1, \ldots, k\}$, where $\sgn(j) = \sgn(j')$ when $m$ is even and $\sgn(j) \neq \sgn(j')$ when $m$ is odd. Moreover, in this case we have
\[
f_{i_n} \circ \cdots \circ f_{i_1}(I_j) = I_{j'; r_1, \ldots, r_n} 
\]
and
\[
f_{i_n} \circ \cdots \circ f_{i_1}(x) = 
\begin{cases}
\rho^{-j'-1} \phi_{r_1} \circ \cdots \circ \phi_{r_m}(\rho^{j+1}x) & \text{for }j<0,\,  m \text{ even}\\
\II(\rho^{j'-1} \phi_{r_1} \circ \cdots \circ \phi_{r_m}(\rho^{j+1}x)) & \text{for }j < 0, \, m \text{ odd}\\
\rho^{-j'-1} \phi_{r_1} \circ \cdots \circ \phi_{r_m}(\rho^{-j+1}\II(x)) & \text{for }j > 0,\, m \text{ even}\\
\II(\rho^{j'-1} \phi_{r_1} \circ \cdots \circ \phi_{r_m}(\rho^{-j+1}\II(x))) & \text{for }j>0, \, m \text{ odd}
\end{cases}.
\]
\end{lem}
\begin{proof} Follows directly from Lemmas~\ref{lem:F} and~\ref{lem:G}, and \eqref{eq:F-phi}, \eqref{eq:F}, \eqref{eq:G-phi}, \eqref{eq:G}.
\end{proof}

\begin{defn}
For $x \in (0,1)$ let $\omega_\infty(x)$ be the set of limit points of all trajectories of $x$ under $\{f_-, f_+\}$, which jump over the central interval infinitely many times, i.e.
\begin{align*}
\omega_\infty(x) = \{&\lim_{s \to \infty} f_{i_{n_s}}\circ \cdots \circ f_{i_1}(x): 
i_1, i_2, \ldots \in \{-,+\},\, n_s \to \infty \text{ as } s \to \infty\\ &\text{and }\{f_{i_n}\circ \cdots \circ f_{i_1}(x)\}_{n = 0}^\infty \text{ jumps over the central interval infinitely many times}\}.
\end{align*}
\end{defn}

\begin{prop}\label{prop:omega} For every $x \in (0,1)$,
\[
\omega_\infty(x) = \Lambda \cup \{0,1\}.
\]
\end{prop}
\begin{proof} 
First, we prove $\omega_\infty(x) \subset \Lambda\cup \{0,1\}$ for $x \in (0,1)$. By Lemma~\ref{lem:jump}(a), we can assume $x \in I$. Take $y \in \omega_\infty(x)$. Then $y = \lim_{s \to \infty} f_{i_{n_s}}\circ \cdots \circ f_{i_1}(x)$, where $n_s \to \infty$ and the trajectory $\{f_{i_n} \circ \cdots \circ f_{i_1}(x)\}_{n=0}^\infty$ jumps over the central interval infinitely many times. By Lemma~\ref{lem:traj}, we have 
\[
f_{i_{n_s}} \circ \cdots \circ f_{i_1}(x) \in I_{j(s); r_1(s), \ldots, r_{m(s)}(s)}
\]
for some $j(s) \in \Z^*$, $m(s) \ge 0$, $r_1(s), \ldots, r_{m(s)}(s) \in \{1, \ldots, k\}$, where $m(s) \to \infty$ as $s \to \infty$. Since $|I_{j(s); r_1(s), \ldots, r_{m(s)}(s)}| \le \rho^{m(s)} \to 0$ as $s \to \infty$, $I_{j(s); r_1(s), \ldots, r_{m(s)}(s)} \cap \Lambda \neq 0$, we have $y \in \overline\Lambda = \Lambda\cup \{0,1\}$. In this way we have showed $\omega_\infty(x) \subset \Lambda\cup \{0,1\}$. 

Now we prove $\Lambda\cup \{0,1\}\subset \omega_\infty(x)$ for $x \in (0,1)$. By Lemma~\ref{lem:inI}, we can assume $x \in I_j$, $j \in \Z^*$.  
Since the system is symmetric, we can assume $j < 0$. Take $y \in \Lambda$. Then $y = x_{j'; r_1, r_2, \ldots}$ for some $j' \in \Z^*$, $r_1, r_2, \ldots \in \{1, \ldots, k\}$. Let
\[
F^{(0)} = 
\begin{cases}
F_{j,j'}&\text{if } j' < 0\\
F_{1,j'}\circ G^+_1\circ F_{j,-1}&\text{if } j' > 0
\end{cases}
\]
and note that $F^{(0)}(x) \in I_{j'}$. Define
\[
F^{(n)} = 
\begin{cases}
F_{-1,j'} \circ G^-_{r_1}\circ G^+_{r_2}\circ \cdots \circ G^-_{r_{n-1}}\circ G^+_{r_n} \circ F_{j',-1}&\text{if } j' < 0\\
F_{1,j'} \circ G^+_{r_1}\circ G^-_{r_2}\circ \cdots \circ G^+_{r_{n-1}}\circ G^-_{r_n}\circ F_{j',1}&\text{if } j' > 0
\end{cases}
\]
for even $n > 0$. Then $F^{(n)}$ is well-defined on $I_{j'}$. Using \eqref{eq:F} and \eqref{eq:G} inductively, we see 
\[
F^{(n)} \circ \cdots \circ F^{(2)}\circ F^{(0)}(x) \in I_{j'; r_1, \ldots, r_n}
\]
for every even $n >0$. Since $|I_{j'; r_1, \ldots, r_n}| \le \rho^n \to 0$ as $n \to \infty$ and $\bigcap_{n \text{ even}} I_{j'; r_1, \ldots, r_n} = \{y\}$, the trajectory defined by $\cdots\circ F^{(n)}\cdots \circ F^{(2)}\circ F^{(0)}(x)$ has $y$ as a limit point and, by Lemma~\ref{lem:traj}, jumps over the central interval infinitely many times. This shows $\Lambda\subset \omega_\infty(x)$.

Take now $y \in \{0,1\}$ and define
\[
F^{(0)} = 
\begin{cases}
F_{j,-1}&\text{if } y = 0\\
G^+_1\circ F_{j,-1}&\text{if } y = 1
\end{cases}
\]
and
\[
F^{(n)} = 
\begin{cases}
F_{-1,-n-1} \circ G^-_1\circ G^+_1\circ \cdots \circ G^-_1\circ G^+_1\circ F_{-n+1,-1} &\text{if } y = 0\\
F_{1,n+1} \circ G^+_1\circ G^-_1 \circ \cdots \circ G^+_1\circ G^-_1\circ F_{n-1,1} &\text{if } y = 1
\end{cases}
\]
for even $n > 0$. Then, arguing as previously, we see that
\[
F^{(n)} \circ \cdots \circ F^{(2)}\circ F^{(0)}(x) \in 
\begin{cases}
I_{-n-1} & \text{if } y = 0\\
I_{n+1} & \text{if } y = 1
\end{cases}
\]
for even $n > 0$, the trajectory defined by $\cdots\circ F^{(n)}\cdots \circ F^{(2)}\circ F^{(0)}(x)$ has $y$ as its limit point and jumps over the central interval infinitely many times. This implies $\Lambda \cup \{0,1\} \in \omega_\infty(x)$.
\end{proof}

\begin{prop}\label{prop:X}
We have
\[
\Lambda = f_-(\Lambda) = f_+(\Lambda).
\]
Moreover, the system $\{f_-, f_+\}$ is minimal in $\Lambda$.
\end{prop}

\begin{proof}
The first assertion follows directly from Lemma~\ref{lem:f(I)}, while Proposition~\ref{prop:omega} implies minimality.
\end{proof}

\subsection*{Singularity of \boldmath $\mu$}

\begin{prop}\label{prop:supp} We have 
\[
\supp \mu = \Lambda \cup \{0, 1\},\quad \mu(\Lambda) = 1.
\]
\end{prop}
\begin{proof} By Proposition~\ref{prop:X}, we have $f(\Lambda) = \Lambda$.
Moreover, $\Lambda$ is closed in $(0,1)$. Hence, Lemma~\ref{lem:supp} implies $\supp \mu \subset \Lambda \cup \{0, 1\}$ and $\mu(\Lambda) = 1$. On the other hand, the system is minimal in $\Lambda$ by Proposition~\ref{prop:X}, so Proposition~\ref{prop:min->supp} gives $\supp \mu = \overline{\Lambda} = \Lambda \cup \{0,1\}$.
\end{proof}
\begin{prop}\label{prop:dimX}
\[
\dim_H \Lambda = \dim_B \Lambda = \frac{\log \eta}{\log \rho} < 1,
\]
where $\eta \in (1/2,1)$ is the unique solution of the equation $\eta^{k+1} - 2 \eta + 1 = 0$.
\end{prop}
\begin{proof} By definition, the maps $\phi_r\colon I_{-1} \to I_{-1}$, $r = 1, \ldots, k$, are contractions and
\[
\phi_r(I_{-1}) = \left[\rho - \rho^r \frac{\rho - \rho^{k+1}}{1 - \rho^{k+1}} ,\rho - \rho^r \frac{\rho-\rho^2}{1 - \rho^{k+1}} \right].
\]
Using \eqref{eq:desc}, we check that $\sup \phi_r(I_{-1}) < \inf \phi_{r+1}(I_{-1})$ for $r = 1, \ldots, k-1$. Consequently, $\{\phi_r\}_{r = 1}^k$ is an iterated function system of contracting similarities with scales $\rho, \ldots, \rho^k$, respectively, satisfying the Strong Separation Condition, i.e.~$\overline{\phi_r(I_{-1})} = \phi_r(I_{-1})$, $r = 1, \ldots, k$, are pairwise  disjoint. Therefore, its limit set $\Lambda_{-1}$ is a Cantor set and its Hausdorff (and box) dimension is equal to the unique positive number $d$ satisfying
\[
\rho^d + \cdots + \rho^{kd} = 1
\]
(see e.g.~\cite[Theorem~9.3]{falconer}). This equation is equivalent to $\eta^{k+1} - 2 \eta + 1 = 0$ for $\eta = \rho^d$. Hence, 
\[
\dim_H \Lambda_{-1} = \dim_B \Lambda_{-1} =  d = \frac{\log \eta}{\log \rho}, 
\]
Since $\Lambda_j$, $j \in \Z^*$, are disjoint similar copies of $\Lambda_{-1}$, we have $\dim_H \Lambda = \dim_H \Lambda_{-1}$, as the Hausdorff dimension is countably stable, see e.g.~\cite[Section 3.2]{falconer}. To see $\dim_B \Lambda = \dim_B \Lambda_{-1}$ note that $\Lambda = \psi(A \times \Lambda_{-1}) \cup \II(\psi(A \times \Lambda_{-1}))$, where $A = \{ \rho^j : j \geq 0\}$ and $\psi:A \times \Lambda_{-1} \to [0,1]$ is given as $\psi(\rho^j, x) = \rho^jx$. Since $\psi$ is Lipschitz and $\dim_B A = 0$, we obtain $\dim_B\Lambda \leq \dim_B A + \dim_B \Lambda_{-1} = \dim_B \Lambda_{-1}$.

The condition \eqref{eq:desc}, equivalent to $\rho < \eta$, implies $\dim_H \Lambda_{-1} < 1$.
\end{proof}

Propositions~\ref{prop:supp} and \ref{prop:dimX} imply the following.

\begin{cor}\label{cor:X}
The measure $\mu$ is singular with $\dim_B (\supp \mu) = \dim_H (\supp \mu) < 1$.
\end{cor}

\subsection*{Dimension of \boldmath $\mu$}

To determine the exact form of $\mu$, consider the coding map  for the IFS $\{\phi_r\}_{r=1}^k$ on $I_{-1}$ given by
\[ 
\pi_{-1} \colon \Sigma_k^+ \to \Lambda_{-1}, \qquad
\pi_{-1}(r_1, r_2, \ldots) = \lim \limits_{n \to \infty} \phi_{r_1} \circ \phi_{r_2} \circ \cdots \circ \phi_{r_n}(x) = x_{-1; r_1, r_2, \ldots}, 
\]
where $\Sigma_k^+ = \{1,\ldots, k\}^{\N}$ and $x$ is any point from $I_{-1}$. Note that $\pi_{-1}$ is a homeomorphism, since the IFS satisfies the strong separation condition.
It follows that $\Lambda$ is homeomorphic to $\Z^* \times \Sigma_k^+$ with the topology defined as the product of the discrete topology on $\Z^*$ and the standard (product) topology on $\Sigma_k^+$. The homeomorphism is given by
\[
\pi \colon \Z^* \times \Sigma_k^+ \to X, \qquad \pi(j, r_1, r_2, \ldots) = x_{j; r_1, r_2, \ldots}= 
\begin{cases}
\rho^{-j-1} \pi_{-1}(r_1, r_2, \ldots) &\text{for } j < 0 \\
\II(\rho^{j-1}\pi_{-1}(r_1, r_2, \ldots)) &\text{for } j > 0
\end{cases}.
\]
Let $\tilde f_-, \tilde f_+\colon \Z^* \times \Sigma_k^+ \to \Z^* \times \Sigma_k^+$
be the lifts by $\pi$ of $f_-|_\Lambda$, $f_+|_\Lambda$, respectively, i.e.
\begin{equation}\label{eq:tildef}
\pi \circ \tilde f_i = f_i \circ \pi, \qquad i \in \{-,+\}. 
\end{equation}
Lemma~\ref{lem:f(I)} implies
\begin{equation} \label{eq:tildef1}
\begin{aligned}
\tilde f_-(j, r_1, r_2,\ldots) &= 
\begin{cases}
(j-1, r_1, r_2, \ldots) &\text{for } j < 0\\
(-1, j, r_1, r_2, \ldots) &\text{for } 1 \le  j \le k\\
(j-k, r_1, r_2, \ldots) &\text{for } j > k
\end{cases}\\
\tilde f_+(j, r_1, r_2, \ldots) &= 
\begin{cases}
(j+k, r_1, r_2,  \ldots) &\text{for } j < -k\\
(1, -j, r_1, r_2,  \ldots) &\text{for } -k \le j \le -1\\
(j+1, r_1, r_2,  \ldots) &\text{for } j > 0
\end{cases}. 
\end{aligned}
\end{equation}
Due to \eqref{eq:tildef}, there is a one-to-one correspondence between stationary probability measures for the system $\{f_-, f_+\}$ on $\Lambda$ with probabilities $p_-, p_+$ and for the system $\{\tilde f_-, \tilde f_+\}$ with probabilities $p_-, p_+$, both considered on the $\sigma$-algebra of Borel sets. Since there is a unique stationary probability measure $\mu$ for $\{f_-, f_+\}$ on $\Lambda$, there is also a unique stationary probability measure $\tilde{\mu}$ for $\{\tilde f_-, \tilde f_+\}$. Moreover, $\mu = \pi_* \tilde{\mu}$.

Now we determine the structure of the measure $\tilde \mu$.

\begin{prop}\label{prop:symbolic_measure}
There exist numbers $c_-, c_+ > 0$ and probabilistic vectors $\beta^- = (\beta_1^-, \ldots, \beta_k^-)$, $\beta^+ = (\beta_1^+, \ldots, \beta_k^+)$, such that $c_- \sum \limits_{j = 1}^\infty \eta_-^j + c_+ \sum \limits_{j = 1}^\infty \eta_+^j = 1$, where $\eta_-, \eta_+ \in (0,1)$ are the unique solutions of the equations 
\[
p_+\eta_-^{k+1} - \eta_-  + p_- = 0, \qquad p_-\eta_+^{k+1} - \eta_+  + p_+ = 0,
\]
respectively, and 
\[ 
\tilde \mu = \sum_{j \in \Z^*} \eta_j \delta_j \otimes \nu_j,
\]
where
\[
\eta_j = 
\begin{cases}
c_-\eta_-^{-j} & \text{for } j < 0\\
c_+\eta_+^j & \text{for } j > 0
\end{cases},\]
$\nu_j$ is a probability measure on $\Sigma_k^+$ given by
\[\nu_j =
\begin{cases}
\PP_{\beta^-} \otimes \PP_{\beta^+} \otimes \PP_{\beta^-} \otimes \PP_{\beta^+} \otimes \cdots & \text{for } j<0\\
\PP_{\beta^+} \otimes \PP_{\beta^-} \otimes \PP_{\beta^+} \otimes \PP_{\beta^-} \otimes \cdots & \text{for } j>0
\end{cases}, \qquad j \in \Z^*,
\]
and $\delta_j$ is the Dirac measure at $j$.

\end{prop}
\begin{proof}
Let 
\[
h^-(x) = p_+ x^{k+1} - x + p_-, \qquad h^+(x) = p_- x^{k+1} - x + p_+.
\]
Since $h^\pm$ are convex, $h^\pm(0) > 0$, $h^\pm(1) = 0$ and, by \eqref{eq:p12}, $(h^\pm)'(1) > 0$, the function $h^\pm$ has a unique zero in $(0, 1)$, which determines the values of $\eta_-$, $\eta_+$. Suppose that 
$c_\pm, \beta_1^\pm, \ldots, \beta_k^\pm > 0$ satisfy
\begin{equation}\label{eq:eq2}
c_- \sum \limits_{j = 1}^\infty \eta_-^j + c_+ \sum \limits_{j = 1}^\infty \eta_+^j = 1, \qquad 
\sum_{r = 1}^k \beta_r^- = 1, \qquad
\sum_{r = 1}^k \beta_r^+ = 1.
\end{equation}
Then the measure 
\[
\nu = \sum_{j \in \Z^*} \eta_j \delta_j \otimes \nu_j
\]
for $\eta_j, \nu_j$ as in Proposition~\ref{prop:symbolic_measure} is a probability measure on $\Z^* \times \Sigma_k^+$. 
Let
\[
[j, r_1, \ldots, r_n] = \{(j', r_1', r_2', \ldots) \in \Z^* \times \Sigma_k^+: j' = j, r_1' = r_1, \ldots, r_n' = r_n\}
\]
for $j \in \Z^*$, $n \ge 0$ and $r_1, \ldots, r_n \in \{1, \ldots, k\}$
be the cylinders in $\Z^* \times \Sigma_k^+$.
By definition,
\begin{equation}\label{eq:cyl}
\nu([j, r_1, \ldots, r_n]) =
\begin{cases}
c_- \eta_-^{-j} \beta_{r_1}^- \beta_{r_2}^+ \cdots \beta_{r_{n-1}}^-\beta_{r_n}^+ & \text{for } j < 0, \; n \text{ even}\\
c_- \eta_-^{-j} \beta_{r_1}^- \beta_{r_2}^+ \cdots \beta_{r_{n-2}}^-\beta_{r_{n-1}}^+\beta_{r_n}^- & \text{for } j < 0, \; n \text{ odd}\\
c_+ \eta_+^j \beta_{r_1}^+ \beta_{r_2}^- \cdots \beta_{r_{n-1}}^+\beta_{r_n}^- & \text{for } j > 0, \; n \text{ even}\\
c_+ \eta_+^j \beta_{r_1}^+ \beta_{r_2}^- \cdots \beta_{r_{n-2}}^+\beta_{r_{n-1}}^-\beta_{r_n}^+ & \text{for } j > 0, \; n \text{ odd}
\end{cases}.
\end{equation}
Now we prove that for some choice of the constants $c_\pm, \beta_1^\pm, \ldots, \beta_k^\pm > 0$ satisfying \eqref{eq:eq2} the measure $\nu$ is stationary for $\{\tilde f_-, \tilde f_+\}$ with probabilities $p_-, p_+$. Note that to show that $\nu$ is stationary, it is enough to check 
\begin{equation}\label{eq:stat}
\nu([j, r_1, \ldots, r_n]) = p_- \; \nu(\tilde f_-^{-1}([j, r_1, \ldots, r_n])) + p_+\; \nu(\tilde f_+^{-1}([j, r_1, \ldots , r_n]))
\end{equation}
for $j \in \Z^*$, even $n \in \N$ and $r_1, \ldots r_n \in \{1, \ldots, k\}$, because the corresponding cylinders $[j, r_1, \ldots, r_n]$ generate the $\sigma$-algebra of Borel sets in $\Z^* \times \Sigma_k^+$. 
By \eqref{eq:tildef1},
\begin{align*}
\tilde f_-^{-1}([j, r_1, \ldots, r_n]) &= 
\begin{cases}
[j+1, r_1, \ldots, r_n] & \text{for } j < -1\\
[r_1, r_2, \ldots, r_n] & \text{for } j = -1\\
[j+k, r_1, \ldots, r_n] & \text{for } j > 0
\end{cases},
\\
\tilde f_+^{-1}([j, r_1, \ldots, r_n]) &= 
\begin{cases}
[j-k, r_1, \ldots, r_n] & \text{for } j < 0\\
[-r_1, r_2,\ldots, r_n] & \text{for } j = 1\\
[j-1, r_1, \ldots, r_n] & \text{for } j > 1
\end{cases}.
\end{align*}
Using this together with \eqref{eq:cyl}, we check that \eqref{eq:stat} for even $n \in \N$ (split into four cases: $j < -1$, $j > 1$, $j = -1$, $j = 1$, respectively) is equivalent to the following system of equations:
\begin{equation}\label{eq:eq}
\begin{cases}
\eta_- = p_- + p_+ \eta_-^{k+1}\\
\eta_+ = p_+ + p_- \eta_+^{k+1}\\
c_- \eta_- \beta_r^- =
p_- c_+ \eta_+^r +
p_+ c_-  \eta_-^{k+1}\beta_r^-  & \text{for } r = 1, \ldots, k\\
c_+ \eta_+ \beta_r^+  =
p_+ c_-  \eta_-^r +
p_- c_+ \eta_+^{k+1} \beta_r^+ & \text{for } r = 1, \ldots, k
\end{cases}
\end{equation}
(where we write $r$ instead of $r_1$). 

Now we solve the system \eqref{eq:eq} together with \eqref{eq:eq2}.
The first two equations of \eqref{eq:eq} agree with the definitions of $\eta_-$, $\eta_+$. Substituting them, respectively, into the third and fourth ones, we obtain
\begin{equation}\label{eq:eq3}
c_- \beta_r^- = c_+ \eta_+^r, \qquad c_+ \beta_r^+ = c_- \eta_-^r.
\end{equation}
Summing this over $r \in \{1, \ldots, k\}$ and using the second and third equation of \eqref{eq:eq2}, we have
\[
c_- = c_+ \frac{\eta_+ - \eta_+^{k+1}}{1- \eta_+},\ c_+ = c_- \frac{\eta_- - \eta_-^{k+1}}{1- \eta_-}
\]
and substituting the second and first equation of \eqref{eq:eq} respectively, we arrive at a single equation
\[
c_-p_- = c_+p_+,
\]
which together with the first equation of \eqref{eq:eq2} gives
\[
c_- = \frac{p_+}{p_+\eta_-/(1-\eta_-) + p_-\eta_+/(1-\eta_+)}, \qquad
c_+ = \frac{p_-}{p_+\eta_-/(1-\eta_-) + p_-\eta_+/(1-\eta_+)}.
\]
Using \eqref{eq:eq3}, we finally obtain
\[
\beta_r^- = \frac{p_-}{p_+}\eta_+^r, \qquad \beta_r^+ = \frac{p_+}{p_-}\eta_-^r, \qquad r = 1, \ldots, k.
\]
The numbers $c_\pm, \beta_1^\pm, \ldots, \beta_k^\pm$ satisfy \eqref{eq:eq} and \eqref{eq:eq2}. In this way we showed that the system of equations \eqref{eq:eq} and \eqref{eq:eq2} has a unique solution for which the measure $\nu$ is stationary. By the uniqueness of such a measure, we have $\nu = \tilde \mu$.
\end{proof}

Finally, we determine the Hausdorff dimension of the measure $\mu$. Since $\mu|_{I_j} = \pi_* (\eta_j \delta_j \otimes \nu_j)$ for $j \in \Z^*$ by Proposition~\ref{prop:symbolic_measure}, we have
\[
\dim_H \mu = \sup_{j \in \Z^*} \dim_H \mu|_{I_j} = \sup_{j \in \Z^*} \dim_H \pi_*(\eta_j \delta_j \otimes \nu_j).
\]
Note that the measure $\pi_*(\eta_j \delta_j \otimes \nu_j)$, supported on the Cantor set $\Lambda_j$, is bi-Lipschitz isomorphic (after normalization) to the measure $\pi_*(\eta_{-1} \delta_{-1} \otimes \nu_{-1})$, which (after normalization) is the self-similar measure for the iterated function system $\{\phi_r \circ \phi_s\}_{r,s=1}^k$ with probabilities $(\beta_r^- \beta_s^+)_{r,s=1}^k$. It is well-known (see e.g.~\cite[Theorem 5.2.5]{edgar}) that the Hausdorff dimension of such a measure is equal to the ratio of the entropy of the measure and its Lyapunov exponent, i.e.
\begin{align*}
\dim_H \pi_*(\eta_j \delta_j \otimes \nu_j) &= \frac{\sum \limits _{r,s=1}^k \beta_r^-\beta_s^+ \log \beta_r^-\beta_s^+}{\sum \limits _{r,s=1}^k \beta_r^-\beta_s^+ \log \rho^{r+s}} = \frac{\sum \limits_{r=1}^{k} (\beta_r^-\log \beta_r^- + \beta_r^+\log \beta_r^+) }
{\sum \limits_{r=1}^{k} (\beta_r^- + \beta_r^+)\log  \rho^r}\\ 
&=  \frac{\sum \limits_{r=1}^k r\Big(\frac{p_+}{p_-} \eta_-^r\log \eta_-  + \frac{p_-}{p_+} \eta_+^r\log \eta_+\Big)}{\sum \limits_{r=1}^k r\Big( \frac{p_+}{p_-}  \eta_-^r + \frac{p_-}{p_+}\eta_+^r \Big)\log \rho}.
\end{align*}

\section{Proof of Theorem~\ref{thm:(k:l)}. Case $l > 1$}
\label{sec:l>1}

\subsection*{Preliminaries}

In Theorem~\ref{thm:(k:l)} we consider a symmetric AM-system $\{f_-, f_+\}$ of disjoint type  with probabilities $p_-, p_+$, positive Lyapunov exponents and a $(k:l)$-resonance for some relatively prime $k,l\in\N$, $k > l$.
In this section we deal with the case $l >1$. Our approach is similar to the case $l=1$, however the combinatorics of the obtained system of intervals is more complicated and produces Cantor sets which are attractors for infinite iterated function systems.

We have
\[
f_-(x)=
\begin{cases}
\rho^l x &\text{for }x\in[0,x_-]\\
\II(\rho^{-k}\II(x)) &\text{for }x\in (x_-, 1]
\end{cases}, \qquad
f_+(x)=
\begin{cases}
\rho^{-k} x &\text{for }x\in[0,x_+]\\
\II(\rho^l\II(x)) &\text{for }x\in (x_+, 1]
\end{cases},
\]
where $\rho \in (0,1)$, $k, l \in \N$, $1 < l < k$ and
\begin{alignat*}{2}
x_- &= \frac{1 - \rho^k}{1 - \rho^{k+l}}, &\qquad x_+ &= \II(x_-) = \frac{\rho^k - \rho^{k+l}}{1 - \rho^{k+l}},\\
f_-(x_-) &= \frac{\rho^l - \rho^{k+l}}{1 - \rho^{k+l}}, &\qquad f_+(x_+) &= \II(f_-(x_-)) = \frac{1 - \rho^l}{1 - \rho^{k+l}}.
\end{alignat*}
In particular, we have
\[
x_+ < f_-(x_-).
\]
A direct computation gives
\begin{equation}\label{eq:k,l}
\II(\rho^k \II (\rho^l x_-)) = x_-. 
\end{equation}
We assume that the system is of disjoint type, which is equivalent to 
\[
\rho^{k+l} - 2\rho^l + 1 > 0
\]
and also (by symmetry) to
\[
f_-(x_-) < \frac 1 2.
\]
Hence, since the system is symmetric, we have
\[
x_+ < f_-(x_-) < \frac 1 2 < f_+(x_+) < x_-.
\]

Consider the function $h(\rho) = \rho^{k+l} -2\rho^{k+1} + 2\rho -1$, $\rho \ge 0$. We have $h(0), h(1/2) < 0$, $h(1) = 0$, $h'(1) < 0$ and $h''$ has exactly one zero in $(0,+\infty)$. This implies that $h$ on $(0, 1)$ has a unique zero $\eta \in (1/2, 1)$, i.e. 
\[
\eta^{k+l} - 2\eta^{k+1} +2 \eta - 1 = 0
\]
and the assumption $\rho < \eta$ is equivalent to 
\begin{equation}\label{eq:ass2}
\rho^{k+l} - 2\rho^{k+1} +2\rho - 1 < 0
\end{equation}
and also to
\begin{equation}\label{eq:<1/2}
\rho x_- < \frac 1 2.
\end{equation}
In particular, this shows that the condition $\rho < \eta$ implies that the system is of disjoint type, which proves Remark~\ref{rem:l=1}.

Finally, notice that the positivity of the Lyapunov exponents of the system is equivalent to
\[
p_-, p_+ \in \Big( \frac{l}{k+l}, \frac{k}{k+l}\Big).
\]

\subsection*{Construction of the set \boldmath $\Lambda$}

Let us define the basic intervals $I_j \in \Z^*$ in the same manner as in the case $l=1$, i.e.
\[
I_{-1} = [\rho f_+ (x_+), \rho x_-] = [\rho f_+(x_+), \rho^{1-l}f_-(x_-)] =[\rho \II(\rho^l x_-), \rho x_-] = \left[ \frac{\rho-\rho^{1+l}}{1 - \rho^{k+l}}, \frac{\rho - \rho^{k+1}}{1 - \rho^{k+l}} \right]
\]
and note that by \eqref{eq:<1/2},
\begin{equation}\label{eq:1/2-kl}
\sup I_{-1} < \frac 1 2.
\end{equation}
For $j \in \Z^*$ let
\[
I_j = 
\begin{cases}
\rho^{-j-1} I_{-1} & \text{for } j < 0\\
\II(\rho^{j-1}I_{-1}) &\text{for } j > 0
\end{cases}.
\]
Let us now explain briefly the differences compared to the case $l=1$. Unlike previously, the union $\bigcup_{j \in \Z^*} I_j$ is no longer forward-invariant under $\{f_-, f_+\}$. More precisely, $f_+(I_{-k} \cup \ldots  \cup I_{-l}) \subset I_{l}$, but $f_+(I_{-l+1} \cup \ldots  \cup I_{-1})$ is situated between $I_l$ and $I_{l+1}$, inside a larger interval $J_l$ (see Lemma \ref{lem:f} and Figure \ref{fig:intervals-kl}). Therefore, our first step is extending the family $\{I_j\}_{j \in \Z^*}$ to a larger family $\{I_{\jj}\}_{\jj \in \JJ}$ consisting of similar copies of intervals $f_+(I_{-l+1}) , \ldots  , f_+(I_{-1})$ and their further iterates which are not contained in the intervals obtained in previous steps of the construction (see Figures~\ref{fig:intervals-kl2} and~\ref{fig:intervals-kl}). As a result, we obtain a forward-invariant family of intervals, which has infinitely many elements inside each of the (disjoint) intervals $J_j$. As before, we iterate the intervals from this family to produce a fully invariant and minimal union of disjoint Cantor sets. The corresponding iterated function system $\{\Phi_\rr\}_{\rr \in \RR}$ on $I_{-1}$ is generated by the action of $f_+$ on the interval $[x_+, f_-(x_-)]$, which maps some of the intervals $I_\jj$ into $I_l$. This infinite IFS has a Cantor attractor $\Lambda_{-1} \subset I_{-1}$, which is copied inside each of the intervals $I_\jj$ to form a suitable invariant minimal set $\Lambda \subset (0,1)$.

Let
\[
J_{-1} = [\rho\II(\rho x_-), \rho x_-]
\]
and note that 
\[
I_{-1} \subset J_{-1}, \qquad \sup I_{-1} = \sup J_{-1}.
\]
As previously, consider the maps
\[
\phi_r(x) = \rho \II(\rho^{r-1}x) = \rho - \rho^r x, \qquad r = 1,\ldots, k 
\]
for $x \in J_{-1}$. Recall that $\phi_r$ are orientation-reversing contracting similarities with $|\phi_r'| = \rho^r$ and $\phi_1 < \cdots < \phi_k$. 

\begin{lem}\label{lem:ifs} We have
\[
\phi_r(J_{-1}) \subset
\begin{cases}
J_{-1} \setminus I_{-1} &\text{for } r = 1, \ldots, l - 1\\
I_{-1} &\text{for } r = l, \ldots, k - 1
\end{cases},
\qquad
\phi_k(I_{-1}) \subset I_{-1}.
\]
Moreover, $\phi_r(J_{-1})$, $r = 1, \ldots, k$, are pairwise disjoint.
\end{lem}
\begin{proof}
By definition,
\[
\phi_r(J_{-1}) = [\rho\II(\rho^r x_-), \rho\II(\rho^r \II(\rho x_-))]\qquad
\phi_r(I_{-1}) = [\rho\II(\rho^r x_-), \rho\II(\rho^r \II(\rho^l x_-))]. 
\]
It is obvious that $\inf\phi_r(J_{-1}) \geq \inf J_{-1}$ for $r  = 1, \ldots, k$ and $\inf\phi_r(J_{-1}) \geq \inf I_{-1}$ for $r  = l, \ldots, k$. The inequality $\sup\phi_r(J_{-1}) < \inf I_{-1}$ for $r = 1, \ldots, l-1$ boils down to \eqref{eq:ass2}, while $\sup\phi_r(J_{-1}) \leq \sup I_{-1}$ for $r = l, \ldots, k-1$ is equivalent to $\rho^{k+l} + \rho^{k-r} + \rho - \rho^{k+1} - \rho^{k+l-r}+1 \leq 0$. For $l \leq r \leq k-1$ it is enough to have $\rho^{k+l}+2\rho-\rho^{k+1} - \rho^k - 1 \leq 0$ (as $\rho^{k-r} \leq \rho$ and $\rho^{k+l-r} \geq \rho^k$). By \eqref{eq:ass2} this can be reduced to $\rho^{k+1} - \rho^k \leq 0$ which is obviously true, since $\rho \in (0, 1)$. This proves the first assertion. To show $\phi_k(I_{-1}) \subset I_{-1}$, it is enough to notice that $\sup\phi_k(I_{-1}) = \sup I_{-1}$ holds due to \eqref{eq:k,l}. To check the disjointness of $\phi_r(J_{-1})$, we notice that the inequality $\sup \phi_r(J_{-1}) < \inf \phi_{r+1}(J_{-1})$, $r  = 1, \ldots, k-1$, is equivalent to \eqref{eq:ass2}. 
\end{proof}

For $j \in \Z^*$ let
\[
\JJ_j = \{j\} \times \Big( \{\emptyset\} \cup \bigcup_{n = 1}^\infty \{1, \ldots, l-1\}^n\Big), \qquad \JJ = \bigcup_{j \in \Z^*} \JJ_j. 
\]
We will denote the elements of $\JJ$ by $\jj = (j, j_1, \ldots, j_n)$, where  $j \in \Z^*$, $n \ge 0$, $j_1, \ldots, j_n \in \{1, \ldots, l-1\}$, with the convention that $j_1, \ldots, j_n$ for $n = 0$ is the empty sequence. 

For $\jj  = (j, j_1, \ldots, j_n) \in \JJ$ define
\[
I_\jj = I_{j,j_1, \ldots, j_n} = 
\begin{cases}
\rho^{-j-1} \phi_{j_1} \circ \cdots \circ \phi_{j_n}(I_{-1}) &\text{for } j < 0\\
\II(\rho^{j-1} \phi_{j_1} \circ \cdots\circ \phi_{j_n}(I_{-1})) &\text{for } j > 0
\end{cases},
\qquad I = \bigcup_{\jj \in \JJ} I_\jj.
\]
Note that this notation is compatible with our previous definition of $I_j$ for $j \in \Z^*$.
Furthermore, for $j \in \Z^*$ let
\[
J_j = 
\begin{cases}
\rho^{-j-1} J_{-1} & \text{for } j < 0\\
\II(\rho^{j-1} J_{-1}) & \text{for } j > 0
\end{cases}, \qquad
J = \bigcup_{j\in\Z} J_j.
\]
The following lemmas describe the combinatorics of the intervals $I_\jj, \jj \in \JJ$.
\begin{lem}\label{lem:intervals-kl}
The following statements hold.
\begin{itemize}
\item[$($a$)$] $I_{-j,j_1, \ldots, j_n} = \II(I_{j,j_1, \ldots, j_n})$, $J_{-j} = \II(J_j)$ for $j \in \Z^*$, $n \ge 0$, $j_1, \ldots, j_n \in \{1, \ldots, l-1\}$.
\item[$($b$)$] The segments $J_j$, $j\in\Z^*$, are pairwise disjoint. 
\item[$($c$)$] For $j \in \Z^*$, the segments $I_\jj$, $\jj \in \JJ_j$, are pairwise disjoint subsets of $J_j$. 
\item[$($d$)$] For $j \in \Z^*$, we have $\inf J_j = \inf I_{j,1}$, $\sup J_j = \sup I_j$ for $j < 0$ and
$\inf J_j = \inf I_j$, $\sup J_j = \sup I_{j,1}$ for $j > 0$. In particular, 
\[
J_j = \conv \bigcup_{\jj \in \JJ_j} I_\jj.
\]
\item[$($e$)$] Let $j \in \Z^*$. Then for $j < 0$ $($resp.~$j > 0)$, the segments $I_{j,j_1}$, $j_1 = 1, \ldots, l-1$, are situated in $J_j$ in the increasing $($resp.~decreasing$)$ order with respect to $j_1$, to the left $($resp.~right$)$ of $I_j$.

\item[$($f$)$]
Let $j \in \Z^*$, $j_1, \ldots, j_n \in \{1, \ldots, l-1\}$ for $n\ge 1$. Then for $j < 0$ and even $n$ or $j > 0$ and odd $n$ $($resp.~$j < 0$ and odd $n$ or $j > 0$ and even $n)$, the segments 
$I_{j,j_1, \ldots, j_{n+1}}$, $j_{n+1} = 1, \ldots, l-1$ are situated in $J_j$ in the increasing $($resp.~decreasing$)$ order with respect to $j_{n+1}$, between $I_{j,j_1, \ldots, j_n}$ and $I_{j,j_1, \ldots, j_{n-1}, j_n+1}$ if $j_n < l-1$, and between $I_{j,j_1, \ldots, j_n}$ and $I_{j,j_1, \ldots, j_{n-1}}$ if $j_n = l-1$. 
\item[$($g$)$] $\inf I_{-k} = x_+$, $\sup I_{-l} = f_-(x_-)$, $\inf I_l = f_+(x_+)$, $\sup I_k = x_-$.
\end{itemize}
See Figures~{\rm \ref{fig:intervals-kl2}} and~{\rm \ref{fig:intervals-kl}}.
\end{lem}

\begin{figure}[ht!]
\begin{center}
\scalebox{0.85}{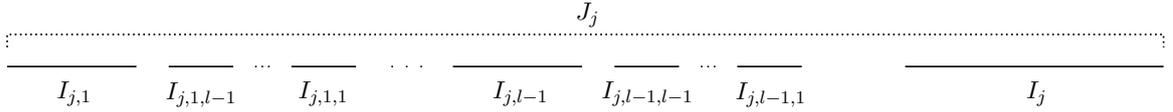}
\end{center}
\caption{A schematic view of the location of the intervals $I_{j,j_1, \ldots, j_n}$ within $J_j$ for $j < 0$.}\label{fig:intervals-kl2}
\end{figure}

\begin{proof} The assertion (a) is straightforward.
To show (b), it is enough to use \eqref{eq:1/2-kl} and check $\sup J_{j-1} < \inf J_j$ for $j < 0$ (and use the symmetry of the system). By a direct computation, the latter inequality is equivalent to \eqref{eq:ass2}. By symmetry and the definition of $I_\jj$ and $J_j$, showing (c)--(f) we can assume $j = -1$. First, we prove (c). 
Since $I_{-1} \subset J_{-1}$, Lemma~\ref{lem:ifs} implies $I_\jj \subset J_{-1}$ for  $\jj \in \JJ_{-1}$. To show the disjointness of $I_\jj$, suppose that $I_{-1,j_1, \ldots, j_n} \cap 
I_{-1,j_1', \ldots, j'_{n'}} \neq \emptyset$ for some distinct $(-1,j_1, \ldots, j_n), (-1,j_1', \ldots, j'_{n'}) \in \JJ_{-1}$. We can assume $n' \geq n$. Applying suitable sequence of inverses of maps $\phi_r$ to both segments, we can suppose $j_1 \neq j'_1$ or $I_{-1,j_1', \ldots, j'_{n'}} = I_{-1}$. In the first case we have a contradiction with the last assertion of Lemma~\ref{lem:ifs}, while the second case contradicts with the first assertion of it. This proves (c). The first part of (d) is straightforward. Together with (c), it shows the second part. The assertion (e) follows from (c) and the fact $\phi_1 < \cdots < \phi_{l-1}$. The first part of (f) holds by a direct checking. In turn, together with the fact that the maps $\phi_r$ reverse the orientation and $\phi_1 < \cdots < \phi_{l-1}$, it proves the second part by induction. The assertion (g) is straightforward. 
\end{proof}

The following lemma is a direct consequence of the definition of the maps $f_\pm$ and Lemma~\ref{lem:intervals-kl}. See Figure~{\rm \ref{fig:intervals-kl}}.

\begin{lem}\label{lem:f}
We have
\begin{align*}
f_-(x) &= 
\begin{cases}
\rho^l x & \text{for } x \in I_k \cup \bigcup_{j = -\infty}^{k-1} J_j\\
\II(\rho^{-k} \II(x)) & \text{for } x \in \bigcup_{j = k}^\infty J_j \setminus I_k
\end{cases},\\
f_+(x) &= 
\begin{cases}
\rho^{-k} x & \text{for } x \in \bigcup_{j = -\infty}^{-k} J_j\setminus I_{-k}\\
\II(\rho^l \II(x)) & \text{for } x \in I_{-k} \cup \bigcup_{j = -k+1}^\infty J_j
\end{cases}.
\end{align*}
Moreover, for $(j, j_1, \ldots, j_n) \in \JJ$, we have:
\begin{alignat*}{2}
f_+(I_{j,j_1, \ldots, j_n}) &= I_{j+k,j_1, \ldots, j_n} &\quad &\text{for } j < -k,\\
f_+(I_{-k,j_1, \ldots, j_n}) &= I_{j_1, j_2, \ldots, j_n} &\quad &\text{for } n > 0,\\
f_+\Big(\conv\Big(I_{-k} \cup \bigcup_{j=-k+1}^{-l} J_j \Big)\Big) &= I_l,&&\\
f_+(I_{j,j_1, \ldots, j_n}) &= I_{l, -j, j_1 \ldots, j_n} &\quad &\text{for } -l+1 \leq j \leq -1,\\
f_+(I_{j,j_1, \ldots, j_n}) &= I_{j+l,j_1, \ldots, j_n} &\quad &\text{for } j > 0.
\end{alignat*}

Analogously,
\begin{alignat*}{2}
f_-(I_{j,j_1, \ldots, j_n}) &= I_{j-l,j_1, \ldots, j_n} &\quad &\text{for } j < 0,\\
f_-(I_{j,j_1, \ldots, j_n}) &= I_{-l, j, j_1 \ldots, j_n} &\quad &\text{for } 1 \leq j \leq l-1,\\
f_-\Big(\conv\Big(I_k \cup \bigcup_{j=l}^{k-1} J_j\Big) \Big) &= I_{-l},&&\\
f_-(I_{k,j_1, \ldots, j_n}) &= I_{-j_1, j_2, \ldots, j_n} &\quad &\text{for } n > 0,\\
f_-(I_{j,j_1, \ldots, j_n}) &= I_{j-k,j_1, \ldots, j_n} &\quad &\text{for } j > k.
\end{alignat*}
\end{lem}

\begin{figure}[ht!]
\begin{center}
\scalebox{0.9}{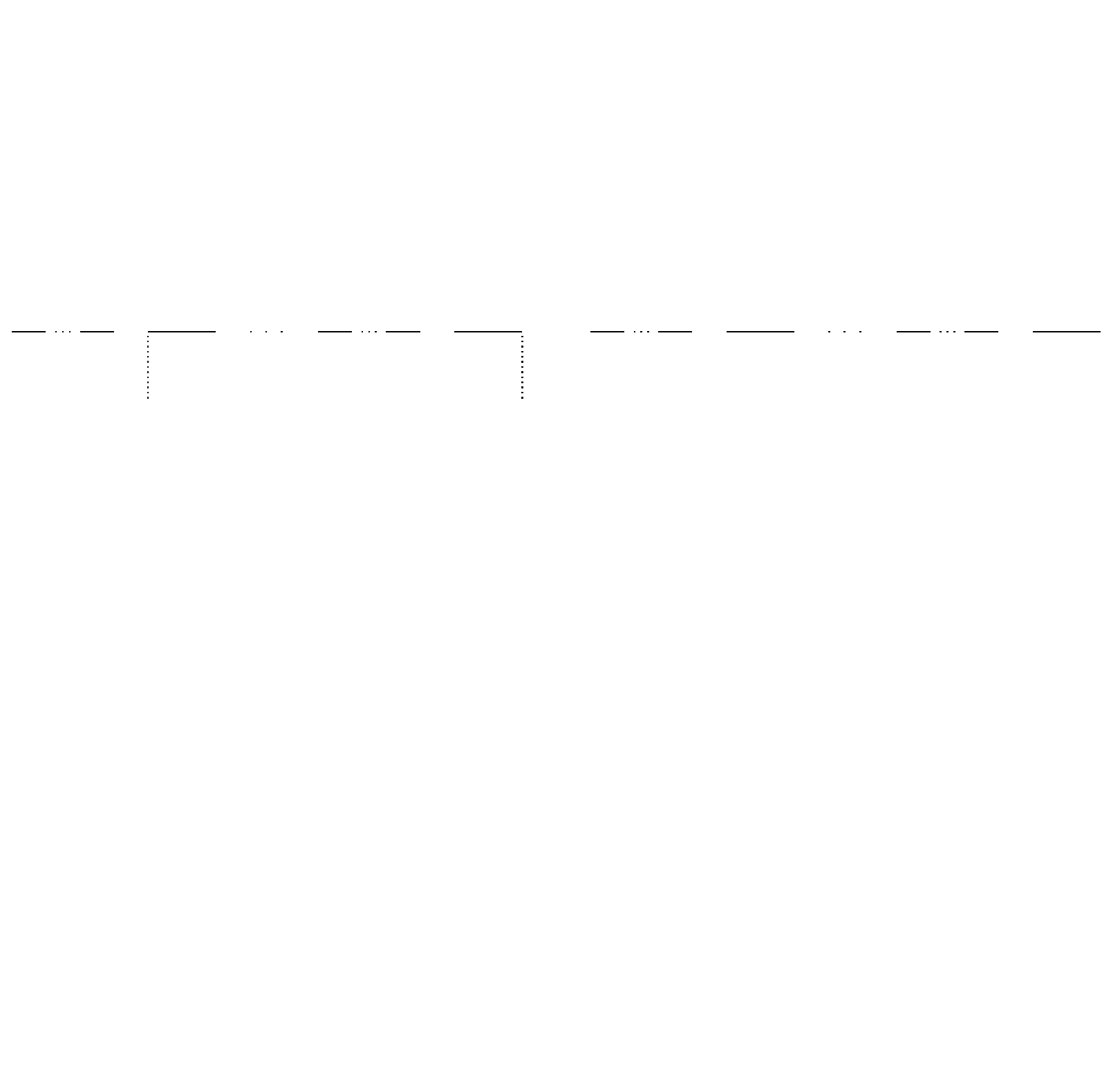}
\end{center}
\caption{A schematic view of the action of $f_+$ on the intervals $I_\jj$.}\label{fig:intervals-kl}
\end{figure}

In particular, Lemma~\ref{lem:f} implies $f(I) \subset I$. More precisely, for every $i \in \{-,+\}$ and $\jj \in \JJ$, 
\[
f_i(I_\jj) \subset I_{\jj'} \qquad\text{for some } \jj' = \jj'(i,\jj) \in \JJ. 
\]

Let
\[
\RR = \{l, \ldots, k\} \cup \bigcup_{n =1}^\infty \{l, \ldots, k-1\}\times\{1, \ldots, l-1\}^n.
\]
We will denote the elements of $\RR$ by $\rr = (r, r_1, \ldots, r_n)$, $n \ge 0$, where $r \in \{l, \ldots, k\}$ in the case $n = 0$, $r \in \{l, \ldots, k-1\}$ in the case $n > 0$ and $r_1, \ldots, r_n \in \{1, \ldots, l-1\}$,  with the convention that $r_1, \ldots, r_n$ for $n = 0$ is the empty sequence. Note that 
\[
\RR \subset \JJ.
\]
For $\rr = (r, r_1, \ldots, r_n) \in \RR$ define the maps
\[
\Phi_\rr = 
\begin{cases}
\phi_r &\text{for } n = 0\\
\phi_r \circ \phi_{r_1} \circ \cdots \circ\phi_{r_n} &\text{for } n > 0
\end{cases}
\]
on the interval $I_{-1}$. By Lemma~\ref{lem:ifs}, 
\[
\Phi_\rr \colon I_{-1} \to I_{-1}, \qquad \rr \in \RR,
\]
so the family $\{\Phi_\rr\}_{\rr \in \RR}$ is a countable infinite iterated function system of contractions in $I_{-1}$ satisfying $\lim \limits_{s \to \infty} |\phi_{\rr_s} (I_{-1})| = 0$ for any sequence $(\rr_s)_{s=1}^\infty$ of mutually distinct elements of $\RR$. Moreover, the definition of $\Phi_\rr$ implies
\[
\{\Phi_\rr(I_{-1})\}_{\rr \in \RR} = \{\phi_r(I_\jj) : r \in\{l, \ldots, k-1\}, \jj\in\JJ_{-1}\} \cup \{\phi_k(I_{-1})\}.
\]
This together with Lemma~\ref{lem:ifs} implies that $\Phi_\rr(I_{-1})$, $\rr \in \RR$, are pairwise disjoint. Similarly as before, we are interested in the limit set of this system. As the family $\{\Phi_\rr\}_{\rr \in \RR}$  is infinite, there are two limit sets one can consider:
\begin{equation}\label{eq:L_def} L = \bigcap \limits_{m=1}^{\infty} \bigcup_{\rr_1, \ldots, \rr_m \in \RR} \Phi_{\rr_1} \circ \cdots \circ \Phi_{\rr_m} (I_{-1})
\end{equation}
and its closure
\[ \Lambda_{-1} = \overline{L}.\]
It is easy to see that they satisfy 
\[
L = \bigcup_{\rr\in \RR}\Phi_\rr(L), \qquad \Lambda_{-1} = \overline{\bigcup_{\rr\in \RR}\Phi_\rr(\Lambda_{-1})} = \bigcap_{m = 1}^\infty  \overline{\bigcup_{\rr_1, \ldots, \rr_m \in \RR} \Phi_{\rr_1} \circ \cdots \circ \Phi_{\rr_m} (I_{-1})}
\]
(see e.g.~\cite[Section 2]{MU}). As our goal is to find the minimal attractor of the system $\{f_-, f_+ \}$ (which equals also the support of $\mu$), we will focus on the $\Lambda_{-1}$. However, we will use the set $L$ in the proof of Proposition \ref{prop:dimX-kl}, as it is better suited for calculating the Hausdorff dimension.

For $\jj = (j,j_1, \ldots, j_n) \in \JJ$ let
\[
\Lambda_\jj = \Lambda_{j,j_1, \ldots, j_n} = 
\begin{cases}
\rho^{-j-1} \phi_{j_1, \ldots, j_n}(\Lambda_{-1}) &\text{for } j < 0\\
\II(\rho^{j-1} \phi_{j_1, \ldots, j_n}(\Lambda_{-1})) &\text{for } j > 0
\end{cases}, \qquad \Lambda = \overline{\bigcup_{\jj\in \JJ} \Lambda_\jj} \cap (0,1),
\]
where we write
\[
\phi_{j_1, \ldots, j_n} = \phi_{j_1} \circ \cdots \circ \phi_{j_n}.
\]
Obviously, $\Lambda_\jj \subset I_\jj$ for $j \in \JJ$ and $\Lambda \subset \bigcup_{j \in \Z*} J_j$. Furthermore, for $m \ge 0$ and $\rr_1, \ldots, \rr_m \in \RR$ let
\[
I_{\jj; \rr_1, \ldots, \rr_m} = 
\begin{cases}
\rho^{-j-1} \phi_{j_1, \ldots, j_n}(\Phi_{\rr_1} \circ \cdots \circ \Phi_{\rr_m} (I_{-1})) &\text{for } j < 0\\
\II(\rho^{j-1} \phi_{j_1, \ldots, j_n}(\Phi_{\rr_1} \circ \cdots \circ \Phi_{\rr_m} (I_{-1}))) &\text{for } j > 0
\end{cases}
\]
(for $m = 0$ the set $I_{\jj; \rr_1, \ldots, \rr_m}$ is equal to $I_\jj$).
As $|\Phi_\rr'| \le \rho$, for $\jj \in \JJ$, $\rr_1, \rr_2, \ldots \in \RR$ we have
\[
|I_{\jj; \rr_1, \ldots, \rr_m}| \le \rho^m \to 0 \text{ as } m \to \infty,
\]
so
\[
\bigcap_{m=1}^\infty I_{\jj; \rr_1, \ldots, \rr_m} = \{x_{\jj; \rr_1, \rr_2, \ldots}\}
\]
for a point $x_{\jj; \rr_1, \rr_2, \ldots} \in \Lambda$ and
\[
\Lambda = \overline{\bigcup_{\jj \in \JJ} \bigcap_{m=1}^\infty \overline{\bigcup_{\rr_1, \ldots, \rr_m \in \RR} I_{\jj; \rr_1, \ldots, \rr_m}}} \cap (0,1)= 
\overline{\left\{ x_{\jj; \rr_1, \rr_2, \ldots}: \: \jj \in \JJ, \rr_1, \rr_2, \ldots \in \RR\right\}}\cap (0,1).
\]

\subsection*{Description of trajectories}

Lemma~\ref{lem:f} implies the following.

\begin{lem}\label{lem:f(I)-kl} For $(j, j_1, \ldots, j_n) \in \JJ$,  $\rr_1, \rr_2, \ldots, \in \RR$ and $m \ge 0$, we have:
\[
f_-(I_{(j, j_1, \ldots, j_n); \rr_1, \ldots, \rr_m}) = 
\begin{cases}
I_{(j-l,j_1, \ldots, j_n); \rr_1, \ldots, \rr_m } &\text{for } j < 0\\
I_{(-l,j,j_1, \ldots, j_n); \rr_1, \ldots, \rr_m } &\text{for } 1 \leq j \leq l-1\\
I_{-l; (j, j_1, \ldots, j_n), \rr_1, \ldots, \rr_m } &\text{for } l \leq j \leq k-1 \text{ or } j = k,\, n = 0 \\
I_{(-j_1, j_2, \ldots, j_n); \rr_1, \ldots, \rr_m } &\text{for } j = k, \, n > 0\\
I_{(j-k,j_1, \ldots, j_n); \rr_1, \ldots, \rr_m } &\text{for } j > k
\end{cases},
\]
\[
f_+(I_{(j, j_1, \ldots, j_n); \rr_1, \ldots, \rr_m}) = 
\begin{cases}
I_{(j+k,j_1, \ldots, j_n); \rr_1, \ldots, \rr_m } &\text{for } j < -k\\
I_{(j_1, j_2, \ldots, j_n); \rr_1, \ldots, \rr_m } &\text{for } j = -k, \, n > 0\\
I_{l; (-j, j_1, \ldots, j_n), \rr_1, \ldots, \rr_m} &\text{for }j = -k, \, n = 0 \text{ or } -k+1 \leq j \leq -l\\
I_{(l,-j,j_1, \ldots, j_n); \rr_1, \ldots, \rr_m }&\text{for } -l+1 \leq j \leq -1\\
I_{(j+l,j_1, \ldots, j_n); \rr_1, \ldots, \rr_m } &\text{for } j > 0
\end{cases},
\]
and
\[
f_-(x_{(j, j_1, \ldots, j_n); \rr_1, \rr_2, \ldots}) = 
\begin{cases}
x_{(j-l,j_1, \ldots, j_n); \rr_1, \rr_2, \ldots} &\text{for } j < 0\\
x_{(-l,j,j_1, \ldots, j_n); \rr_1, \rr_2, \ldots} &\text{for } 1 \leq j \leq l-1\\
x_{-l; (j, j_1, \ldots, j_n), \rr_1, \rr_2, \ldots} &\text{for } l \leq j \leq k-1\text{ or } j = k, \, n = 0 \\
x_{(-j_1, j_2, \ldots, j_n); \rr_1, \rr_2, \ldots} &\text{for } j = k, \, n > 0\\
x_{(j-k,j_1, \ldots, j_n); \rr_1, \rr_2, \ldots} &\text{for } j > k
\end{cases},
\]
\[
f_+(x_{(j, j_1, \ldots, j_n); \rr_1, \rr_2, \ldots}) = 
\begin{cases}
x_{(j+k,j_1, \ldots, j_n); \rr_1, \rr_2, \ldots} &\text{for } j < -k\\
x_{(j_1, j_2, \ldots, j_n); \rr_1, \rr_2, \ldots} &\text{for } j = -k, \, n > 0\\
x_{l; (-j, j_1, \ldots, j_n), \rr_1, \rr_2, \ldots} &\text{for } j = -k, \, n = 0 \text{ or } -k+1 \leq j \leq -l\\
x_{(l,-j,j_1, \ldots, j_n); \rr_1, \rr_2, \ldots}&\text{for } -l+1 \leq j \leq -1\\
x_{(j+l,j_1, \ldots, j_n); \rr_1, \rr_2, \ldots} &\text{for } j > 0
\end{cases}.
\]
\end{lem}

The next lemma follows directly from Lemma~\ref{lem:f}.

\begin{lem}\label{lem:jump-kl} The following statements hold.
\begin{itemize}

\item[(a)] If a trajectory $\{f_{i_N}\circ \cdots \circ f_{i_1}(x)\}_{N = 0}^\infty$, for $x \in (0,1)$, jumps over the central interval at the time $s$, for $s \ge 0$, then $f_{i_{s+1}} \circ \cdots \circ f_{i_1}(x) \in I_{-l} \cup I_l$.
\item[(b)] A trajectory $\{f_{i_N}\circ \cdots \circ f_{i_1}(x)\}_{N = 0}^\infty$, for $x \in J$, jumps over the central interval at the time $s$, for $s \ge 0$, if and only if
\begin{alignat*}{4}
&f_{i_s} \circ \cdots \circ f_{i_1}(x) \in I_{-k} \cup \bigcup_{j=-k+1}^{-l} J_j, \qquad &i_{s+1} &= +, \qquad
&f_{i_{s+1}} \circ \cdots \circ f_{i_1}(x) &\in I_l\\
&\text{or}\\
&f_{i_s} \circ \cdots \circ f_{i_1}(x) \in I_k \cup \bigcup_{j=l}^{k-1} J_j, \qquad &i_{s+1} &= -,\qquad 
&f_{i_{s+1}} \circ \cdots \circ f_{i_1}(x) &\in I_{-l}.
\end{alignat*}
\end{itemize}
In particular, for given $\jj \in \JJ$ and $i_1, i_2, \ldots \in \{-, +\}$, for all $x \in I_\jj$ the trajectories $\{f_{i_N}\circ \cdots \circ f_{i_1}(x)\}_{N = 0}^\infty$ jump over the central interval at the same times.
\end{lem}

Since $k,l$ are relatively prime, there exist $N_1, N_2 > 0$ such that $N_1 l - N_2 k  = 1$. Let
\[
F_- = f_+^{N_2} \circ f_-^{N_1}, \qquad F_+ = f_-^{N_2} \circ f_+^{N_1}.
\]
Then
\begin{alignat*}{3}
F_-(J_j) &= J_{j-1}, &\qquad F_-(x) &= \rho x &\qquad\text{for }x \in J_j, \; j < 0,\\
F_+(J_j) &= J_{j+1}, &\qquad F_+(x) &= \II\rho\II( x) &\qquad\text{for }x \in J_j, \; j > 0.
\end{alignat*}
For $j, j' \in \Z^*$ such that $\sgn(j) = \sgn(j')$, define
\[
F_{j,j'}\colon J_j \onto J_{j'}, \qquad
F_{j, j'} =
\begin{cases}
F_-^{j-j'}|_{J_j} &\text{for }j < 0, \;  j' \leq j\\
f_+^{\lceil (j'-j)/k\rceil} \circ F_-^{j-j'-k\lceil (j'-j)/k\rceil}|_{J_j} &\text{for } j > 0, \; j' > j\\
F_+^{j'-j}|_{J_j} &\text{for } j < 0, \;  j' \geq j\\
f_-^{\lceil (j-j')/k\rceil} \circ F_+^{j'-j-k\lceil (j-j')/k\rceil}|_{J_j} &\text{for } j > 0, \; j' < j
\end{cases}.
\]
We have $F_{j,j'} = f_{i_N}\circ \cdots \circ f_{i_1}|_{J_j}$ for some $i_1, \ldots, i_N \in \{-,+\}$, $N \ge 0$, and, by Lemma~\ref{lem:f}, 
\begin{equation}\label{eq:F-phi-kl}
F_{j, j'}(x) =
\begin{cases}
\rho^{j-j'}x  &\text{for } j < 0\\
\II(\rho^{-j+j'}\II(x)) &\text{for } j > 0
\end{cases}
\end{equation}
for $x \in J_j$. In particular, this implies
\[
F_{j, j} = \id|_{J_j}, \qquad F_{j', j''}  \circ  F_{j, j'} = F_{j, j''}
\]
for $j, j', j'' \in \Z^*$ such that $\sgn(j) = \sgn(j') =\sgn(j'')$. By Lemma~\ref{lem:f(I)-kl}, 
\begin{equation}\label{eq:F-kl}
\begin{aligned}
F_{j, j'}(I_{(j, j_1, \ldots, j_n); \rr_1, \ldots, \rr_m}) &= I_{(j', j_1, \ldots, j_n); \rr_1, \ldots, \rr_m},\\ 
F_{j, j'}(x_{(j, j_1, \ldots, j_n); \rr_1, \rr_2, \ldots}) &= x_{(j', j_1, \ldots, j_n); \rr_1, \rr_2, \ldots}
\end{aligned}
\end{equation}
for $j_1, \ldots, j_n \in \{1, \ldots, l-1\}$, $\rr_1, \rr_2, \ldots \in \RR$, $n,m \ge 0$.

\begin{lem}\label{lem:inI-kl} For every $x \in (0,1)$ there exists $i_1, \ldots, i_n \in \{-,+\}$, $n \ge 0$, such that $f_{i_n} \circ \cdots \circ f_{i_1}(x) \in I$. 
\end{lem}
\begin{proof} If $x \in J_j$ for $j < 0$ (resp.~$j > 0$), then it is enough to notice that by Lemma~\ref{lem:f}, we have $f_+ \circ F_{j,-l}(x) \in I_l$ (resp.~$f_- \circ F_{j,l}(x) \in I_{-l}$). Suppose $x \in (0,1) \setminus J$. Enumerate the components of $(0,1) \setminus J$ by $U_j$, $j \in \Z$, such that $U_j$ is the gap between $J_{j-1}$ and $J_j$ for $j < 0$, $U_0$ is the gap between $J_{-1}$ and $J_1$, and $U_j$ is the gap between $J_j$ and $J_{j+1}$ for $j > 0$. Since the system is symmetric, we can assume $x \in U_j$, $j \le 0$. Then, by Lemma~\ref{lem:f}, we have $f_-\circ f_+^{\lfloor \frac{-j}{k} \rfloor + 1}(x) \in I_{-l}$.
\end{proof}

Define 
\[
G^-_j\colon J_1 \to J_{-1}, \qquad G^+_j\colon J_{-1} \to J_1, \qquad j \in \{1, \dots, l-1\},
\]
by
\[
G^-_j = F_{-l,-1} \circ f_- \circ F_{1,j}, \qquad G^+_j = F_{l,1} \circ f_+ \circ F_{-1,-j}.
\]
Note that $G^\pm_j = f_{i_N}\circ \cdots \circ f_{i_1}|_{J_{\mp 1}}$ for some  $i_1, \ldots, i_N \in \{-,+\}$, $N \ge 0$. By \eqref{eq:F-phi-kl}, we have
\begin{equation}\label{eq:G-phi-kl}
G^-_j = \phi_j \circ \II|_{J_1}, \qquad G^+_j = \II \circ \phi_j 
\end{equation}
and by Lemma~\ref{lem:f(I)-kl} and \eqref{eq:F-kl},
\begin{equation}\label{eq:G-kl}
\begin{aligned}
G^-_j(I_{(1, j_1, \ldots, j_n); \rr_1, \ldots, \rr_m}) &= I_{(-1, j, j_1, j_2, \ldots, j_n); \rr_1, \ldots, \rr_m},\\
G^+_j(I_{(-1, j_1, \ldots, j_n); \rr_1, \ldots, \rr_m}) &= I_{(1, j, j_1, j_2, \ldots, j_n); \rr_1, \ldots, \rr_m},\\
G^-_j(x_{(1, j_1, \ldots, j_n); \rr_1, \rr_2, \ldots}) &= x_{(-1, j, j_1, j_2, \ldots, j_n); \rr_1, \rr_2, \ldots},\\
G^+_j(x_{(-1, j_1, \ldots, j_n); \rr_1, \rr_2, \ldots}) &= x_{(1, j, j_1, j_2, \ldots, j_n); \rr_1, \rr_2, \ldots}
\end{aligned}
\end{equation}
for $j_1, \ldots, j_n \in \{1, \ldots, l-1\}$, $\rr_1, \rr_2, \ldots \in \RR$, $n,m \ge 0$.

Define also
\[
H_-\colon \bigcup_{\{(1, j_1, \ldots, j_n) \in \JJ: n >0\}} I_{1, j_1, \ldots, j_n}\to J_{-1}, \qquad 
H_+\colon \bigcup_{\{(-1, j_1, \ldots, j_n) \in \JJ: n >0\}}I_{-1, j_1, \ldots, j_n} \to J_1
\]
by
\[
H_-|_{I_{1, j_1, \ldots, j_n}} = F_{-j_1,-1} \circ f_- \circ F_{1,k}|_{I_{1, j_1, \ldots, j_n}}, \qquad
H_+|_{I_{-1, j_1, \ldots, j_n}} = F_{j_1,1} \circ f_+ \circ F_{-1,-k}|_{I_{-1, j_1, \ldots, j_n}}.
\]
Again, $H_\pm = f_{i_N}\circ \cdots \circ f_{i_1}|_{I_{\mp 1, j_1, \ldots, j_n}}$ for some $i_1, \ldots, i_N \in \{-,+\}$, $N \ge 0$. By \eqref{eq:F-phi-kl} and \eqref{eq:G-phi-kl},
\begin{equation}\label{eq:H-phi-kl}
\begin{aligned}
H_-|_{I_{1, j_1, \ldots, j_n}} &= (G^+_{j_1})^{-1}|_{I_{1, j_1, \ldots, j_n}} = \phi_{j_1}^{-1}\circ \II|_{I_{1, j_1, \ldots, j_n}},\\
H_+|_{I_{-1, j_1, \ldots, j_n}}  &= (G^-_{j_1})^{-1}|_{I_{-1, j_1, \ldots, j_n}} = \II \circ \phi_{j_1}^{-1}|_{I_{1, j_1, \ldots, j_n}}
\end{aligned}
\end{equation}
while Lemma~\ref{lem:f(I)-kl} and \eqref{eq:F-kl} give
\begin{equation}\label{eq:H-kl}
\begin{aligned}
H_-(I_{(1, j_1, \ldots, j_n); \rr_1, \ldots, \rr_m}) &= I_{(-1, j_2, \ldots, j_n); \rr_1, \ldots, \rr_m},\\
H_+(I_{(-1, j_1, \ldots, j_n); \rr_1, \ldots, \rr_m}) &= I_{(1, j_2, \ldots, j_n); \rr_1, \ldots, \rr_m},\\
H_-(x_{(1, j_1, \ldots, j_n); \rr_1, \rr_2, \ldots}) &= x_{(-1, j_2, \ldots, j_n); \rr_1, \rr_2, \ldots},\\
H_+(x_{(-1, j_1, \ldots, j_n); \rr_1, \rr_2, \ldots}) &= x_{(1, j_2, \ldots, j_n); \rr_1, \rr_2, \ldots}
\end{aligned}
\end{equation}
for $j, j_1, \ldots, j_n \in \{1, \ldots, l-1\}$, $n > 0$, $\rr_1, \rr_2, \ldots \in \RR$, $m \ge 0$.

We introduce the following notation. For $\jj = (j, j_1, \ldots, j_n)\in \JJ$ we set $\jj < 0$ (resp.~$\jj > 0$) if $j < 0$ (resp.~$j > 0$). We also write $-\jj = (-j, j_1, \ldots, j_n)$ and set $\sgn(\jj) = \sgn(j)$, $n(\jj) = n$. 

For $\jj =  (j, j_1, \ldots, j_n), \jj' =  (j', j'_1, \ldots, j'_{n'})\in \JJ$, such that $\sgn(\jj) = \sgn(\jj')$ and $n(\jj)-n(\jj')$ is even, or $\sgn(\jj) \neq \sgn(\jj')$ and $n(\jj)-n(\jj')$ is odd, define
\[
F_{\jj,\jj'}\colon I_\jj \onto I_{\jj'}
\]
by
\[
F_{\jj, \jj'} = 
\begin{cases}
F_{-1,j'} \circ G^-_{j_1'} \circ G^+_{j_2'} \circ \cdots \circ G^-_{j_{n'-1}'} \circ G^+_{j'_{n'}} \circ (H_- \circ H_+)^{n/2} \circ F_{j,-1}\\
F_{1,j'} \circ G^+_{j_1'} \circ G^-_{j_2'} \circ \cdots \circ G^+_{j_{n'-2}'}\circ G^-_{j_{n'-1}'} \circ G^+_{j'_{n'}} \circ (H_- \circ H_+)^{n/2} \circ F_{j,-1}\\
F_{1,j'} \circ G^+_{j_1'} \circ G^-_{j_2'} \circ \cdots \circ G^+_{j_{n'-1}'} \circ G^-_{j'_{n'}} \circ H_+ \circ (H_- \circ H_+)^{\lfloor n/2\rfloor} \circ F_{j,-1}\\
F_{-1,j'} \circ G^-_{j_1'} \circ G^+_{j_2'} \circ \cdots \circ G^-_{j_{n'-2}'}\circ G^+_{j_{n'-1}'} \circ G^-_{j'_{n'}} \circ H_+ \circ (H_- \circ H_+)^{\lfloor n/2\rfloor} \circ F_{j,-1}\\
F_{1,j'} \circ G^+_{j_1'} \circ G^-_{j_2'} \circ \cdots \circ G^+_{j_{n'-1}'} \circ G^-_{j'_{n'}} \circ (H_+ \circ H_-)^{n/2} \circ F_{j,1}\\
F_{-1,j'} \circ G^-_{j_1'} \circ G^+_{j_2'} \circ \cdots \circ G^-_{j_{n'-2}'}\circ G^+_{j_{n'-1}'} \circ G^-_{j'_{n'}} \circ (H_+ \circ H_-)^{n/2} \circ F_{j,1}\\
F_{-1,j'} \circ G^-_{j_1'} \circ G^+_{j_2'} \circ \cdots \circ G^-_{j_{n'-1}'} \circ G^+_{j'_{n'}} \circ H_- \circ (H_+ \circ H_-)^{\lfloor n/2\rfloor} \circ F_{j,1}\\
F_{1,j'} \circ G^+_{j_1'} \circ G^-_{j_2'} \circ \cdots \circ G^+_{j_{n'-2}'}\circ G^-_{j_{n'-1}'} \circ G^+_{j'_{n'}} \circ H_- \circ (H_+ \circ H_-)^{\lfloor n/2\rfloor} \circ F_{j,1}
\end{cases}
\]
for
\[
\begin{cases}
\jj < 0,\; n(\jj) \text{ even}, \; n(\jj') \text{ even}\\
\jj < 0,\; n(\jj) \text{ even}, \; n(\jj') \text{ odd}\\
\jj < 0,\; n(\jj) \text{ odd}, \; n(\jj') \text{ even}\\
\jj < 0,\; n(\jj) \text{ odd}, \; n(\jj') \text{ odd}\\
\jj > 0,\; n(\jj) \text{ even}, \; n(\jj') \text{ even}\\
\jj > 0,\; n(\jj) \text{ even}, \; n(\jj') \text{ odd}\\
\jj > 0,\; n(\jj) \text{ odd}, \; n(\jj') \text{ even}\\
\jj > 0,\; n(\jj) \text{ odd}, \; n(\jj') \text{ odd}
\end{cases},
\]
respectively. Note that in the case $n = n' = 0$ the definition of $F_{\jj, \jj'} = F_{j, j'}$ agrees with the previous one. We have $F_{\jj, \jj'} = f_{i_N}\circ \cdots \circ f_{i_1}|_{I_\jj}$ for some $i_1, \ldots, i_N \in \{-,+\}$, $N \ge 0$. 

By \eqref{eq:F-phi-kl}, \eqref{eq:G-phi-kl} and \eqref{eq:H-phi-kl},
\begin{equation}\label{eq:F2-phi-kl}
F_{\jj, \jj'}(x) =
\begin{cases}
\rho^{-j'-1} (\phi_{j'_1, \ldots, j'_{n'}} \circ
\phi_{j_1, \ldots, j_n}^{-1}(\rho^{j+1} x)) &\text{if } \jj < 0, \; \jj' < 0\\
\II(\rho^{j'-1} (\phi_{j'_1, \ldots, j'_{n'}} \circ \phi_{j_1, \ldots, j_n}^{-1} (\rho^{j+1} x))) &\text{if } \jj < 0, \; \jj' > 0\\
\rho^{-j'-1} (\phi_{j'_1, \ldots, j'_{n'}} \circ \phi_{j_1, \ldots, j_n}^{-1} (\rho^{-j+1} \II(x))) &\text{if } \jj > 0, \; \jj' < 0\\
\II(\rho^{j'-1} (\phi_{j'_1, \ldots, j'_{n'}} \circ \phi_{j_1, \ldots, j_n}^{-1}(\rho^{-j+1} \II(x)))) &\text{if } \jj > 0, \; \jj' > 0
\end{cases}
\end{equation}
for $x \in I_\jj$. In particular, this gives
\[
F_{\jj, \jj} = \id|_{I_\jj}, \qquad F_{\jj', \jj''}  \circ  F_{\jj, \jj'} = F_{\jj, \jj''}
\]
for suitable $\jj, \jj', \jj'' \in \Z^*$. Moreover, \eqref{eq:F-kl}, \eqref{eq:G-kl} and \eqref{eq:H-kl} imply
\begin{equation}\label{eq:F2-kl}
F_{\jj, \jj'}(I_{\jj; \rr_1, \ldots, \rr_m}) = I_{\jj'; \rr_1, \ldots, \rr_m}, \qquad
F_{\jj, \jj'}(x_{\jj; \rr_1, \rr_2, \ldots}) = x_{\jj'; \rr_1, \rr_2, \ldots}
\end{equation}

for $\rr_1, \rr_2, \ldots \in \RR$, $m \ge 0$.

\begin{lem}\label{lem:F2-kl}
A trajectory $\{f_{i_N} \circ \cdots \circ f_{i_1}(x)\}_{N=0}^\infty$ of a point $x \in I_\jj$, $\jj \in \JJ$, does not jump over the central interval at any time $0 \le s < N$, for some $N \ge 0$, if and only if 
\[
f_{i_N} \circ \cdots \circ f_{i_1}|_{I_\jj} = F_{\jj, \jj'}
\]
for $\jj'\in \JJ$ such that $f_{i_N} \circ \cdots \circ f_{i_1}(x) \in I_{\jj'}$, where $\sgn(\jj) = \sgn(\jj')$ and $n(\jj)-n(\jj')$ is even, or $\sgn(\jj) \neq \sgn(\jj')$ and $n(\jj)-n(\jj')$ is odd.
\end{lem}
\begin{proof} If a trajectory $\{f_{i_N} \circ \cdots \circ f_{i_1}(x)\}_{N=0}^\infty$ of $x \in I_\jj$ does not jump over the central interval at any time $0 \le s < N$, then by Lemmas~\ref{lem:f} and~\ref{lem:jump-kl},
\[
f_{i_N} \circ \cdots \circ f_{i_1}(I_\jj) = I_{\jj'},
\]
where $\jj'\in \JJ$ such that $\sgn(\jj) = \sgn(\jj')$ and $n(\jj)-n(\jj')$ is even, or $\sgn(\jj) \neq \sgn(\jj')$ and $n(\jj)-n(\jj')$ is odd. Consequently, $F_{\jj, \jj'}$ is defined on $I_\jj$ and $(F_{\jj, \jj'})^{-1} \circ f_{i_N} \circ \cdots \circ f_{i_1}|_{I_\jj}$ is an increasing affine homeomorphism from $I_\jj$ onto itself, so it is identity. Therefore, $f_{i_N} \circ \cdots \circ f_{i_1}|_{I_\jj} = F_{\jj, \jj'}$. The other implication follows from Lemmas~\ref{lem:f} and~\ref{lem:jump-kl} and the definitions of the maps $F_{j,j'}, G_j^\pm, H_\pm$. 

\end{proof}

Define, for $\rr \in \RR$,
\begin{alignat*}{3}
&G^{+,-}_\rr\colon I_1 \to I_{-1}, &\quad &G^{-,+}_\rr\colon I_{-1} \to I_1 &\qquad &\text{for $n(\rr)$ even},\\
&G^{-,-}_\rr\colon I_{-1} \to I_{-1}, &\quad &G^{+,+}_\rr\colon I_1 \to I_1 &\qquad &\text{for $n(\rr)$ odd},
\end{alignat*}
setting
\begin{alignat*}{2}
G^{+,-}_\rr &= F_{-l,-1} \circ f_- \circ F_{1, \rr}, \qquad
&G^{-,+}_\rr &= F_{l,1} \circ f_+ \circ F_{-1, -\rr},\\
G^{-,-}_\rr &= F_{-l,-1} \circ f_- \circ F_{-1, \rr}, \qquad  
&G^{+,+}_\rr &= F_{l,1} \circ f_+ \circ F_{1, -\rr}.
\end{alignat*}
Note that $G^{\pm,\pm}_\rr = f_{i_N}\circ \cdots \circ f_{i_1}|_{I_{\pm 1}}$ for some $i_1, \ldots, i_N \in \{-,+\}$, $N \ge 0$. 
By Lemma~\ref{lem:f} and~\eqref{eq:F2-phi-kl}, we have
\begin{equation}\label{eq:G2-phi-kl}
\begin{aligned}
G^{+,-}_\rr &= \Phi_\rr \circ \II|_{I_1}, &\qquad 
G^{-,+}_\rr &= \II \circ \Phi_\rr,\\
G^{-,-}_\rr &= \Phi_\rr, &\qquad 
G^{+,+}_\rr &= \II \circ \Phi_\rr \circ \II|_{I_1},
\end{aligned}
\end{equation}
while by Lemma~\ref{lem:f(I)-kl} and~\eqref{eq:F2-kl},
\begin{equation}\label{eq:G2-kl}
\begin{aligned}
G^{+,-}_\rr(I_{1; \rr_1, \ldots, \rr_m}) &= I_{-1; \rr, \rr_1, \ldots, \rr_m}, &\qquad 
G^{+,-}_\rr(x_{1; \rr_1, \rr_2, \ldots}) &= x_{-1; \rr, \rr_1, \rr_2, \ldots}\\
G^{-,+}_\rr(I_{-1; \rr_1, \ldots, \rr_m}) &= I_{1; \rr, \rr_1, \ldots, \rr_m}, &\qquad 
G^{-,+}_\rr(x_{-1; \rr_1, \rr_2, \ldots}) &= x_{1; \rr, \rr_1, \rr_2, \ldots}\\
G^{-,-}_\rr(I_{-1; \rr_1, \ldots, \rr_m}) &= I_{-1; \rr, \rr_1, \ldots, \rr_m}, &\qquad 
G^{-,-}_\rr(x_{-1; \rr_1, \rr_2, \ldots}) &= x_{-1; \rr, \rr_1, \rr_2, \ldots}\\
G^{+,+}_\rr(I_{1; \rr_1, \ldots, \rr_m}) &= I_{1; \rr, \rr_1, \ldots, \rr_m}, &\qquad 
G^{+,+}_\rr(x_{1; \rr_1, \rr_2, \ldots}) &= x_{1; \rr, \rr_1, \rr_2, \ldots}
\end{aligned}
\end{equation}
for $\rr_1, \rr_2, \ldots \in \RR$, $m \geq 0$.

\begin{lem}\label{lem:G2-kl}
A trajectory $\{f_{i_N} \circ \cdots \circ f_{i_1}(x)\}_{N=0}^\infty$ of a point $x \in I$ jumps over the central interval at the time $s$, for some $s \ge 0$, if and only if one of the four following possibilities:
\begin{alignat*}{5}
&f_{i_s} \circ \cdots \circ f_{i_1}(x) \in I_{-\rr}, &&\qquad
f_{i_{s+1}}|_{I_{-\rr}} &&= F_{1,l} \circ  G^{-,+}_\rr \circ F_{-\rr,-1}, &&\qquad &&n(\rr) \text{ even},\\
&f_{i_s} \circ \cdots \circ f_{i_1}(x) \in I_{-\rr}, &&\qquad
f_{i_{s+1}}|_{I_{-\rr}} &&= F_{1,l} \circ  G^{+,+}_\rr \circ F_{-\rr,1},&&\qquad &&n(\rr) \text{ odd},\\
&f_{i_s} \circ \cdots \circ f_{i_1}(x) \in I_\rr, &&\qquad
f_{i_{s+1}}|_{I_\rr} &&= F_{-1,-l} \circ  G^{+,-}_\rr \circ F_{\rr,1},&&\qquad &&n(\rr) \text{ even},\\
&f_{i_s} \circ \cdots \circ f_{i_1}(x) \in I_\rr, &&\qquad
f_{i_{s+1}}|_{I_\rr} &&= F_{-1,-l} \circ  G^{-,-}_\rr \circ F_{\rr,-1}&&\qquad &&n(\rr) \text{ odd},
\end{alignat*}
holds for some $\rr \in \RR$.
\end{lem}
\begin{proof} Follows directly from Lemma~\ref{lem:jump-kl}.

\end{proof}

\begin{lem}\label{lem:traj-kl}
A trajectory $\{f_{i_N} \circ \cdots \circ f_{i_1}(x)\}_{N=0}^\infty$ of a point $x \in I_\jj$, $\jj \in \JJ$, jumps over the central interval $($exactly$)$ at the times 
$s_1, \ldots, s_m$, for some $0 \le s_1 < \cdots < s_m < N$, $0 \le m \le N$, 
if and only if
\[
f_{i_N} \circ \cdots \circ f_{i_1}|_{I_\jj} = 
F_{t', \jj'} \circ G^{\sigma_1, \sigma_0}_{\rr_1} \circ \cdots \circ G^{\sigma_m, \sigma_{m-1}}_{\rr_m} \circ F_{\jj, t},
\]
for some $\rr_1, \ldots, \rr_m \in \RR$, where $\sigma_s \in \{-, +\}$, $s = 0, \ldots, m,\ t, t' \in \{-1,1\}$, are defined by backward induction as
\begin{align*}
\sigma_m &= 
\begin{cases}
- &\text{if } \jj < 0, \, n(\jj) \text{ is even, or } \jj > 0, \, n(\jj) \text{ is odd}\\
+ &\text{if } \jj < 0, \, n(\jj) \text{ is odd, or } \jj > 0, \, n(\jj) \text{ is even}
\end{cases},\ \sgn(t)=\sigma_m\\
\sigma_{s-1} &= 
\begin{cases}
- &\text{if } \sigma_s = -, \, n(\rr_s) \text{ is odd, or } \sigma_s = +, \, n(\rr_s) \text{ is even}\\
+ &\text{if } \sigma_s = -, \, n(\rr_s) \text{ is even, or } \sigma_s = +, \, n(\rr_s) \text{ is odd}
\end{cases} \qquad \text{for } s = m, \ldots, 1,
\end{align*}
and
\[
\sgn(t) = \sigma_m, \qquad \sgn (t') = \sigma_0
\]
with $\jj' \in \JJ$ such that $\sgn(\jj') = \sigma_0$ and $n(\jj')$ is even, or 
$\sgn(\jj') \neq \sigma_0$ and $n(\jj')$ is odd. Moreover, in this case we have
\[
f_{i_N} \circ \cdots \circ f_{i_1}(I_\jj) = I_{\jj'; \rr_1, \ldots, \rr_m}
\]
and
\begin{align*}
&f_{i_N} \circ \cdots \circ f_{i_1}(x)\\ 
&=\begin{cases}
\rho^{-j'-1}(\phi_{j'_1, \ldots, j'_n} \circ \Phi_{r_1} \circ \cdots \circ \Phi_{r_m}\circ 
\phi_{j_1, \ldots, j_n}^{-1}(\rho^{j+1}x)) & \text{for }\jj<0,\, \jj'<0\\
\II(\rho^{j'-1}(\phi_{j'_1, \ldots, j'_n}\circ  \Phi_{r_1} \circ \cdots \circ \Phi_{r_m}\circ 
\phi_{j_1, \ldots, j_n}^{-1}(\rho^{j+1}x))) & \text{for }\jj < 0, \, \jj'>0\\
\rho^{-j'-1}(\phi_{j'_1, \ldots, j'_n} \circ  \Phi_{r_1} \circ \cdots \circ \Phi_{r_m}\circ 
\phi_{j_1, \ldots, j_n}^{-1}(\rho^{-j+1}\II(x))) & \text{for }\jj > 0,\, \jj'<0\\
\II(\rho^{j'-1}(\phi_{j'_1, \ldots, j'_n} \circ  \Phi_{r_1} \circ \cdots \circ \Phi_{r_m}\circ 
\phi_{j_1, \ldots, j_n}^{-1}(\rho^{-j+1}\II(x)))) & \text{for }\jj>0, \, \jj'>0
\end{cases},
\end{align*}
where $\jj = (j, j_1, \ldots, j_n)$, $\jj' = (j', j'_1, \ldots, j'_{n'})$.
\end{lem}
\begin{proof} The definitions of $\sigma_s$, $t, t'$ and the conditions for $\jj'$ imply that all the considered maps are well-defined. The assertions of the lemma follow directly from Lemmas~\ref{lem:F2-kl} and~\ref{lem:G2-kl}, and \eqref{eq:F2-phi-kl}, \eqref{eq:F2-kl}, \eqref{eq:G2-phi-kl}, \eqref{eq:G2-kl}.
\end{proof}

\begin{lem}\label{lem:boldF}
For every $\jj, \jj' \in \JJ$, $m \geq 0$ and $\rr_1, \ldots, \rr_m \in \RR$, there exists a map 
\[
\FFF_{\jj, \jj'; \rr_1, \ldots, \rr_m}\colon I_\jj \to I_{\jj';\rr_1, \ldots, \rr_m}
\]
such that $\FFF_{\jj, \jj'; \rr_1, \ldots, \rr_m} = f_{i_N}\circ \cdots \circ f_{i_1}|_{I_\jj}$ for some $i_1, \ldots, i_N \in \{-,+\}$, $N \ge 0$ and any trajectory of $x \in \JJ$ defined by $\cdots \circ f_{i_N}\circ \cdots \circ f_{i_1}(x)$ jumps over the central interval at the times 
$s_1, \ldots, s_{m+1}$, for some $0 \le s_1 < \cdots < s_{m+1} < N$.
\end{lem}
\begin{proof}
Since the system is symmetric, we can assume $\jj < 0$.

Let
\[
t = 
\begin{cases}
-1 &\text{if } n(\jj) \text{ is even}\\
1 &\text{if } n(\jj)\text{ is odd}
\end{cases},
\qquad
t' = 
\begin{cases}
-1 &\text{if } \jj' < 0, \: n(\jj') \text{ is even}, \:\text{or } \jj' > 0, \: n(\jj')\text{ is odd}\\
1 &\text{if } \jj' < 0, \: n(\jj') \text{ is odd}, \:\text{or } \jj' > 0, \: n(\jj')\text{ is even}
\end{cases}
\]
and
\[
p = \#\{s \in \{1, \ldots, m\}: n(\rr_s) \text{ is even}\}.
\]
Define $\rr_{m+1} \in \RR$ by
\[
\rr_{m+1} = 
\begin{cases}
l &\text{if } t = t', \, p \text{ is odd, or }
t \neq t', \, p \text{ is even}\\
(l, 1) &\text{if } t = t', \: p \text{ is even, or }
t \neq t', \, p \text{ is odd}
\end{cases}.
\]
We have
\begin{equation}\label{eq:n(r')}
n(\rr_{m+1}) \text{ is } 
\begin{cases}
\text{even} &\text{if } t = t', \: p \text{ is odd, or }
t \neq t', \, p \text{ is even}\\
\text{odd} &\text{if } t = t', \: p \text{ is even, or }
t \neq t', \, p \text{ is odd}
\end{cases}.
\end{equation}
Furthermore, define $\sigma_s \in \{-, +\}$, $s = 0, \ldots, m+1$ by 
\begin{align*}
\sigma_{m+1} &= \sgn(t),\\
\sigma_{s-1} &= 
\begin{cases}
- &\text{if } \sigma_s = -, \, n(\rr_s) \text{ is odd, or } \sigma_s = +, \, n(\rr_s) \text{ is even}\\
+ &\text{if } \sigma_s = -, \, n(\rr_s) \text{ is even, or } \sigma_s = +, \, n(\rr_s) \text{ is odd}
\end{cases} \qquad \text{for } s = m+1 , \ldots, 1.
\end{align*}
By the definition of $t$, we have 
\[
\sigma_{m+1} = 
\begin{cases}
- &\text{if } n(\jj) \text{ is even}\\
+ &\text{if } n(\jj) \text{ is odd}
\end{cases}.
\]
Note that $\sigma_{s-1} \neq \sigma_s$ if and only if $n(\rr_s)$ is even. Therefore, as $\sigma_{m+1} = \sgn(t)$ we obtain
\[
\sigma_0 = 
\begin{cases}
\sgn(t) &\text{if } p \text{ is even}, \, n(\rr_{m+1}) \text{ is odd, or } p \text{ is odd}, \, n(\rr_{m+1}) \text{ is even}\\
-\sgn(t) &\text{if } p \text{ is even}, \, n(\rr_{m+1}) \text{ is even, or } p \text{ is odd}, \, n(\rr_{m+1}) \text{ is odd}
\end{cases},
\]
where $-\sgn(t) = -$ (resp.~$+$) if $\sgn(t) = +$ (resp.~$-$). This together with \eqref{eq:n(r')} implies 
\[
\sigma_0 = \sgn(t').
\]
Moreover, by the definition of $t'$, we have $\sgn(\jj') = \sigma_0$ and $n(\jj')$ is even, or $\sgn(\jj') \neq \sigma_0$ and $n(\jj')$ is odd.
This implies that if we define
\[
\FFF_{\jj, \jj'; \rr_1, \ldots, \rr_m} = F_{t', \jj'} \circ G^{\sigma_1, \sigma_0}_{\rr_1} \circ \cdots \circ G^{\sigma_{m+1}, \sigma_m}_{\rr_{m+1}} \circ F_{\jj, t},
\]
then by Lemma~\ref{lem:traj-kl} (with $m$ replaced by $m+1$), $\FFF_{\jj, \jj'; \rr_1, \ldots, \rr_m}$ is well-defined on $I_\jj$ and  $\FFF_{\jj, \jj'; \rr_1, \ldots, \rr_m} = f_{i_N}\circ \cdots \circ f_{i_1}|_{I_\jj}$ for some $i_1, \ldots, i_N \in \{-,+\}$, $N \ge 0$. Moreover, any trajectory of $x \in \JJ$ defined by $\cdots \circ f_{i_N}\circ \cdots \circ f_{i_1}(x)$ jumps over the central interval at the times $s_1, \ldots, s_{m+1}$, for some $0 \le s_1 < \cdots < s_{m+1} < N$. By \eqref{eq:F2-kl} and \eqref{eq:G2-kl}, 
\[
\FFF_{\jj, \jj'; \rr_1, \ldots, \rr_m}(I_\jj) = I_{\jj';\rr_1, \ldots, \rr_{m+1}} \subset I_{\jj';\rr_1, \ldots, \rr_m}.
\]
\end{proof}

\begin{prop}\label{prop:omega-kl} For every $x \in (0,1)$,
\[
\omega_\infty(x) = \overline{\Lambda} = \Lambda \cup \{0,1\}.
\]
\end{prop}
\begin{proof} 
First, we prove $\omega_\infty(x) \subset \Lambda \cup \{0,1\}$ for $x \in (0,1)$. By Lemma~\ref{lem:jump-kl}(a), we can assume $x \in I$. Take $y \in \omega_\infty(x)$. We have $y = \lim_{s \to \infty} f_{i_{N_s}}\circ \cdots \circ f_{i_1}(x)$, where $N_s \to \infty$ and the trajectory $\{f_{i_N} \circ \cdots \circ f_{i_1}(x)\}_{N=0}^\infty$ jumps over the central interval infinitely many times. By Lemma~\ref{lem:traj-kl}, 
\[
f_{i_{N_s}} \circ \cdots \circ f_{i_1}(x) \in I_{\jj(s); \rr_1(s), \ldots, \rr_{m(s)}(s)}
\]
for some $\jj(s) \in \JJ$ and $\rr_1(s), \ldots, \rr_{m(s)}(s) \in \RR$, where $m(s) \to \infty$ as $s \to \infty$. Moreover, $|I_{\jj(s); \rr_1(s), \ldots, \rr_{m(s)}(s)}| \le \rho^{m(s)} \to 0$ as $s \to \infty$ and $I_{\jj(s); \rr_1(s), \ldots, \rr_{m(s)}(s)} \cap \Lambda \neq 0$. Hence, $y \in \overline\Lambda = \Lambda \cup \{0,1\}$, which shows $\omega_\infty(x) \subset \Lambda \cup \{0,1\}$. 

Now we prove $\Lambda \cup \{0,1\}\subset \omega_\infty(x)$ for $x \in (0,1)$. By Lemma~\ref{lem:inI-kl}, we can assume $x \in I_\jj$ for some $\jj \in \JJ$.  
Take $y \in \Lambda$. Then $y = \lim_{s \to \infty} x_{\jj'(s); \rr_1(s), \rr_2(s), \ldots}$ for some $\jj'(s) \in \JJ$ and $\rr_1(s), \rr_2(s), \ldots \in \RR$, $s \ge 0$. Using Lemma~\ref{lem:boldF}, define inductively
\begin{align*}
F^{(0)} &= \FFF_{\jj, \jj'(0)},\\
F^{(s)} &= \FFF_{\jj'(s-1), \jj'(s); \rr_1(s), \ldots, \rr_s(s)} \qquad \text{for } s > 0.
\end{align*}
By Lemma~\ref{lem:boldF}, the trajectory of $x$ under $\{f_-, f_+\}$ defined by $\cdots\circ F^{(s)}\cdots\circ F^{(0)}(x)$ is well-defined and jumps over the central interval infinitely many times. Moreover, 
\[
F^{(s)} \circ \cdots \circ  F^{(0)}(I_\jj) \subset I_{\jj'(s); \rr_1(s), \ldots, \rr_s(s)},
\]
so
\[
|F^{(s)} \circ \cdots \circ  F^{(0)}(x) - y| \leq |I_{\jj'(s); \rr_1(s), \ldots, \rr_s(s)}| + |y - x_{\jj'(s); \rr_1(s), \rr_2(s), \ldots}| \to 0
\]
as $s \to \infty$, since $|I_{\jj'(s); \rr_1(s), \ldots, \rr_s(s)}| \le \rho^s \to \infty$. Hence, $y$ is a limit point of this trajectory.

Take now $y \in\{0,1\}$. Then, by Lemma~\ref{lem:boldF}, we see 
\begin{align*}
\FFF_{2s-1,-2s}\circ \cdots \circ \FFF_{-2,3}\circ \FFF_{1,-2}\circ \FFF_{\jj,1}(x) &\in I_{-2s},\\
\FFF_{-2s,2s+1}\circ \FFF_{2s-1,-2s}\circ \cdots \circ \FFF_{-2,3}\circ \FFF_{1,-2}\circ \FFF_{\jj,1}(x)&\in I_{2s+1}
\end{align*}
for $s > 0$, the trajectory defined by
\[
\cdots\circ \FFF_{-2s,2s+1}\circ \FFF_{2s-1,-2s}\circ \cdots \circ \FFF_{-2,3}\circ \FFF_{1,-2}\circ \FFF_{\jj,1}(x)
\]
jumps over the central interval infinitely many times and has $y$ as its limit point. Hence, $\Lambda\cup \{0,1\}\subset \omega_\infty(x)$.

\end{proof}

\begin{prop}\label{prop:X-kl}
We have
\[
\Lambda = f_-(\Lambda) = f_+(\Lambda).
\]
Moreover, the system $\{f_-, f_+\}$ is minimal in $\Lambda$.
\end{prop}

\begin{proof}
The first assertion follows directly from Lemma~\ref{lem:f(I)-kl}, while Proposition~\ref{prop:omega-kl} implies minimality.
\end{proof}

\subsection*{Singularity of \boldmath$\mu$}

\begin{prop}\label{prop:supp-kl} We have 
\[
\supp \mu = \Lambda \cup \{0, 1\},\quad \mu(\Lambda) = 1.
\]
\end{prop}
\begin{proof}
Similarly as for the case $l=1$, it is enough to use Proposition~\ref{prop:X-kl}.
\end{proof}

\begin{prop}\label{prop:dimX-kl}
\[
\dim_H \Lambda =  \frac{\log \eta}{\log \rho} < 1,
\]
where $\eta \in (1/2, 1)$ is the unique solution of the equation $\eta^{k+l} -2 \eta^{k+1} + 2 
\eta - 1 =0$.
\end{prop}
\begin{proof}  Our first goal is to determine the dimension of $\Lambda_{-1}$. We begin with calculating the dimension of the $L$ defined in \eqref{eq:L_def}. Recall that $\{\Phi_\rr\}_{\rr \in \RR}$ is an iterated function system of contracting similarities on $I_{-1}$, satisfying the Strong Separation Condition.
It is well-known (see e.g.~\cite[Theorem~3.15]{MU}) that for such systems $\dim_H L$ is equal to the (unique) zero of the topological pressure function
\[
P(t) = \lim_{n\to\infty} \frac 1 n \log \sum_{\rr_1, \ldots, \rr_n \in \RR} \|(\Phi_{\rr_1} \circ \cdots \circ \Phi_{\rr_n})'\|^t,
\]
provided the system is regular (i.e.~zero of the pressure function exists). Since $\Phi_\rr$ are affine, we have
\begin{equation}\label{eq:pressure}
\begin{aligned}
P(t) &= \log \sum_{\rr\in\RR} \|\Phi_\rr'\|^t\\
&= \log \left(\sum_{r=l}^k|\phi_r'|^t + \sum_{n=1}^\infty\sum_{r=l}^{k-1} \sum_{r_1, \ldots, r_n = 1}^{l-1} |(\phi_r \circ \phi_{r_1} \circ \cdots \circ \phi_{r_n})'|^t \right)\\
&= \log \left(|\phi_k'|^t + \sum_{r=l}^{k-1}  |\phi_r'|^t \sum_{n=0}^\infty \Big(\sum_{r_1=1}^{l-1}  |\phi_{r_1}'|^t \Big)^n\right)\\
&=\log \left( \rho^{kt} + \sum_{r=l}^{k-1}  \rho^{rt} \sum_{n=0}^\infty \Big(\sum_{r_1=1}^{l-1}  \rho^{r_1t} \Big)^n\right)\\
&= \log \frac{\rho^{lt}-2\rho^{(k+1)t}+\rho^{(k+l)t}}{1-2\rho^t+\rho^{lt}}
\end{aligned}
\end{equation}
provided $\rho^t + \cdots + \rho^{(l-1)t} < 1$, which is equivalent to $\rho^{lt}-2\rho^t+1 > 0$. Since by Lemma~\ref{lem:ifs}, $\phi_r(J_{-1})$, $r = 1, \ldots, k$ are pairwise disjoint subset of $J_{-1}$, we have $\rho^t + \cdots + \rho^{(l-1)t} < 1$ for $t=1$. It follows that $\rho^{lt}-2\rho^t+1 > 0$ for $t \in (t_0, 1]$, where $t_0 = \inf\{t>0: P(t) < \infty\} \in (0,1)$ is the unique solution of the equation $\rho^{lt_0}-2\rho^{t_0}+1 = 0$. Moreover, the condition $P(1) < 0$ is equivalent to 
\[
\frac{\rho^l-2\rho^{k+1}+\rho^{k+l}}{1-2\rho+\rho^l} < 1,
\]
which is the same as \eqref{eq:ass2}. Since $t \mapsto P(t)$ is strictly  decreasing and continuous whenever it is finite (see \cite{MU}) and $\lim_{t \to t_0^+} P(t) = +\infty$, we see that there exists $d \in (t_0, 1)$ such that $P(d) = 0$. By \eqref{eq:pressure}, we have $\eta = \rho^d$, so
\[
\dim_H L =  d = \frac{\log \eta}{\log \rho} < 1. 
\]
We will prove now that $\dim_H \Lambda_{-1} = \dim_H L$, i.e.~taking the closure does not increase the Hausdorff dimension of $L$. To that end, let $L(\infty)$ be the ``asymptotic boundary'' of the system $\{ \Phi_\rr\}_{\rr \in \RR}$, i.e.~the set of all limit points of sequences $(x_s)_{s=1}^\infty$, where $x_s \in \Phi_{\rr_s}(I_{-1})$ and $\{\rr_s\}_{s=1}^\infty$ consists of mutually distinct elements of $\RR$. It follows from \cite[Lemma~2.1]{MU} that
\[ 
\Lambda_{-1} = \overline{L} = L \cup \bigcup \limits_{m=0}^{\infty}\; \bigcup \limits_{\rr_1, \ldots , \rr_m \in \RR} \Phi_{\rr_1}\circ \cdots \circ \Phi_{\rr_m}(L(\infty)). 
\]
As the above sum is countable and the transformations $\Phi_{\rr}$ are bi-Lipschitz, we obtain
\[ 
\dim_H\Lambda_{-1} = \max \{ \dim_H L, \dim_H L(\infty) \}. 
\]
Using Lemmas~\ref{lem:ifs} and~\ref{lem:intervals-kl}, it is easy to see that
\[
L(\infty) = \bigcup \limits_{r=l}^{k-1} \phi_r(K),
\]
where $K$ is the limit set of the iterated function system $\{\phi_r\}_{r = 1}^{l-1}$ on $J_{-1}$. By Lemma~\ref{lem:ifs}, this system satisfies the Strong Separation Condition, so its box and Hausdorff dimension are both equal to the unique solution $t_0 \in (0,1)$ of the equation $\rho^{lt_0}-2\rho^{t_0}+1 = 0$. As noted above, we have $t_0 < d$, hence $\dim_H\Lambda_{-1} = d$. By Lemma~\ref{lem:intervals-kl}, the sets $\Lambda_\jj$, $\jj \in \JJ$, are disjoint similar copies of $\Lambda_{-1}$, so $\dim_H \bigcup_{\jj \in \JJ} \Lambda_\jj = \dim_H \Lambda_{-1}$. To end the proof, note that $\Lambda \setminus \bigcup_{\jj \in \JJ} \Lambda_\jj = \bigcup_{j >0 } \big( \rho^j K \cup \II( \rho^j K) \big)$, hence
\[ 
\dim_H \big( \Lambda \setminus \bigcup_{\jj \in \JJ} \Lambda_\jj \big)= t_0 < d.
\]
Finally, this implies $\dim_H \Lambda = d$.
\end{proof}

The following proposition gives some information about the structure of the measure $\mu$ in the case of equal probabilities $p_-, p_+$.

\begin{prop}\label{prop:struct} Suppose $p_- = p_+ = 1/2$. Then for $\jj = (j, j_1, \ldots, j_n) \in \JJ$, $\rr_1, \ldots, \rr_m$, $m \ge 0$, we have
\[
\mu(I_{\jj; \rr_1, \ldots, \rr_m}) =  m_\jj \beta_{\rr_1} \cdots \beta_{\rr_m},
\]
for
\begin{align*}
m_\jj = m_{j, j_1, \ldots, j_n} &=  \mu(I_\jj) = A_{1,j_1, \ldots, j_n} \lambda_1^{|j|}+ \cdots + A_{p,j_1, \ldots, j_n} \lambda_p^{|j|}\\
&+ A_{p+1,j_1, \ldots, j_n} \lambda_{p+1}^{|j|} + \overline{A_{p+1,j_1, \ldots, j_n}} \:\overline{\lambda_{p+1}}^{|j|}+\cdots + A_{q,j_1, \ldots, j_n} \lambda_q^{|j|}+ \overline{A_{q,j_1, \ldots, j_n}} \:\overline{\lambda_q}^{|j|},
\end{align*}
where $A_{1,j_1, \ldots, j_n}, \ldots , A_{p,j_1, \ldots, j_n} \in \R$, $A_{p+1,j_1, \ldots, j_n}, \ldots , A_{q,j_1, \ldots, j_n} \in \C$, moreover $\lambda_1, \ldots, \lambda_p$ $($resp. $\lambda_{p+1}, \overline{\lambda_{p+1}}, \ldots, \lambda_q, \overline{\lambda_q})$ are real $($resp.~non-real$)$ roots of the polynomial $x^{k+l}-2x^l-1$ of moduli smaller than $1$ and
\[
\beta_\rr = \frac{m_\rr}{2m_l - m_{l+k}}.
\]
\end{prop}
\begin{proof} Let $m_\jj = \mu(I_\jj)$ for $\jj = (j, j_1, \ldots, j_n) \in \JJ$ and define $\beta_\rr$ for $\rr \in \RR$ as in the proposition. 
Note that the assumption $p_- = p_+ = 1/2$ and the uniqueness of $\mu$ imply (recall that $-\jj = (-j, j_1, \ldots , j_n)$ for $\jj = (j, j_1, \ldots , j_n)$)
\begin{equation}\label{eq:sym-eta}
m_{-\jj} = m_\jj.
\end{equation}
Furthermore, by Lemma~\ref{lem:f} and the stationarity of $\mu$, for every fixed $j_1, \ldots, j_n$ we have
\[
m_{j+k, j_1, \ldots, j_n} = 2m_{j, j_1, \ldots, j_n} - m_{j-l, j_1, \ldots, j_n}
\]
for every $j \in \N$, $j \geq l+1$. This defines a linear difference equation with characteristic polynomial $x^{k+l}-2x^l-1$. It is well-known (see e.g.~\cite{difference-equations}) that a solution of such an equation has the form 
\begin{align*}
m_{j, j_1, \ldots, j_n} &= A_{1,j_1, \ldots, j_n} \lambda_1^j+ \cdots + A_{p,j_1, \ldots, j_n} \lambda_p^j\\
&+ A_{p+1,j_1, \ldots, j_n} \lambda_{p+1}^j + \overline{A_{p+1,j_1, \ldots, j_n}} \:\overline{\lambda_{p+1}}^j+\cdots + A_{q,j_1, \ldots, j_n} \lambda_q^j+ \overline{A_{q,j_1, \ldots, j_n}} \:\overline{\lambda_q}^j,
\end{align*}
$j \in \N$, where $\lambda_1, \ldots, \lambda_p$ $($resp.~$\lambda_{p+1}, \overline{\lambda_{p+1}}, \ldots, \lambda_q, \overline{\lambda_q})$ are real $($resp.~non-real$)$ roots of the characteristic polynomial and $A_{1,j_1, \ldots, j_n}, \ldots , A_{p,j_1, \ldots, j_n} \in \R$, $A_{p+1,j_1, \ldots, j_n}, \ldots , A_{q,j_1, \ldots, j_n} \in \C$. Since $\sum_{j = 1}^\infty m_{j, j_1, \ldots, j_n} \le \mu(I) = 1$, in fact we take into account only the roots of moduli smaller than $1$. This proves that $m_\jj$ has the form described in the proposition. 

To show $\mu(I_{\jj; \rr_1, \ldots, \rr_m}) =  m_\jj \beta_{\rr_1} \cdots \beta_{\rr_m}$, note that by Lemma~\ref{lem:f} and the stationarity of $\mu$, 
\[
m_l = \frac 1 2 \sum_{\rr \in \RR} m_\rr + \frac 1 2 m_{l+k},
\]
which together with Proposition~\ref{prop:supp-kl} implies
\[
\beta_\rr > 0, \qquad \sum_{\rr \in \RR}\beta_\rr = 1.
\]
Let $\nu(I_{\jj; \rr_1, \ldots, \rr_m}) = m_\jj \beta_{\rr_1} \cdots \beta_{\rr_m}$ for $\jj \in \JJ$, $\rr_1, \ldots, \rr_m$, $m \ge 0$. Since the family of sets $I_{\jj; \rr_1, \ldots, \rr_m}$ generates the $\sigma$-algebra of Borel sets in $\Lambda$, $\nu$ extends to a Borel probability measure on $\Lambda$. Therefore, by the uniqueness of the stationary measure, to prove the proposition it is sufficient to check that $\nu$ is stationary. It is enough to verify 
\begin{equation}\label{eq:stat-I}
\nu(I_{\jj; \rr_1, \ldots, \rr_m}) = \frac{1}{2} \nu(f^{-1}_-(I_{\jj; \rr_1, \ldots, \rr_m})) + \frac{1}{2} \nu(f^{-1}_+(I_{\jj; \rr_1, \ldots, \rr_m})).
\end{equation}

By Lemma~\ref{lem:f(I)-kl}, for $\jj = (j, j_1, \ldots, j_n) \in \JJ$, we have 
\begin{equation}\label{eq:f^-1-kl}
\begin{aligned}
f_-^{-1}(I_{(j, j_1, \ldots, j_n); \rr_1, \ldots, \rr_m}) &=
\begin{cases}
I_{(j+l,j_1, \ldots, j_n); \rr_1, \ldots, \rr_m} &\text{for } j \le -l-1\\
I_{(j_1, \ldots, j_n); \rr_1, \ldots, \rr_m} &\text{for } j = -l, \, n >0\\
I_{\rr_1; \rr_2, \ldots, \rr_m} &\text{for } j = -l, \, n = 0 \\
I_{(k, -j, j_1, \ldots, j_n); \rr_1, \ldots, \rr_m} &\text{for } -l+1 \le j \le -1\\
I_{(j+k,j_1, \ldots, j_n); \rr_1, \ldots, \rr_m} &\text{for } j >0
\end{cases},\\
f_+^{-1}(I_{(j, j_1, \ldots, j_n); \rr_1, \ldots, \rr_m}) &=
\begin{cases}
I_{(j-k,j_1, \ldots, j_n); \rr_1, \ldots, \rr_m} &\text{for } j < 0\\
I_{(-k, j, j_1, \ldots, j_n); \rr_1, \ldots, \rr_m} &\text{for } 1 \le j \le l-1\\
I_{-\rr_1; \rr_2, \ldots, \rr_m} &\text{for } j = l, \, n = 0 \\
I_{(-j_1, j_2, \ldots, j_n); \rr_1, \ldots, \rr_m} &\text{for } j=l, \, n > 0\\
I_{(j-l,j_1, \ldots, j_n); \rr_1, \ldots, \rr_m} &\text{for } j \ge l+1
\end{cases}.
\end{aligned}
\end{equation}
By \eqref{eq:sym-eta} and \eqref{eq:f^-1-kl}, the statement \eqref{eq:stat-I} is equivalent to the systems of equations
\[
\begin{cases}
m_{j, j_1, j_2, \ldots, j_n}= \frac{1}{2}m_{k, j, j_1, j_2, \ldots, j_n} + \frac{1}{2}m_{j+k, j_1, j_2, \ldots, j_n} & \text{for } 1 \leq j \le l-1\\
m_l \beta_\rr= \frac{1}{2}m_\rr + \frac{1}{2}m_{l+k}\beta_\rr &\\
m_{l, j_1, j_2, \ldots, j_n}= \frac{1}{2}m_{j_1, j_2, \ldots, j_n} + \frac{1}{2}m_{l+k, j_1, j_2, \ldots, j_n} & \text{for } n > 0\\
m_{j, j_1, j_2, \ldots, j_n}= \frac{1}{2}m_{j-l, j_1, j_2, \ldots, j_n} + \frac{1}{2}m_{j+k, j_1, j_2, \ldots, j_n} & \text{for } j \ge l+1\\
\end{cases},
\]
where $(j, j_1, j_2, \ldots, j_n) \in \JJ$ and $\rr \in \RR$. The second equation is equivalent to the definition of $\beta_\rr$ and the remaining ones hold due to \eqref{eq:sym-eta}, \eqref{eq:f^-1-kl} for $m=0$, and the fact that $\mu$ is stationary.

\end{proof}

\section{Proof of Theorem~\ref{thm:conjugacy}}\label{sec:conj}

Let $\Lambda(f)$ and $\Lambda(g)$ be the sets constructed in Section~\ref{sec:l=1} (in the case $l=1$) or Section~\ref{sec:l>1} (in the case $l>1$) for the systems $\{f_-, f_+\}$ and $\{g_-, g_+ \}$, respectively. Following the notation used in these sections, we have 
\begin{align*}
\Lambda(f) &= \{ x_{j; r_1, r_2, \ldots}^{(f)}: \: j \in \Z^*, r_1, r_2, \ldots \in \{1, \ldots, k\}\},\\
\Lambda(g) &= \{ x_{j; r_1, r_2, \ldots}^{(g)}: \: j \in \Z^*, r_1, r_2, \ldots \in \{1, \ldots, k\}\}
\end{align*}
in the case $l = 1$ and
\begin{align*}
\Lambda(f) &= \overline{\Big\{ x_{\jj; \rr_1, \rr_2, \ldots}^{(f)}: \: \jj \in \JJ, \rr_1, \rr_2, \ldots \in \RR\Big\}}\cap (0,1),\\
\Lambda(g) &= \overline{\Big\{ x_{\jj; \rr_1, \rr_2, \ldots}^{(g)}: \: \jj \in \JJ, \rr_1, \rr_2, \ldots \in \RR\Big\}}\cap (0,1)
\end{align*}
in the case $l > 1$. We define the conjugating homeomorphism $h$ setting 
\[
h(x_{j; r_1, r_2, \ldots}^{(f)}) = x_{j; r_1, r_2, \ldots}^{(g)}
\]
in the case $l = 1$ and
\[
h(x_{\jj; \rr_1, \rr_2, \ldots}^{(f)}) = x_{\jj; \rr_1, \rr_2, \ldots}^{(g)}
\]
in the case $l > 1$ (with a unique continuous extension to $\Lambda$). By the definition of $\Lambda(f), \Lambda(g)$, the map $h$ is an increasing homeomorphism between $\Lambda(f)$ and $\Lambda(g)$, while Lemmas~\ref{lem:f(I)} and~\ref{lem:f(I)-kl} imply that it conjugates $\{f_-,f_+\}|_{\Lambda(f)}$ to $\{g_-,g_+\}|_{\Lambda(g)}$. It is easy to see that $h$ can be extended to an increasing homeomorphism of $[0,1]$ conjugating $\{f_-,f_+\}$ to $\{g_-,g_+\}$, such that $h$ is affine on each component of $(0,1) \setminus \Lambda(f)$. For completeness, below we present a detailed construction for the case $l = 1$, leaving the case $l > 1$ to the reader.

From the considerations preceding Proposition~\ref{prop:symbolic_measure}, it follows that $\{f_-, f_+\}|_{\Lambda(f)}$ and $\{g_-, g_+\}|_{\Lambda(g)}$ are both conjugated to the system $\{\tilde f_-, \tilde f_+\}$ acting on $\Z^* \times \Sigma_k$. Hence, there exists a homeomorphism $h \colon \Lambda(f) \to \Lambda(g)$ conjugating $\{f_-, f_+\}$ on $\Lambda(f)$ to  $\{g_-, g_+ \}$ on $\Lambda(g)$. We claim that $h$ can be extended in a continuous and equivariant manner to the interval $[0,1]$. To show this, we describe the structure of the complement of $\Lambda(f)$ in $[0,1]$. 

Like in the proof of Lemma~\ref{lem:inI}, let
\[
U_0 = (f_-(x_-), f_+(x_+)) = (f_-(x_-), \II(f_-(x_-))) = \left( \frac{\rho - \rho^{k+1}}{1 - \rho^{k+1}}, \frac{1 - \rho}{1 - \rho^{k+1}} \right)
\]
and for $j \in \Z^*$ define
\[
U_j = 
\begin{cases}
\rho^{-j} U_0 &\text{for } j < 0\\
\II(\rho^j U_0) &\text{for } j > 0
\end{cases}.
\]
By Lemma~\ref{lem:intervals}, the following statements hold.
\begin{enumerate}[\rm (a)]
\item $U_{-j} = \II(U_j)$ for $j \in \Z$.
\item The sets $U_j$, $j \in \Z$, are pairwise disjoint and together with $I_j$, $j \in \Z^*$, form a partition of $(0,1)$, where $U_j$ is the gap between $I_{j-1}$ and $I_j$ for $j < 0$, $U_0$ is the gap between $I_{-1}$ and $I_1$, and $U_j$ is the gap between $I_j$ and $I_{j+1}$ for $j > 0$.
\item $f_-(U_j) = U_{j-1}$ for $j \le 0$, $f_-(I_1 \cup U_1 \cup \cdots \cup I_{k-1} \cup U_{k-1} \cup I_k) = I_{-1}$, $f_-(U_j) = U_{j-k}$ for $j \ge k$.
\item $f_+(U_j) = U_{j+k}$ for $j \le -k$, $f_+(I_{-k} \cup U_{-k+1} \cup \cdots \cup I_{-2} \cup U_{-1} \cup I_{-1}) = I_1$, $f_+(U_j) = U_{j+1}$ for $j \ge 0$.
\end{enumerate}

For $s=1, \ldots, k-1$, define
\[ 
U_{-1}^s = f_(U_s) = \rho(U_s) \subset I_{-1}. 
\]
Note that $U_{-1}^s$ are the gaps between cylinders of the first order for the iterated function system $\{ \phi_1, \ldots, \phi_k\}$ on $I_{-1}$. More precisely, $U_{-1}^1,\ldots, U_{-1}^{k-1}$ together with $I_{-1;1}, \ldots, I_{-1;k}$ form a partition of $I_{-1}$, and are situated in the order
\[ 
I_{-1;1}, U_{-1}^1, I_{-1;2}, U_{-1}^2, \ldots, I_{-1;k-1}, U_{-1}^{k-1}, I_{-1;k}. 
\]
For $j\in \Z^*$, $s \in \{1, \ldots, k-1\}$ and $r_1, \ldots, r_n \in \{1, \ldots, k\}$, $n \geq 0$, define
\[ 
U^s_{j; r_1,\ldots, r_n} = 
\begin{cases}
\rho^{-j-1} \phi_{r_1} \circ \cdots \circ \phi_{r_n} (U_{-1}^s) &\text{for } j<0\\
\II(\rho^{j-1} \phi_{r_1} \circ \cdots \circ \phi_{r_n} (U_{-1}^s)) &\text{for } j>0
\end{cases},
\]
where $\phi_{r_1} \circ \ldots \circ \phi_{r_n} = \id$, $U^s_{j; r_1,\ldots, r_n} = U_j^s $ for $n=0$, which agrees with the previous definition for $j=-1$.
Note that for a fixed $j \in \Z^*$, the collection of disjoint intervals $\{ U^s_{j; r_1, \ldots, r_n} : 1 \leq s \leq k-1,\ n \geq 0,\ r_1, \ldots, r_n \in \{ 1, \ldots, k \}   \}$ forms the complement of the Cantor set $\Lambda_j$ and
\[ 
(0,1) \setminus \Lambda(f) = \bigcup \limits_{j \in \Z} U_j \cup
\bigcup \limits_{j \in \Z^*} \bigcup_{s=1}^{k-1} \;
\bigcup \limits_{ n=0}^{\infty} \; \bigcup \limits_{r_1, \ldots, r_n = 1}^k U^s_{j; r_1, \ldots, r_n}  
\]
with the union being disjoint. We can carry the same construction for the system $\{g_-, g_+\}$, yielding a decomposition
\[ 
(0,1) \setminus \Lambda(g) = \bigcup \limits_{j \in \Z} U_j \cup
\bigcup \limits_{j \in \Z^*} \bigcup_{s=1}^{k-1} \;
\bigcup \limits_{ n=0}^{\infty} \; \bigcup \limits_{r_1, \ldots, r_n = 1}^k
V^s_{j; r_1, \ldots, r_n} 
\]
for analogously defined $V_j, V^s_{j; r_1, \ldots, r_n}$.
By Lemma \ref{lem:intervals}, for $j \in \Z$,
\begin{equation}
\begin{aligned}\label{eq:Uj}
f_-(U_j) &=
\begin{cases}
U_{j-1} & \text{for } j \leq 0 \\
U^j_{-1} & \text{for } 1 \leq j \le k-1\\
U_{j-k} & \text{for } j \geq k
\end{cases},&\quad 
f_+(U_j) &=
\begin{cases}
U_{j+k} & \text{for } j \leq -k \\
U^{-j}_{1} & \text{for } -k +1 \le j \leq -1\\
U_{j+1} & \text{for } j \geq 0
\end{cases},\\
g_-(V_j) &=
\begin{cases}
V_{j-1} & \text{for } j \leq 0 \\
V^j_{-1} & \text{for } 1 \leq j \le k-1\\
V_{j-k} & \text{for } j \geq k
\end{cases},&\quad 
g_+(U_j) &= 
\begin{cases}
V_{j+k} & \text{for } j \leq -k \\
V^{-j}_{1} & \text{for } -k+1 \le j \leq -1\\
V_{j+1} & \text{for } j \geq 0
\end{cases}
\end{aligned}
\end{equation}
and for $j \in \Z^*$, $s \in \{1, \ldots, k-1\}$, $r_1, \ldots, r_n \in \{1, \ldots, k\}$, $n \ge 0$, 
\begin{equation}\label{eq:Ujl}
\begin{aligned} 
f_-(U^s_{j; r_1, \ldots, r_n}) &=
\begin{cases}
U^s_{j-1; r_1, \ldots, r_n} & \text{ for } j \leq 0 \\
U^s_{-1; j, r_1, \ldots, r_n} & \text{ for } 1 \leq j \leq k-1\\
U^s_{j-k; r_1, \ldots, r_n} & \text{ for } j \geq k
\end{cases},\\
f_+(U^s_{j; r_1, \ldots, r_n}) &=
\begin{cases}
U^s_{j+k; r_1, \ldots, r_n} & \text{ for } j \leq -k \\
U^s_{1; -j, r_1, \ldots, r_n} & \text{ for } -k+1 \leq j \leq -1\\
U^s_{j+1; r_1, \ldots, r_n} & \text{ for } j \geq 0
\end{cases},\\
g_-(V^s_{j; r_1, \ldots, r_n}) &=
\begin{cases}
V^s_{j-1; r_1, \ldots, r_n} & \text{ for } j \leq 0 \\
V^s_{-1; j, r_1, \ldots, r_n} & \text{ for } 1 \leq j \leq k-1\\
V^s_{j-k; r_1, \ldots, r_n} & \text{ for } j \geq k
\end{cases},\\
g_+(V^s_{j; r_1, \ldots, r_n}) &=
\begin{cases}
V^s_{j+k; r_1, \ldots, r_n} & \text{ for } j \leq -k \\
V^s_{1; -j, r_1, \ldots, r_n} & \text{ for } -k+1 \leq j \leq -1\\
V^s_{j+1; r_1, \ldots, r_n} & \text{ for } j \geq 0
\end{cases}.
\end{aligned}
\end{equation}
We can now extend $h$ to an increasing homeomorphism of $[0,1]$ as follows: on $U_j,\ j \in \Z$, we define $h$ to be the unique affine increasing homeomorphism such that $h(U_j) = V_j$ and on $U^s_{j; r_1, \ldots, r_n}$, $j \in \Z^*, s \in \{1, \ldots, k-1\}$, $n \ge 0$, $r_1, \ldots, r_n \in \{1, \ldots, k\}$, we set $h$ to be the unique affine increasing homeomorphism such that $h(U^s_{j; r_1, \ldots, r_n}) = V^s_{j; r_1, \ldots, r_n}$. Finally, we set $h(0) = 0, h(1) = 1$. It is easy to see that $h$ is a homeomorphism of $[0,1]$. Using $(\ref{eq:Uj})$ and $(\ref{eq:Ujl})$ we see that
\[ 
f_{\pm}(U_j) = h^{-1} \circ g_{\pm} \circ h (U_j) \qquad \text{and} \qquad f_{\pm}(U^s_{j;  r_1, \ldots, r_n}) =  h^{-1} \circ g_{\pm} \circ h (U^s_{j; r_1, \ldots, r_n}).   
\]
Since $f_{\pm}$ and $h^{-1} \circ g_{\pm} \circ h$ are both affine and increasing on each of the above intervals, we have $f_{\pm} = h^{-1} \circ g_{\pm} \circ h$ on each of them.

\section{Proof of Theorem~\ref{thm:res-full}}\label{sec:res-full}

We consider a symmetric AM-system with probabilities $p_- = p_+ = 1/2$ and positive Lyapunov exponents, which exhibits $(5:2)$-resonance and satisfies $\rho  = \eta$. The latter condition is equivalent to
\begin{equation}\label{eq:rho=eta}
\rho^7 - 2\rho^6 +2\rho - 1 = 0
\end{equation}
and to $\rho x_-=1/2$. Note that this implies $f_-(x_-) = \rho^2 x_- < 1/2$, so the system is of disjoint type (see the beginning of the proof of Theorem~\ref{thm:(k:l)} in the case $l > 1$).

Define segments $J_j$, $j \in \Z^*$ as in the case $\rho < \eta$. We have 
\[
J_j = 
\begin{cases}
[\rho^{-j}/2, \rho^{-j+1}/2] &\text{for } j < 0\\
[\II(\rho^j/2), \II(\rho^{j-1}/2)] &\text{for } j > 0
\end{cases},
\]
so the segments $J_j$ have pairwise disjoint interiors, each two consecutive intervals (according to the order in $\Z^*$) have a common endpoint and $\bigcup_{j \in \Z^*}J_j = (0,1)$. Similarly, defining maps $\phi_r$ and intervals $I_{\jj}$, $\jj \in \JJ$ as in the case $\rho < \eta$ and proceeding as in the proofs of Lemmas~\ref{lem:ifs} and~\ref{lem:intervals-kl}, we check that for each $j \in \Z^*$, the intervals $I_{\jj}$, $\jj \in \JJ_j$ are contained in $J_j$, have disjoint interiors and satisfy $\sum_{\jj \in \JJ_j} |I_{\jj}| = |J_{\jj}|$. Analogously, we can define maps $\Phi_\rr$, $\rr \in \RR$, intervals $I_{\jj; \rr_1, \ldots, \rr_m}$ and sets $\Lambda_\jj, \Lambda$ in the same way as in the case $\rho < \eta$. The maps $\Phi_\rr$  form an iterated function system in $I_{-1}$, such that the intervals $\Phi_\rr(I_{-1})$ have disjoint interiors and $\sum_{\rr \in \RR} |\Phi_\rr(I_{-1})| = \sum_{\rr \in \RR} |I_{-1;\rr}| = |I_{-1}|$.
Hence, $\Lambda_{-1} = I_{-1}$ and the pressure \eqref{eq:pressure} satisfies $P(1) = 0$. The combinatorics of the intervals $I_{\jj; \rr_1, \ldots, \rr_m}$ is the same as in the case $\rho < \eta$, so Lemmas~\ref{lem:f} and~\ref{lem:f(I)-kl} and Propositions~\ref{prop:omega-kl} and~\ref{prop:X-kl} still hold. We have $\Lambda_\jj = I_\jj$ for $\jj \in \JJ$ and $\Lambda = (0,1)$.

By Theorem~\ref{thm:stationary}, there exists a unique stationary measure $\mu$, and Proposition~\ref{prop:X-kl} implies $\supp \mu = \Lambda \cup \{0, 1\} = [0,1]$. By Proposition~\ref{prop:nonatom}, the measure $\mu$ is non-atomic. Hence, the measure of the endpoints of the intervals $I_{\jj; \rr_1, \ldots, \rr_m}$ is zero. In particular, Proposition~\ref{prop:struct} holds in this case with the same proof. 

The above facts show that $\{\Phi_\rr\}_{\rr \in \RR}$ is a countable iterated function system of contracting similarities on $I_{-1}$ satisfying the Open Set Condition, with the attractor $\Lambda_{-1} = I_{-1}$. By Proposition~\ref{prop:struct}, the probability measure 
\[
\mu_{-1} = \frac{\mu|_{I_{-1}}}{\mu(I_{-1})}
\]
is the self-similar measure for this system with probabilities $\beta_\rr$, $\rr \in \RR$. 

To prove Theorem~\ref{thm:res-full}, we show $\dim_H \mu < 1$. Since by Proposition~\ref{prop:struct}, the measure $\mu$ is a countable linear combination of $\mu_{-1}$ and its similar copies $\mu|_{I_\jj}$, $\jj \in \JJ$, it is sufficient to show $\dim_H \mu_{-1} < 1$. Let 
\[
h(\mu_{-1}) = -\sum_{\rr \in \RR}\beta_\rr \log \beta_\rr
\]
be the entropy of $\mu_{-1}$. The proof splits into two cases depending whether $h(\mu_{-1})$ is finite or infinite. To shorten the proof, we do not determine which case actually takes place, but we consider both possibilities. 

Suppose first that $h(\mu_{-1})$ is infinite. Then we have $\dim_H \mu_{-1} \leq t_0 < 1$, where $t_0 = \inf\{t> 0: P(t) < \infty\}$ is the unique solution of the equation $\rho^{lt_0}-2\rho^{t_0}+1 = 0$ (see the proof of Proposition~\ref{prop:dimX-kl}). This fact follows from \cite[Proposition~3.1]{baker-jurga}, which is based on \cite[Theorem~4.1]{kifer-peres}. Actually, the mentioned results in \cite{baker-jurga,kifer-peres} are formulated for a more specific class of iterated function systems, but the proofs are valid in the general case of self-similar systems on the interval.

Suppose now that $h(\mu_{-1})$ is finite. Recall that the self-similar iterated function system $\{\Phi_\rr\}_{\rr\in\RR}$ on $I_{-1}$ is regular with the attractor $\Lambda_{-1} = I_{-1}$. In particular, 
the normalized Lebesgue measure $\mathcal L = \Leb|_{I_{-1}}/|I_{-1}|$ is the Gibbs and equilibrium state for the geometrical potential in dimension $1$ and also the $1$-conformal measure for this system on $I_{-1}$ (see \cite[Section 4.4]{MUbook}). Moreover, the Lyapunov exponent 
\[
\chi(\mathcal L) = \sum_{\rr \in \RR}\|\Phi_\rr'\| \log \|\Phi_\rr'\|
\]
of the measure $\mathcal L$ is finite, since (similarly as in \eqref{eq:pressure}) by the definition of the set $\RR$ in the considered case,
\[
\chi(\mathcal L) = \sum_{r = 2}^4 \sum_{n=1}^\infty \rho^{r+n} \log (\rho^{r+n}) + 
\sum_{r = 2}^5 \rho^r \log (\rho^r) > -\infty.
\]
In such a situation \cite[Theorem 4.4.7]{MUbook} (see also \cite[Theorem 4.6]{HMU}) asserts that either the self-similar measure $\mu_{-1}$ is equal to $\mathcal L$ or $\dim_H \mu_{-1} < \dim_H \Lambda_{-1} = 1$. Therefore, to end the proof of the theorem, it is sufficient to show $\mu_{-1} \neq \mathcal L$.

Suppose $\mu_{-1} =\mathcal L$. Then
\begin{equation}\label{eq:mu-lambda}
\frac{\mu(I_{-1;r})}{\mu(I_{-1})} = \frac{|I_{-1;r}|}{|I_{-1}|} = \rho^r
\end{equation}
for $r \in \{2,3,4,5\}$. Consider the characteristic polynomial $x^{k+l}-2x^l+1$ from Proposition~\ref{prop:struct}. In the considered case it has the form
\[
h(x) = x^7-2x^2+1 = (x-1)(x^3+x^2-1)(x^3+x+1).
\]
Computing the derivatives, we check that the polynomial $x^3+x^2-1$ has a unique real root $\alpha \in (0, 1)$, while $x^3+x+1$ has a unique real root $\beta \in (-1, 0)$. By Viete's formulas for these polynomials, the remaining non-real roots of $h$ have moduli greater than $1$. Therefore, Proposition~\ref{prop:struct} implies that for $j \in \Z^*$ and $r \in \{2,3,4,5\}$,
\begin{equation}\label{eq:A,B}
\mu(I_j) = A \alpha^{|j|} + B \beta^{|j|}, \qquad \mu(I_{-1;r}) = \frac{\mu(I_{-1}) \mu(I_r)}{2 \mu(I_2) - \mu(I_7)}
\end{equation}
for some $A, B \in \R$. Since $\mu(I_j) > 0$, we have $(A,B) \neq (0,0)$.

By \eqref{eq:mu-lambda} and \eqref{eq:A,B},
\[
q(A \alpha^r + B \beta^r) = \rho^r, \qquad r = 2, 3, 4, 5,
\]
where $q = 2 \mu(I_2) - \mu(I_7) > 0$. This implies $A \alpha^{r+1} + B \beta^{r+1} = \rho (A \alpha^r + B \beta^r)$ for $r = 2, 3, 4$, which gives
\[
\Big(\frac{\alpha}{\beta}\Big)^rA(\alpha-\rho) = B(\rho-\beta), \qquad r = 2, 3, 4.
\]
We have $(A,B) \neq (0,0)$. Moreover, $\rho \neq \beta$ because $\rho > 0$, $\beta < 0$. If $\rho = \alpha$, then by \eqref{eq:rho=eta} and the definition of $\alpha$,
\[
\rho^6-\rho^5-\rho^4-\rho = \frac{\rho^7-2\rho^6+2\rho-1}{\rho - 1} + \rho^3+\rho^2-1 = 0 
\]
which is impossible since $\rho \in (0,1)$. Hence, $A, B, \alpha - \rho, \rho - \beta \neq 0$ and we can write
\[
\Big(\frac{\alpha}{\beta}\Big)^r = \frac{B(\rho-\beta)}{A(\alpha-\rho)}, \qquad r = 2, 3, 4,
\]
which implies $\alpha = \beta$ and makes a contradiction. This ends the proof of Theorem~\ref{thm:res-full}.

\bibliography{singular}

\end{document}